\newcommand{\bf}{\mathbf}
\newcommand{\rm}{\mathrm}
\documentclass[reqno]{amsart}
\usepackage{amsmath,amssymb,amsthm,amsfonts}
\usepackage{hyperref}
\usepackage{appendix}
\usepackage{graphicx}
\usepackage{setspace}

\usepackage{color}

\newtheorem{lemma}{Lemma}[section]
\newtheorem{theorem}{Theorem}[section]
\newtheorem{definition}{Definition}[section]
\newtheorem{proposition}{Proposition}[section]
\newtheorem{remark}{Remark}[section]
\newtheorem{corollary}{Corollary}[section]
\numberwithin{equation}{section}
\allowdisplaybreaks    %ÔÊÐíÔÚ¶àÐÐÊýѧ¹«Ê½Öм任ҳ
\newcommand{\p}{\partial}
\newcommand{\eps}{\varepsilon}
\newcommand{\y}{\langle y \rangle}
\newcommand{\intx}{\int_0^x}
\newcommand{\inty}{\int_0^\infty}
\begin{document}
\title[The boundary layer theory for steady MHD: nonshear flows]{Validity of Prandtl layer theory for steady magnetohydrodynamics over a moving plate with nonshear outer ideal MHD flows}
\thanks{$^*$Corresponding author}
\author[Shijin Ding]{Shijin Ding}
\address[S. Ding]{South China Research Center for Applied Mathematics and Interdisciplinary Studies, South China Normal University,
Guangzhou, 510631, China}\address{School of Mathematical Sciences, South China Normal University,
Guangzhou, 510631, China}
\email{dingsj@scnu.edu.cn}
\author[Zhijun Ji]{Zhijun Ji$^*$}
\address[Z. Ji]{South China Research Center for Applied Mathematics and Interdisciplinary Studies, South China Normal University,
Guangzhou, 510631, China}\address{School of Mathematical Sciences, South China Normal University,
Guangzhou, 510631, China}
\email{zhijunji@m.scnu.edu.cn}
\author[Zhilin Lin]{Zhilin Lin}
\address[Z. Lin]{School of Mathematical Sciences, South China Normal University,
Guangzhou, 510631, China}
\email{zllin@m.scnu.edu.cn}

\date{\today}

\begin{abstract}
In this paper, we validate the boundary layer theory for 2D steady viscous incompressible magnetohydrodynamics (MHD) equations in a domain $\{(X, Y)\in[0, L]\times\mathbb{R}_+\}$ under the assumption of a moving boundary at $\{Y=0\}$. The validity of the boundary layer expansion and the convergence rates are established in Sobolev sense. We extend the results for the case with the shear outer ideal MHD flows \cite{DLX} to the case of the nonshear flows.
\end{abstract}
\maketitle
\vspace{-5mm}

%\newpage
\section{Introduction}
\subsection{Formulation of the problem}
This paper is concerned with the following steady viscous incompressible magnetohydrodynamics system in $\Omega:=[0,L]\times\mathbb{R}_+$ with moving boundary conditions on velocity field and the perfect conducting boundary conditions on magnetic field at $\{Y=0\}$:
\begin{equation}\label{1.1}
\begin{cases}
  (U\p_X+V\p_Y)U-(H\p_X+G\p_Y)H+\p_X P=\nu\eps(\p_{XX}+\p_{YY})U, \\
  (U\p_X+V\p_Y)V-(H\p_X+G\p_Y)G+\p_Y P=\nu\eps(\p_{XX}+\p_{YY})V, \\
  (U\p_X+V\p_Y)H-(H\p_X+G\p_Y)U=\kappa\eps(\p_{XX}+\p_{YY})H,\\
  (U\p_X+V\p_Y)G-(H\p_X+G\p_Y)V=\kappa\eps(\p_{XX}+\p_{YY})G,\\
  \p_X U+\p_Y V=0,\quad\p_X H+\p_Y G=0,\\
  (U,V,\p_Y H,G)|_{Y=0}=(u_b,0,0,0),
\end{cases}
\end{equation}
where $(U,V)$ and $(H,G)$ are velocity and magnetic field respectively, and the given constant $u_b>0$ stands for the moving speed of the plate. Here we assume that the viscosity and resistivity coefficients have the same order of a small parameter $\eps>0$.

It is interesting to study the asymptotic behavior of solutions to (\ref{1.1}) as $\eps\rightarrow 0$. In the present paper, we suppose that the outer ideal MHD flows are prescribed by
\begin{equation*}
  (u_e^0,v_e^0,h_e^0,g_e^0,p_e^0)(X,Y),
\end{equation*}
which satisfy the following ideal MHD system with non-penetration boundary conditions for both velocity and magnetic field at $Y=0$ and $Y\rightarrow+\infty$:
\begin{align}\label{u_e^0}
\begin{cases}
  (u_e^0\p_X+v_e^0\p_Y)u_e^0-(h_e^0\p_X+g_e^0\p_Y)h_e^0+\p_X p_e^0=0,\\
  (u_e^0\p_X+v_e^0\p_Y)v_e^0-(h_e^0\p_X+g_e^0\p_Y)g_e^0+\p_Y p_e^0=0,\\
  (u_e^0\p_X+v_e^0\p_Y)h_e^0-(h_e^0\p_X+g_e^0\p_Y)u_e^0=0,\\
  (u_e^0\p_X+v_e^0\p_Y)g_e^0-(h_e^0\p_X+g_e^0\p_Y)v_e^0=0,\\
  \p_X u_e^0+\p_Y v_e^0=0,\quad \p_X h_e^0+\p_Y g_e^0=0,\\
  (v_e^0,g_e^0)|_{Y=0}=(v_e^0,g_e^0)|_{Y\rightarrow +\infty}=0.
\end{cases}
\end{align}

Obviously, there is a mismatch in the tangential velocity and magnetic field between the viscous MHD flows $(U,\p_YH)(X,0)=(u_b,0)$ and the ideal MHD flows $(u_e^0,h_e^0)(X,0)=(\overline u_e^0,\overline h_e^0)$ on the boundary $\{Y=0\}$, which contradicts to the convergence in the vanishing viscosity and resistivity limit process. Following the idea of Prandtl \cite{Prandtl}, to correct the mismatch, the boundary layer corrector functions should be introduced in a thin layer with width of $\sqrt\eps$ near $\{Y=0\}$. For the boundary layer functions, we will work with the scaling boundary layer variable $(x,y)$ as follows:
\begin{equation*}
  x=X,\quad y=\frac{Y}{\sqrt\eps}.
\end{equation*}
And we introduce the scaled unknowns
\begin{equation}\label{scale}
  (U^{\eps},V^{\eps},H^{\eps},G^{\eps},P^{\eps})(x,y)= (U,\frac{V}{\sqrt\eps}, H,\frac{G}{\sqrt\eps},P)(X,Y),
\end{equation}
which satisfy the divergence-free conditions as well. With the boundary layer scaled variables, the problem (\ref{1.1}) can be rewritten as
\begin{align}\label{scalesystem}
\begin{cases}
  (U^{\eps}\p_x+V^{\eps}\p_y)U^{\eps}-(H^{\eps}\p_x+G^{\eps}\p_y)H^{\eps}+\p_x P^{\eps}=\nu\eps\p_{xx}U^{\eps}+\nu\p_{yy}U^{\eps}, \\
  (U^{\eps}\p_x+V^{\eps}\p_y)V^{\eps}-(H^{\eps}\p_x+G^{\eps}\p_y)G^{\eps}+\frac{\p_y P^{\eps}}{\eps}=\nu\eps\p_{xx}V^{\eps}+\nu\p_{yy}V^{\eps}, \\
  (U^{\eps}\p_x+V^{\eps}\p_y)H^{\eps}-(H^{\eps}\p_x+G^{\eps}\p_y)U^{\eps}
  =\kappa\eps\p_{xx}H^{\eps}+\kappa\p_{yy}H^{\eps},\\
  (U^{\eps}\p_x+V^{\eps}\p_y)G^{\eps}-(H^{\eps}\p_x+G^{\eps}\p_y)V^{\eps}
  =\kappa\eps\p_{xx}G^{\eps}+\kappa\p_{yy}G^{\eps},\\
  \p_x U^\eps+\p_y V^\eps=0,\quad \p_x H^\eps+\p_y G^\eps=0,\\
  (U^\eps,V^\eps,\p_y H^\eps,G^\eps)|_{y=0}=(u_b,0,0,0).
\end{cases}
\end{align}

The key point in this paper is to construct the approximate solutions to (\ref{scalesystem}) and derive the convergence rates in Sobolev space. To this end, we introduce the following asymptotic expansions:
\begin{equation}\label{expansion}
  (U^{\eps},V^{\eps},H^{\eps},G^{\eps},P^{\eps})
  =(u_{app},v_{app},h_{app},g_{app},p_{app})
   +\eps^{\frac{1}{2}+\gamma}(u^{\eps},v^{\eps},h^{\eps},g^{\eps},p^{\eps}),
\end{equation}
for some constant $\gamma>0$, where $(u_{app},v_{app},h_{app},g_{app},p_{app})$ and $(u^{\eps},v^{\eps},h^{\eps},g^{\eps},p^{\eps})$ are the approximate solutions and the error solutions to the exact solutions of problem (\ref{scalesystem}), respectively. The approximate solutions are defined by
\begin{equation}\label{app}
\begin{cases}
  u_{app}=u_e^0(x,\sqrt\eps y)+u_p^0(x,y)+\sqrt{\eps}[u_e^1(x,\sqrt\eps y)+u_p^1(x,y)],\\
  v_{app}=\frac{v_e^0}{\sqrt{\eps}}(x,\sqrt\eps y)+v_p^0(x,y)+v_e^1(x,\sqrt\eps y)+\sqrt{\eps}v_p^1(x,y),\\
  h_{app}=h_e^0(x,\sqrt\eps y)+h_p^0(x,y)+\sqrt{\eps}[h_e^1(x,\sqrt\eps y)+h_p^1(x,y)],\\
  g_{app}=\frac{g_e^0}{\sqrt{\eps}}(x,\sqrt\eps y)+g_p^0(x,y)+g_e^1(x,\sqrt\eps y)+\sqrt{\eps}g_p^1(x,y),\\
  p_{app}=p_e^0(x,\sqrt\eps y)+p_p^0(x,y)+\sqrt{\eps}[p_e^1(x,\sqrt\eps y)+p_p^1(x,y)]+\eps p_p^2(x,y).
\end{cases}
\end{equation}
Here we note that the ideal MHD profiles $(u_e^i,v_e^i,h_e^i,g_e^i,p_e^i)$ are functions of Eulerian variables $(x,\sqrt\eps y)$, while the boundary layer profiles $(u_p^i,v_p^i,h_p^i,g_p^i,p_p^i)$ are of boundary layer scaled variables $(x,y)$. Then we can calculate the error equations for the remainder terms $(u^{\eps},v^{\eps},h^{\eps},g^{\eps},p^{\eps})$(see \eqref{u_epssystem} in Section \ref{sec3}) by subtracting the following approximations:
\begin{equation}\label{R_app_expansion}
\begin{cases}
  R_{app}^1=(u_{app}\p_x+v_{app}\p_y)u_{app}-(h_{app}\p_x+g_{app}\p_y)h_{app}+\p_x p_{app}-\nu\Delta_{\eps}u_{app}, \\
  R_{app}^2=(u_{app}\p_x+v_{app}\p_y)v_{app}-(h_{app}\p_x+g_{app}\p_y)g_{app}+\frac{\p_y p_{app}}{\eps}-\nu\Delta_{\eps}v_{app}, \\
  R_{app}^3=(u_{app}\p_x+v_{app}\p_y)h_{app}-(h_{app}\p_x+g_{app}\p_y)u_{app}-\kappa\Delta_{\eps}h_{app},\\
  R_{app}^4=(u_{app}\p_x+v_{app}\p_y)g_{app}-(h_{app}\p_x+g_{app}\p_y)v_{app}-\kappa\Delta_{\eps}g_{app},
\end{cases}
\end{equation}
where $\Delta_{\eps}=\eps\p_x^2+\p_y^2$. In addition, we denote $\Delta=\p_X^2+\p_Y^2$ for later use.

To construct the approximate solutions of problem \eqref{scalesystem}, we would first discuss the boundary conditions of each profile.

(i) Concerning the leading-order profiles, the boundary conditions for the prescribed outer ideal MHD flows $(u_e^0,v_e^0,h_e^0,g_e^0)(x,Y)$ are given by \eqref{u_e^0}$_6$ and
\begin{equation*}
  u_e^0(x,0):=\overline u_e^0(x),\quad h_e^0(x,0):=\overline h_e^0(x).
\end{equation*}
Since the boundary layer profile $(u_p^0,v_p^0,h_p^0,g_p^0)(x,y)$ is introduced to correct the mismatch between the viscous MHD flows and the ideal MHD flows on the boundary $\{Y=0\}$, we impose the following boundary conditions by matching the corresponding order of $\eps$ in the approximate solutions in \eqref{app}:
\begin{equation*}
  \overline u_e^0+u_p^0(x,0)=u_b,\quad\p_y h_p^0(x,0)=0,\quad (v_p^0,g_p^0)(x,0)=-(v_e^1,g_e^1)(x,0).
\end{equation*}
And the tangential components vanish at the infinity
\begin{equation*}
\lim_{y\rightarrow\infty}u_p^0(x,y)=\lim_{y\rightarrow\infty}h_p^0(x,y)=0.
\end{equation*}
In addition, when taking $x$-variable as ``time"-variable, $(u_p^0,v_p^0,h_p^0,g_p^0)$ enjoy a parabolic type system, so the boundary value on $\{x=0\}$ taken as the ``initial data" is required as well, we give
\begin{equation*}
  (u_p^0,h_p^0)(0,y)=(u_{p,0}^0,h_{p,0}^0)(y).
\end{equation*}

(ii) With regard to the first-order ideal MHD flows $(u_e^1,v_e^1,h_e^1,g_e^1)$, we turn to study the equations for stream functions $(\Phi,\Psi)$, which are defined by $\nabla^\bot\Phi=(u_e^1,v_e^1)$ and $\nabla^\bot\Psi=(h_e^1,g_e^1)$. Thanks to the analysis in Subsection \ref{sec2.4}, the two stream functions enjoy an equality $\eqref{streamrelation}$, then we only need to discuss the elliptic equation \eqref{phieq} for the stream function $\Phi$. Therefore, it is necessary for us to impose the boundary conditions at $\{x=0,L\}$ that
\begin{equation*}
  \Phi(0,Y)=\Phi_0(Y),\quad \Phi(L,Y)=\Phi_L(Y),
\end{equation*}
and the boundary condition at $\{Y=0\}$ is given by
\begin{align*}
    \Phi|_{Y=0}=1+\int_0^x v_p^0(s,0)\rm{ds}.
\end{align*}
 Note that the condition at $\{Y=0\}$ is equivalent to $\Phi_x(x,0)=-v_e^1(x,0)=v_p^0(x,0)$, where the general constant has been selected to 1, without loss of generality. The data $\Phi_0(Y)$ and $\Phi_L(Y)$ are taken to be sufficiently smooth and decay exponentially fast at infinity, and we assume the compatibility conditions $\Phi_0(0)=1$ and $\Phi_L(0)=1+\int_0^L v_p^0(s,0)\rm{ds}$. Moreover, the boundary values are supposed to satisfy the compatibility conditions stated in Definition \ref{wellprepared}.

(iii) Since the first-order boundary profile $(u_p^1,v_p^1,h_p^1,g_p^1)$ satisfies a linear parabolic type system \eqref{u_p}, we impose the following boundary conditions
\begin{align*}
  (u_p^1,\p_y h_p^1)(x,0)=-(\overline u_e^1,\overline{\p_Y h_e^0})(x),\quad (v_p^1,g_p^1)(x,0)=(0,0),
\end{align*}
and
\begin{align*}
    \lim_{y\rightarrow\infty}u_p^1(x,y)=\lim_{y\rightarrow\infty}h_p^1(x,y)=0,\quad (u_p^1,h_p^1)(0,y)=(u_{p,0}^1,h_{p,0}^1)(y).
\end{align*}
To keep the boundary condition $\p_yh_{app}(x,y)|_{y=0}=0$ for the approximate solutions, we will introduce a boundary corrector to cancel the boundary value of $\p_Yh_e^1(x,Y)$ on $\{Y=0\}$ in Subsection \ref{subsec2.5}. Moreover, it is noted that a cut-off function will be also introduced to localize $(v_p^1,g_p^1)$ in Subsection \ref{subsec2.5}, since the vertical components $(v_p^1,g_p^1)$ are constructed by $(u_p^1,h_p^1)$ through the divergence free conditions, which possibly leads to $\lim_{y\rightarrow \infty}(v_p^1,g_p^1)\neq 0$.

(iv) For the final profile $(u^\eps,v^\eps,h^\eps,g^\eps)$, the boundary conditions are given by
\begin{align*}
  \begin{cases}
  (u^\eps ,v^\eps ,\p_y h^\eps ,g^\eps)|_{y=0}=(0,0,0,0),\quad
  (u^\eps,v^\eps,h^\eps,g^\eps)|_{x=0}=(0,0,0,0),\\
  p^\eps-2\nu\eps\p_x u^\eps|_{x=L}=0,\quad
  \p_yu^\eps+\nu\eps\p_x v^{\epsilon}|_{x=L}=0,\quad
  (h^\eps,\p_x g^\eps)|_{x=L}=(0,0).
\end{cases}
\end{align*}

Furthermore, we remark that the above boundary functions are assumed to be smooth and exponentially decay to zero at infinity.

\subsection{Review about some known works}

Let us briefly review some well known results for the Prandtl layer theory and some related works. It should be noted that the known results are concluded in some special framework of analytic functions \cite{Maekawa,Sammartino1,Sammartino2,Zhang} and then relaxed to Gevrey class \cite{Gerard1,Gerard2,LWX}. However, in the time dependent case, the validity for the Prandtl layer theory in Sobolev space still remains open. It is very natural and interesting to consider the problem in the steady case. This type of results was initiated by Guo and Nguyen \cite{GN17}, in which the local in $x$ validity for the Prandtl layer theory was established for the case of the outside shear Euler flows $(u_e^0(Y),0)$  in the domain $[0,L]\times\mathbb{R}_+$ for small constant $L>0$. Similar results was extended to the rotating disk \cite{Iyerrotating}, the bounded domain \cite{Iyerbounded,LD} and the case with outside nonshear Euler flows \cite{Iyernonshear}. The global in $x$ expansion for the case that the Euler flow is $(1,0)$ was verified by Iyer in a series of works \cite{Iyerglobal1,Iyerglobal2,Iyerglobal3}. However, all these results are obtained with the moving boundary condition, and then this moving boundary assumption was removed by Guo and Iyer \cite{GuoIyer} through taking the self-similar Blasius profile as the zero-order boundary layer corrector functions. Gao and Zhang \cite{GaoZhang} removed the moving boundary condition and the small condition of $L$ simultaneously for the case of shear Euler flows by estimating stream functions for the remainders. Very recently, S. Iyer and N. Masmoudi \cite{IyerMasmoudi} verified Prandtl boundary layer theory globally in the $x$-variable for a large class of boundary layers, including the entire one parameter family of the classical Blasius profiles, with sharp decay rates.

Considering the Prandtl layer theory in magnetohydrodynamics is very interesting and challenging. As mentioned in the famous literature by Oleinik et al. \cite{Oleinik} (page 500-503),

{\it ``15. For the equations of the magnetohydrodynamic boundary layer, all problems of the above type are still open,"}

\noindent the boundary layer theory in MHD has been an important topic for long time.
Due to the coupling magnetic field with velocity field, the analysis for the boundary layer in MHD is more difficult than that in Navier-Stokes flows. To study the problem, it is very natural to find the stabilizing effect of the magnetic filed. With this idea, there are some works on the boundary layer theory of the MHD equations. On one hand, for unsteady flows, Liu, Xie and Yang established the well-posedness theory of boundary layer equations \cite{LXYwell} and the convergence theory for boundary layer expansion \cite{LXYjus} with the assumption of nondegenerate tangential magnetic field instead of monotonicity condition on the velocity filed. On the other hand, for steady MHD flows, the MHD equations without magnetic diffusion was studied by Wang and Ma \cite{WangMa}, in which global existence and non-existence theory of the boundary layer system were obtained for different ratios of the magnetic field and the velocity field. Very recently, Ding, Lin and Xie \cite{DLX} verified the Prandtl boundary layer ansatz of the steady MHD flows with a moving boundary on the domain $[0,L]\times\mathbb{R}_+$ for the case with outer shear ideal MHD flows $(U_0(Y),0,H_0(Y),0)$.

\subsection{Main result}
To state our main results, we define the following ${\mathcal{X}}-Norm$ which will be used to control our remaider solutions:
\begin{align}\label{X-norm}
  \|u^\eps,v^\eps,h^\eps,g^\eps\|_{\mathcal{X}}
  &:=\| \{u^\eps_y,h^\eps_y,\sqrt\eps(u^\eps_x,h^\eps_x)\}\cdot y \|_{ L^2}+\|v^\eps_y,g^\eps_y,\sqrt\eps(v^\eps_x,g^\eps_x)\|_{ L^2}\nonumber\\
  &\quad +\| \{u^\eps_{yy},h^\eps_{yy},\sqrt\eps (u^\eps_{xy},h^\eps_{xy}),\eps(u^\eps_{xx},h^\eps_{xx})\}\cdot y \|_{ L^2}\nonumber\\
  &\quad +\|u^\eps,v^\eps,h^\eps,g^\eps\|_B+\eps^{\frac{\gamma}{2}}\| u^\eps,h^\eps,\sqrt\eps(v^\eps,g^\eps) \|_{ L^\infty},
\end{align}
where the boundary term is defined by
\begin{equation}\label{B-norm}
  \|u^\eps,v^\eps,h^\eps,g^\eps\|_B:=\| \{u^\eps_y,\sqrt\eps(u^\eps_x,h^\eps_x)\}\cdot y \|_{ L^2(x=L)}+\| \sqrt\eps(u^\eps_x,h^\eps_x) \|_{ L^2(x=L)}.
\end{equation}

Now we can state our main result:
\begin{theorem}\label{th1.1}
Let $u_b > 0$ be a constant tangential velocity of the viscous MHD flows on the boundary $\{Y = 0\}$, and the given positive non-shear ideal MHD flows $[u_e^0,v_e^0,h_e^0,g_e^0,p_e^0](X,Y)$ satisfy the following hypotheses£º
\begin{align}
\label{condition_for_idealMHD1}
  &0< c_0\leq  h_e^0\ll u_e^0 \leq C_0<\infty,\\
\label{condition_for_idealMHD3}
  &\bigg\| \frac{v_e^0}{Y} \bigg\|_{ L^\infty}\ll 1,\\
\label{condition_for_idealMHD4}
  &\|Y^k \nabla^m (v_e^0,g_e^0)\|_{ L^\infty}<\infty\qquad \mathrm{for\  sufficiently \ large \ } k,m\geq0,\\
\label{condition_for_idealMHD5}
  &\|Y^k \nabla^m (u_e^0,h_e^0)\|_{ L^\infty}<\infty\qquad \mathrm{for\  sufficiently \ large \ }  k\geq0, m\geq1,\\
\label{condition_for_idealMHD6}
  &\|\langle Y\rangle\p_Y (u_e^0,h_e^0)\|_{ L^\infty}<\delta_0\qquad \mathrm{for\  suitable \ small \ }  \delta_0>0.
\end{align}
In addition, let $m\geq 5$ be an integer, the outer ideal MHD flows are assumed to enjoy
\begin{equation}\label{M_0}
  M_0:=\sum_{i=0}^{m+2}\|(\overline u_e^0,\overline h_e^0,\overline p_e^0)(x)\|_{H^{m+2-i}(0,L)}<+\infty.
\end{equation}
Moreover, for some positive constants $\vartheta_0,\eta_0$ and small $\sigma_0$, suppose that
\begin{align}\label{condition_for_u_p^0}
\begin{cases}
  &\overline u_e^0+u_p^0(0,y)>\overline h_e^0+h_p^0(0,y)\geq \vartheta_0,\\
  &|u_e^0(0,Y)+u_p^0(0,y)|\gg |h_e^0(0,Y)+h_p^0(0,y)|\geq \eta_0,\\
  &|\y^{l+1}\p_y(u_p^0,h_p^0)(0,y)|\leq \frac{1}{2}\sigma_0,\\
  &|\y^{l+1}\p_y^2(u_p^0,h_p^0)(0,y)|\leq \frac{1}{2}\vartheta_0^{-1}.
\end{cases}
\end{align}
Then, there exist the remainder solutions $[u^\eps,v^\eps,h^\eps,g^\eps]$ in the space $\mathcal{X}$ satisfying
\begin{equation}\label{boundednessforXnorm}
  \|u^\eps,v^\eps,h^\eps,g^\eps\|_{ \mathcal{X}}\lesssim 1,
\end{equation}
in $[0,L]\times[0,+\infty)$, where the positive number $L$ is sufficiently small relative to universal constant.
\end{theorem}

\begin{corollary}\label{convergencelemma}
Under the hypothesis stated in Theorem \ref{th1.1} with the profile $(u_p^0,h_p^0)$ constructed in Section \ref{sec2}, it holds that
\begin{equation}\label{convergence}
  \|(U-u_e^0-u_p^0,H-h_e^0-h_p^0)\|_{ L^\infty}+\|(V-v_e^0,G-g_e^0)\|_{ L^\infty}\lesssim\sqrt\eps.
\end{equation}
\end{corollary}

\begin{remark}\label{remark1}
The ideal MHD flows $[u_e^0,v_e^0,h_e^0,g_e^0,p_e^0]$ which satisfy the assumptions of \eqref{condition_for_idealMHD1}-\eqref{M_0} do exist. See Appendix \ref{ap1} for details.
\end{remark}

\begin{remark}\label{remark2}
In our arguments, the conditions in \eqref{condition_for_u_p^0} play important role in the analysis for constructing the boundary layer profiles and the remainder profiles.  See Section \ref{sec2.1}, Section \ref{sec2.4}, Section \ref{sec3} and Appendix \ref{ap2} for more details. The readers can refer to the paper \cite{DLX} in pages 5-6 as well. It will be discussed whether the conditions are essential or not in our forthcoming work.

\end{remark}

\subsection{Main ideas and the sketch of the proof}
In this paper, we are going to justify the boundary layer expansion for the steady MHD equations with regard to outer non-shear ideal flows, which is different from the shear case in \cite{DLX}. Therefore, it is necessary to compare the key points in our analysis with \cite{DLX}. In particular, the main ideas and some key observations are explained as follows:

(i) The main difficulty of constructing the leading-order boundary correctors $(u_p^0,v_p^0,h_p^0,g_p^0)$ is the loss of $x$-derivatives, which is the same as the case of non-stationary boundary layer equations. To overcome this obstacle, we apply a modified energy method inspired by the work from Liu, Yang and Xie \cite{LXYwell}, see also \cite{DLX} for the steady version. Under the essential assumption that the tangential magnetic field has a lower positive bound, the cancelation is applied to avoid the loss of regularity by using the stream function of the magnetic fields. See Subsection \ref{sec2.1} and Appendix \ref{ap2}  for more details.

(ii) As we will see in Subsection \ref{sec2.3}, first-order ideal MHD profile $(u_e^1,v_e^1,h_e^1,g_e^1)$ enjoy system \eqref{u_e^1system}. To establish the estimates, the key point is to use the positivity of the second-order operator $-\p_{yy}+\frac{u_{syy}}{u_s}$ (see \cite{GN17} in pages 8-9 for details). The well-posedness and the estimates for system can not be deduced directly by the standard theory because of the coupling effects of velocity and magnetic fields. However, due to the non shear structure of the ideal MHD flows, one can not reduce this coupling system to a simple decoupling system by using the method introduced in \cite{DLX}. Therefore, it is very necessary to find a new relationship to decouple the unknowns. To this end, it is convenient to introduce the stream functions $\Phi,\Psi$ for the velocity and magnetic fields, respectively.

The first key observation is that the equations for the magnetic fields in the  first-order ideal MHD system can be rewritten as
$$\nabla_{x,Y}(v_e^0h_e^1+h_e^0v_e^1-u_e^0g_e^1-u_e^1g_e^0)=0,$$
which implies that
$$v_e^0h_e^1+h_e^0v_e^1-u_e^0g_e^1-u_e^1g_e^0\equiv \mathrm{constant}=:b.$$
Let $Y \to \infty$, we have
$$b=0,$$
where we have used the behavior of $(v^0_e,g^0_e,v^1_e,g^1_e)\to (0,0,0,0)$ at $Y\to \infty$.
Or equivalently, in the formulation of the stream functions:
\begin{equation}\label{observe1}
  (u_e^0\p_x+v_e^0\p_Y)\Psi=(h_e^0\p_x+g_e^0\p_Y)\Phi.
\end{equation}
The other key observation is deduced from the third equation in system \eqref{u_e^0} for the ideal MHD flows $(u_e^0,v_e^0,h_e^0,g_e^0)$:
$$ v_e^0h_e^0 -g_e^0u_e^0=0,$$
which implies that
\begin{equation}\label{observe2}
  h_e^0=k(x,Y)u_e^0,\quad g_e^0=k(x,Y)v_e^0,
\end{equation}
for some known $k(x,Y)$. It is noted that $0<k<1$ uniform in $(x,Y)$ by using the assumption \eqref{condition_for_idealMHD1} stated in our main theorem.

Combing the above two observations, it gives
\begin{equation*}
  (u_e^0\p_x+v_e^0\p_Y)\Psi=(u_e^0\p_x+v_e^0\p_Y)k\Phi,
\end{equation*}
in which we have used the following equality by virtue of the divergence-free conditions for velocity field and magnetic field:
\begin{equation}\nonumber
  (u_e^0\p_x+v_e^0\p_Y)k(x,Y)=\p_xh_e^0+\p_Yg_e^0-k(x,Y)(\p_x u_e^0+\p_Y v_e^0)=0.
\end{equation}
Therefore, we obtain a linear first-order partial differential equation read as
\begin{equation}\nonumber
  (\p_x+\frac{v_e^0}{u_e^0}\p_Y)f=0,
\end{equation}
in which we denote $f:=\Psi-k\Phi$. Define the characteristic curve of the above problem as the following ordinary differential equation
\begin{equation}\nonumber
  \frac{dY}{dx}=\frac{v_e^0}{u_e^0}(x,Y)
\end{equation}
with data $f_0:=f(0,Y)=\Psi_0-k_0\Phi_0$. According to characteristic method and local well-posedness theory of ODE, we get
\begin{equation}\label{observe}
  \Psi=k\Phi+F(u_e^0,v_e^0,f_0),
\end{equation}
in which $F$ is a function determined by the data $u_e^0,v_e^0,f_0$. And hence, the last equality \eqref{observe} gives the relationship between two stream functions.

Using the above equality \eqref{observe} together with the first observation \eqref{observe1}, we can deduce the following elliptic equation \eqref{phieq} for $\Phi$,
\begin{equation}\label{equaPhi}
\begin{split}
  &-\Delta\Phi=\mathcal{F}(\Phi),
\end{split}
\end{equation}
with the source term
\begin{align*}
\mathcal{F}(\Phi)=&-\frac{k\Delta k}{1-k^2}\cdot\Phi-\frac{2k\nabla k}{1-k^2}\cdot\nabla\Phi-\frac{k}{1-k^2}\Delta F\\
 &+\frac{1}{1-k^2}\intx H(\Phi(s,Y))ds+\frac{1}{1-k^2}G_0(\Phi_0(Y)),
\end{align*}
where the definition of $H(\Phi)$ and the data $G_0(\Phi_0(Y))$ will be given in Subsection \ref{sec2.3}.

Let us give the sketch of contraction mapping principle to determine $\Phi$ via \eqref{equaPhi} with suitable boundary conditions. Indeed, for any $\tilde{\Phi}$, by the standard theory of elliptic system, there exists a unique solution to the following equation
$$-\Delta\Phi=\mathcal{F}(\tilde{\Phi}),$$
which produces a solution mapping $\mathbf{T}(\tilde{\Phi})=\Phi$. Recall the following facts about the source term: $k$ is small in $L^\infty$ sense, the integral term is also small for uniform suitably small $L>0$, then one can verify the contraction of the mapping $\mathbf{T}:\tilde\Phi\mapsto \Phi$ for $\tilde\Phi,\Phi\in X:=\{\Phi|\|Y^n\Phi\|_{H^m}<C(n,m)\}$, provided that $k,L$ are suitably small, which gives the fixed point $\mathbf{T}(\Phi)=\Phi$, and thus produces the solution to the original problem \eqref{equaPhi}. With this, the magnetic fields can be determined by the formula \eqref{observe2}. Refer to Subsection \ref{sec2.3} for the details.

(iii) The first-order boundary layer profile $(u_p^1,v_p^1,h_p^1,g_p^1)$ satisfy a linear parabolic system with nonlocal terms $v_p^1\p_y u_p^0-g_p^1\p_yh_p^0$ and $v_p^1\p_y h_p^0-g_p^1\p_y u_p^0$. Since the second equation of the system of $(u_p^1,v_p^1,h_p^1,g_p^1)$ can be rewritten to a total differential form, it is possible for us to deduce a equation for the stream function of the magnetic field, which inspires us to introduce the new functions to cancel the nonlocal terms. The proof of this part is similar to that of the zero-order boundary layer $(u_p^0,v_p^0,h_p^0,g_p^0)$ with a modified energy method. See Subsection \ref{sec2.4} for details.

(iv) To figure out the nonlinear problem of the remainder profile $(u^\eps,v^\eps,h^\eps,g^\eps)$, the key point lies in establishing the estimates for the linearized problem. Compared with the case of the shear flows in \cite{DLX}, there is a leading order effect of the non-shear flows
resulted from the presence of nonzero $v_e^0, g_e^0$ of scaling $\mathcal{O}(\frac{1}{\sqrt\eps})$. It arouses us to focus on the terms $v_s\p_y(u^\eps,v^\eps,h^\eps,g^\eps)$ and $g_s\p_y(u^\eps,v^\eps,h^\eps,g^\eps)$ from the linear elements defined by \eqref{S_u} in the positivity estimate process (see Lemma \ref{positivitylemma}). Precisely, take $g_s\p_y h^\eps$ as an example, we rewrite the term as follows:
$$g_s\p_y h^\eps=\frac{g_e^0}{\sqrt\eps y}\cdot y\p_y h^\eps+(g_s-\frac{g_e^0}{\sqrt\eps})\cdot \p_y h^\eps,$$
the first term in the right-hand side will produce a $y$-weighted term $\|h^\eps_y\cdot y\|_{L^2}$. This new term can be made small by $\left\| \frac{g_e^0}{Y} \right\|_{ L^\infty}$ through using the following estimate
$$\bigg\|\frac{g_e^0}{Y}\bigg\|_{L^\infty}\lesssim \bigg\|\frac{v_e^0}{Y}\bigg\|_{L^\infty}\left\|\frac{h^0_e}{u^0_e}\right\|_{L^\infty} \ll 1,$$
where the equation
$$v_e^0h_e^0 -g_e^0u_e^0=0$$
derived from \eqref{u_e^0} is used.

To handle the leading order effect, inspired by the paper \cite{Iyernonshear}, we introduce a new $y$-weighted estimate to control the terms $\|u^\eps_y\cdot y\|_{L^2}$ and $\|h^\eps_y\cdot y\|_{L^2}$ which are generated from the positivity estimate. In our approach, the multipliers $\frac{\p_yu^\eps y^2(1-x)}{u_s}$ and $\frac{\p_yh^\eps y^2(1-x)}{u_s}$ will be tested in the equations. With the factor $1-x$, one can integrate by parts with the respect of $\p_x$ to obtain the desired $y$-weighted terms. In addition, the factor $\frac{1}{u_s}$ is taken to relax the smallness condition imposed on the approximate solution $h_s$.

The rest of the paper is arranged as follows. In Section \ref{sec2}, we construct the approximate solutions by obtaining the weighted estimates of each profile step by step, in which the rigorous proof of the zero-order profile $(u_p^0,v_p^0,h_p^0,g_p^0)$ is postponed to Appendix \ref{ap2}. In Section \ref{sec3}, the existence and the estimates for the remainder profile $(u^\eps,v^\eps,h^\eps,g^\eps)$ in $ {\mathcal{X}}-Norm$ will be achieved.

{\bf{Notation.}}

For reader's convenience, we introduce some weighted Sobolev spaces. For any number $l\in\mathbb{R}$, denote by $L^2_{l}$ the weighted Lebesgue space:
\begin{equation*}
  L^2_l:=\bigg\{f(x,y)\big|~~ f(x,\cdot):\mathbb{R}_+\rightarrow \mathbb{R},\|f\|_{L^2_l}:=(\inty\langle y\rangle^{2l}|f|^2dy)^\frac{1}{2}< +\infty\bigg\},
\end{equation*}
where $\langle y\rangle=\sqrt{1+y^2}$. For any integers $m,\beta,k$ with $\beta+k=m$, and $D^\alpha:=\p_x^\beta\p_y^k$, define the weighted Sobolev space $H_l^m$ by
\begin{equation*}
  H_l^m:=\bigg\{f(x,y)\big|~~ f(x,\cdot):\mathbb{R}_+\rightarrow \mathbb{R},\|f\|_{H_l^m}:=(\inty \langle y\rangle^{2(l+k)}|D^\alpha f|^2dy)^\frac{1}{2}< +\infty\bigg\}.
\end{equation*}
For simplicity, we denote $L^2(\Omega)$ and $L^2(0,+\infty)$ by $L^2$ and $L^2_y$. The integral form $\int_\Omega fdxdy$ and $\inty fdy$ will be simplified to $\iint f$ and $\inty f$, respectively. Also, for simplicity, we denote the trace of a function $f$ on $\{Y=0\}$ as $\bar f$, i.e., $\bar f:=f(x,0)$. In addition, the usual notations will be adopted unless extra statement.

\section{The approximate solutions}\label{sec2}

\subsection{Zero-order boundary layer}\label{sec2.1}
In this subsection, we will construct the zeroth-order boundary layer $(u_p^0,v_p^0,h_p^0,g_p^0,p_p^0)$ and build the well-posedness theory.

Plugging the expansion \eqref{app} into equations \eqref{R_app_expansion}, then the leading order terms in \eqref{R_app_expansion}$_{1,3,4}$ read as
\begin{align*}
   R^{u,0}=
  &(u_e^0+u_p^0)\p_x(u_e^0+u_p^0)+(v_p^0+v_e^1)\p_y(u_e^0+u_p^0)+v_e^0\p_y(u_e^1+u_p^1)+v_e^0\p_Y u_e^0\\
  &-(h_e^0+h_p^0)\p_x(h_e^0+h_p^0)-(g_p^0+g_e^1)\p_y(h_e^0+h_p^0)-g_e^0\p_y(h_e^1+h_p^1)-g_e^0\p_Y h_e^0\\
  &+\p_x(p_e^0+p_p^0)-\nu\p_y^2(u_e^0+u_p^0)+\frac{1}{\sqrt\eps}(v_e^0\p_yu_p^0-g_e^0\p_yh_p^0),\\
   R^{h,0}=
  &(u_e^0+u_p^0)\p_x(h_e^0+h_p^0)+(v_p^0+v_e^1)\p_y(h_e^0+h_p^0)+v_e^0\p_y(h_e^1+h_p^1)+v_e^0\p_Y h_e^0\\
  &-(h_e^0+h_p^0)\p_x(u_e^0+u_p^0)-(g_p^0+g_e^1)\p_y(u_e^0+u_p^0)-g_e^0\p_y(u_e^1+u_p^1)-g_e^0\p_Y u_e^0\\
  &-\kappa\p_y^2(h_e^0+h_p^0)+\frac{1}{\sqrt\eps}(v_e^0\p_yh_p^0-g_e^0\p_yu_p^0),\\
    R^{g,0}=
  &\frac{1}{\sqrt\eps}[u_p^0g_{ex}^0+v_e^0g_{py}^0-h_p^0v_{ex}^0-g_e^0v_{py}^0]+v_e^0(g_{eY}^1+g_{py}^1)-g_e^0(v_{eY}^1+v_{py}^1)+v_p^1g_{ey}^0\\
  &+(u_e^0+u_p^0)\p_x(g_p^0+g_e^1)+(v_p^0+v_e^1)\p_y(g_p^0+g_e^1)+(v _p^0+v_e^1)g_{eY}^0+(u_e^1+u_p^1)g_{ex}^0\\
  &-g_p^1v_{ey}^0-(h_e^0+h_p^0)\p_x(v_p^0+v_e^1)-(g_p^0+g_e^1)\p_y(v_p^0+v_e^1)-(g_p^0+g_e^1)v_{eY}^0\\
  &-(h_e^1+h_p^1)v_{ex}^0-\kappa\p_y^2(g_p^0+g_e^1).
\end{align*}
It should be noted that the ideal MHD profiles $(u_e^i,v_e^i,h_e^i,g_e^i,p_e^i)$ are always evaluated at $(x,Y)$, while the boundary layer profiles $(u_p^i,v_p^i,h_p^i,g_p^i,p_p^i)$ are at $(x,y)$. So we rewrite the following terms in $R^{u,0}, R^{h,0}$ as
\begin{align*}
  (v_p^0+v_e^1)\p_y u_e^0=\sqrt\eps(v_p^0+v_e^1)\p_Y u_e^0,\quad
  v_e^0\p_y u_e^1=\sqrt\eps v_e^0\p_Y u_e^1,\\
  (g_p^0+g_e^1)\p_y h_e^0=\sqrt\eps(g_p^0+g_e^1)\p_Y h_e^0,\quad
  g_e^0\p_y h_e^1=\sqrt\eps g_e^0\p_Y h_e^1,\\
  (v_p^0+v_e^1)\p_y h_e^0=\sqrt\eps(v_p^0+v_e^1)\p_Y h_e^0,\quad
  v_e^0\p_y h_e^1=\sqrt\eps v_e^0\p_Y h_e^1,\\
  (g_p^0+g_e^1)\p_y u_e^0=\sqrt\eps(g_p^0+g_e^1)\p_Y u_e^0,\quad
  g_e^0\p_y u_e^1=\sqrt\eps g_e^0\p_Y u_e^1,\\
  -\nu\p_y^2u_e^0=-\nu\eps\p_Y^2u_e^0,\quad
  -\kappa\p_y^2h_e^0=-\kappa\eps\p_Y^2h_e^0,
\end{align*}
which will be put into the next order in $\eps$. In addition, the following terms in $R^{u,0}$ and $R^{h,0}$ shall be carefully treated respectively
\begin{align*}
  &u_e^0\p_x u_p^0+v_e^1\p_y u_p^0+u_p^0\p_x u_e^0-h_e^0\p_x h_p^0-g_e^1\p_y h_p^0-h_p^0\p_x h_e^0\\
  =&\overline u_e^0\p_x u_p^0+\overline v_e^1\p_y u_p^0+u_p^0\overline{u_{ex}^0}+\sqrt\eps y(u_{eY}^0\p_x u_p^0+v_{eY}^1\p_y u_p^0)+u_p^0( u_{ex}^0-\overline{u_{ex}^0})\\
  &-\overline h_e^0\p_x h_p^0-\overline g_e^1\p_y h_p^0-h_p^0\overline{h_{ex}^0}-\sqrt\eps y(h_{eY}^0\p_x h_p^0+g_{eY}^1\p_y h_p^0)-h_p^0( h_{ex}^0-\overline{h_{ex}^0})+E_1,
\end{align*}
\begin{align*}
  &u_e^0\p_x h_p^0+v_e^1\p_y h_p^0+u_p^0\p_x h_e^0-h_e^0\p_x u_p^0-g_e^1\p_y u_p^0-h_p^0\p_x u_e^0\\
  =&\overline u_e^0\p_x h_p^0+\overline v_e^1\p_y h_p^0+u_p^0\overline{h_{ex}^0}+\sqrt\eps y(u_{eY}^0\p_x h_p^0+v_{eY}^1\p_y h_p^0)+u_p^0(h_{ex}^0- \overline{h_{ex}^0})\\
  &-\overline h_e^0\p_x u_p^0-\overline g_e^1\p_y u_p^0-h_p^0\overline{u_{ex}^0}-\sqrt\eps y(h_{eY}^0\p_x u_p^0+g_{eY}^1\p_y u_p^0)-h_p^0(u_{ex}^0- \overline{u_{ex}^0})+E_3,
\end{align*}
the same for the following terms in $R^{g,0}$
\begin{align*}
  &u_e^0\p_x g_p^0+v_e^1\p_y g_p^0+u_p^0\p_x g_e^1-h_e^0\p_x v_p^0-g_e^1\p_y v_p^0-h_p^0\p_x v_e^1\\
  &~~~+\frac{1}{\sqrt\eps}[u_p^0g_{ex}^0+v_e^0g_{py}^0-h_p^0v_{ex}^0-g_e^0v_{py}^0]\\
  =&\overline u_e^0\p_x g_p^0+\overline v_e^1\p_y g_p^0+\sqrt\eps y(u_{eY}^0\p_x g_p^0+v_{eY}^1\p_y g_p^0)+u_p^0\overline{g_{ex}^1}\\
  &-\overline h_e^0\p_x v_p^0-\overline g_e^1\p_y v_p^0-\sqrt\eps y(h_{eY}^0\p_x v_p^0+g_{eY}^1\p_y v_p^0)-h_p^0\overline{ v_{ex}^1}\\
  &+y u_p^0\p_x\overline{g_{eY}^0}+y g_{py}^0\overline{v_{eY}^0}-y h_p^0\p_x\overline{v_{eY}^0}-y v_{py}^0\overline{g_{eY}^0}+E_4,
\end{align*}
where
\begin{align}\label{2.6}
\begin{cases}
  E_1=
  &\eps\p_x u_p^0\int_0^y \int_y^\theta\p_Y^2 u_e^0(\sqrt\eps\tau)d\tau d\theta
  +\eps\p_y u_p^0\int_0^y \int_y^\theta\p_Y^2 v_e^1(\sqrt\eps\tau)d\tau d\theta\\
  &\quad-\eps\p_x h_p^0\int_0^y \int_y^\theta\p_Y^2 h_e^0(\sqrt\eps\tau)d\tau d\theta
  -\eps\p_y h_p^0\int_0^y \int_y^\theta\p_Y^2 g_e^1(\sqrt\eps\tau)d\tau d\theta,\\
  E_3=
  &\eps\p_x h_p^0\int_0^y \int_y^\theta\p_Y^2 u_e^0(\sqrt\eps\tau)d\tau d\theta
  +\eps\p_y h_p^0\int_0^y \int_y^\theta\p_Y^2 v_e^1(\sqrt\eps\tau)d\tau d\theta\\
  &\quad-\eps\p_x u_p^0\int_0^y \int_y^\theta\p_Y^2 h_e^0(\sqrt\eps\tau)d\tau d\theta
  -\eps\p_y u_p^0\int_0^y \int_y^\theta\p_Y^2 g_e^1(\sqrt\eps\tau)d\tau d\theta,\\
  E_4=
  &\eps\p_x g_p^0\int_0^y \int_y^\theta\p_Y^2 u_e^0(\sqrt\eps\tau)d\tau d\theta
  +\eps\p_y g_p^0\int_0^y \int_y^\theta\p_Y^2 v_e^1(\sqrt\eps\tau)d\tau d\theta\\
  &\quad -\eps\p_x v_p^0\int_0^y \int_y^\theta\p_Y^2 h_e^0(\sqrt\eps\tau)d\tau d\theta
  -\eps\p_y v_p^0\int_0^y \int_y^\theta\p_Y^2 g_e^1(\sqrt\eps\tau)d\tau d\theta\\
  &\quad +\sqrt\eps u_p^0\int_0^y\int_0^\theta\p_x\p_Y^2g_e^0(\sqrt\eps\tau)d\tau d\theta
  +\sqrt\eps g_{py}^0\int_0^y\int_0^\theta\p_Y^2v_e^0(\sqrt\eps\tau)d\tau d\theta\\
  &\quad -\sqrt\eps h_p^0\int_0^y\int_0^\theta\p_x\p_Y^2v_e^0(\sqrt\eps\tau)d\tau d\theta
  -\sqrt\eps v_{py}^0\int_0^y\int_0^\theta\p_Y^2g_e^0(\sqrt\eps\tau)d\tau d\theta\\
  &\quad +\sqrt\eps u_p^0\int_0^y\p_Yg_{ex}^1(\sqrt\eps\tau)d\tau-\sqrt\eps h_p^0\int_0^y\p_Yv_{ex}^1(\sqrt\eps\tau)d\tau.
\end{cases}
\end{align}

Therefore, we obtain an nonlinear MHD boundary layer system for the leading order terms $(u_p^0,v_p^0,h_p^0,g_p^0,p_p^0)$
\begin{align}\label{u_p^0}
\begin{cases}
  &[(\overline u_e^0+u_p^0)\p_x+(v_p^0+\overline v_e^1+y\overline{v_{eY}^0 })\p_y] u_p^0+u_p^0\overline{u_{ex}^0}\\
  &\quad -[(\overline h_e^0+h_p^0)\p_x+(g_p^0+\overline g_e^1+y\overline{g_{eY}^0}  )\p_y] h_p^0-h_p^0\overline{h_{ex}^0} +\p_x p_p^0
  =\nu\p_y^2 u_p^0,\\
  &[(\overline u_e^0+u_p^0)\p_x+(v_p^0+\overline v_e^1+y\overline{v_{eY}^0} )\p_y] h_p^0+u_p^0\overline{h_{ex}^0}\\
  &\quad -[(\overline h_e^0+h_p^0)\p_x+(g_p^0+\overline g_e^1+y\overline{g_{eY}^0} )\p_y] u_p^0-h_p^0\overline{u_{ex}^0}
  =\kappa\p_y^2 h_p^0,\\
  &[(\overline u_e^0+u_p^0)\p_x+(v_p^0+\overline v_e^1+y\overline{v_{eY}^0} )\p_y] (g_p^0+\overline g_e^1+y\overline{g_{eY}^0})\\
  &\quad -[(\overline h_e^0+h_p^0)\p_x+(g_p^0+\overline g_e^1+y\overline{g_{eY}^0} )\p_y] (v_p^0+\overline v_e^1+y\overline{v_{eY}^0})
  =\kappa\p_y^2 g_p^0,\\
  &p_{py}^0=0,\\
  &
  \p_x u_p^0+\p_y v_p^0=\p_x h_p^0+\p_y g_p^0=0,\\
  &(v_p^0,g_p^0)(x,y)=\int_y^{+\infty}\p_x(u_p^0,h_p^0)(x,z)dz,\\
  &(\overline{v_e^1},\overline{g_e^1})=-(v_p^0,g_p^0)(x,0)=-\int_0^{\infty}\p_x(u_p^0,h_p^0)(x,z)dz,\\
  &(u_p^0,\p_y h_p^0)(x,0)=(u_b-\overline u_e^0,0),\quad (u_p^0,h_p^0)(0,y)=(u_{p,0}^0,h_{p,0}^0)(y).
\end{cases}
\end{align}
The fourth equlity of \eqref{u_p^0} implies that the leading order of boundary layers for pressure $p_p^0(x,y)$ should be matched to the outflow pressure on $\{y=0\}$, that is $p_p^0(x,y)=\overline p_e^0(x)$. Next, we will study system \eqref{u_p^0} but ignore the third equation, since the third equation is equivalent to the second equation, by using Bernoulli's law, divergence-free conditions and the boundary condition \eqref{u_p^0}$_7$.

Then, $R^{u,0},R^{h,0}$ are reduced to
\begin{align}\label{R^u,0}
\begin{cases}
  R^{u,0}=
 &\sqrt\eps(v_p^0+v_e^1)\p_Y u_e^0+\sqrt\eps v_e^0\p_Y u_e^1-\sqrt\eps(g_p^0+g_e^1)\p_Y h_e^0-\sqrt\eps g_e^0\p_Y h_e^1\\
 &+\sqrt\eps y(u_{eY}^0 u_{px}^0+v_{eY}^1 u_{py}^0)-\sqrt\eps y(h_{eY}^0 h_{px}^0+g_{eY}^1  h_{py}^0)-\nu\eps\p_Y^2 u_e^0+E_1\\
 &+v_e^0 u_{py}^1-g_e^0 h_{py}^1-y\overline{ v_{eY}^0} u_{py}^0+y\overline{g_{eY}^0} h_{py}^0+u_p^0(u_{ex}^0-\overline{u_{ex}^0})-h_p^0(h_{ex}^0-\overline{h_{ex}^0})\\
 &+\frac{1}{\sqrt\eps}(v_e^0\p_yu_p^0-g_e^0\p_yh_p^0),\\
  R^{h,0}=
 &\sqrt\eps(v_p^0+v_e^1)\p_Y h_e^0+\sqrt\eps v_e^0\p_Y h_e^1-\sqrt\eps(g_p^0+g_e^1)\p_Y u_e^0-\sqrt\eps g_e^0\p_Y u_e^1\\
 &+\sqrt\eps y(u_{eY}^0 h_{px}^0+v_{eY}^1 h_{py}^0)-\sqrt\eps y(h_{eY}^0 u_{px}^0+g_{eY}^1 u_{py}^0)-\kappa\eps\p_Y^2 h_e^0+E_3\\
 &+v_e^0 h_{py}^1-g_e^0 u_{py}^1-y\overline{v_{eY}^0} h_{py}^0+y\overline{g_{eY}^0} u_{py}^0+u_p^0(h_{ex}^0-\overline{h_{ex}^0})-h_p^0(u_{ex}^0-\overline{u_{ex}^0})\\
 &+\frac{1}{\sqrt\eps}(v_e^0\p_yh_p^0-g_e^0\p_yu_p^0),\\
\end{cases}
\end{align}
which would be put into next order in $\eps$.
  %&E_4+ u_p^1g_{ex}^0+v_e^0g_{py}^1-h_p^1v_{ex}^0-g_e^0v_{py}^1+v_p^0(g_{eY}^0-\overline{g_{eY}^0})-g_p^0(v_{eY}^0-\overline{v_{eY}^0})\nonumber\\
And $R^{g,0}$ can be simplified to
\begin{align}\label{R^g,0}
  R^{g,0}=
  &E_4+\sqrt\eps y(u_{eY}^0\p_x g_p^0+v_{eY}^1\p_y g_p^0)-\sqrt\eps y(h_{eY}^0\p_x v_p^0+g_{eY}^1\p_y v_p^0)\nonumber\\
  &+(u_p^1\p_x-v_{py}^1)(\sqrt\eps y\overline{g_{eY}^0}+\eps\int_0^y\int_0^\theta\p_Y^2g_e^0(\sqrt\eps\tau)d\tau d\theta)\nonumber\\
  &-(h_p^1\p_x -g_{py}^1)(\sqrt\eps y\overline{v_{eY}^0}+\eps\int_0^y\int_0^\theta\p_Y^2v_e^0(\sqrt\eps\tau)d\tau d\theta)\\
  &+\sqrt\eps v_p^0\int_0^y\p_Y^2g_e^0(\sqrt\eps\tau)d\tau-\sqrt\eps g_p^0\int_0^y\p_Y^2v_e^0(\sqrt\eps\tau)d\tau\nonumber\\
  &+\sqrt\eps[(v_p^0+v_e^1)g_{eY}^1-(g_p^0+g_e^1)v_{eY}^1+v_p^1g_{eY}^0-g_p^1v_{eY}^0]-\kappa\eps\p_Y^2g_e^1,\nonumber
  \end{align}
where we have used the following two facts, one is
\begin{align*}
  \overline {u_e^0g_{eY}^0-h_e^0v_{eY}^0}=0
\end{align*}
from \eqref{u_e^0}$_4$, the other is
\begin{align*}
  \overline{ u_e^0g_{ex}^1+v_e^1g_{eY}^0-h_e^0v_{ex}^1-g_e^1v_{eY}^0}=0
\end{align*}
from \eqref{u_e^1system}$_4$.

At present, to solve the nonlinear boundary layer system \eqref{u_p^0}, we introduce an auxiliary function $\phi(y)\in C^{\infty}(\mathbb{R}_+)$ to  homogenize the boundary on $y=0,\infty$, that is
\begin{equation*}
  \phi(y)=
\begin{cases}
  y,~~y\geq 2R_0,\\
  0,~~0\leq y \leq R_0,
\end{cases}
\end{equation*}
for some constant $R_0>0$, and define the new unknowns as follows
\begin{align}\label{2.9}
\begin{cases}
  u=u_p^0+\overline u_e^0-u_b-(\overline u_e^0-u_b)\phi'(y),\\
  v=\overline v_e^1+v_p^0+ \overline{u_{ex}^0}(\phi(y)-y)=(\overline v_e^1+v_p^0+y\overline{v_{eY}^0})+\overline{u_{ex}^0}\phi(y),\\
  h=h_p^0+\overline h_e^0-\overline h_e^0\phi'(y),\\
  g=\overline g_e^1+g_p^0+\overline{h_{ex}^0}(\phi(y)-y)=(\overline g_e^1+g_p^0+y\overline{g_{eY}^0})+\overline{h_{ex}^0}\phi(y).
\end{cases}
\end{align}
Then, it is easy to see that the new unknowns $(u,v,h,g)$ satisfy the following boundary conditions and divergence-free conditions
\begin{align}\label{2.10}
\begin{cases}
  (u,h)|_{x=0}=(u_{p,0}^0+(\overline u_e^0-u_b)(1-\phi'),h_{p,0}^0+\overline h_e^0-\overline h_e^0\phi')\triangleq(u_0,h_0)(y),\\
  (u,v,\p_y h,g)|_{y=0}=(0,0,0,0),\\
  (u,h)\rightarrow (0,0),~{\rm{as}}~ y\rightarrow \infty,\\
  \p_x u+\p_y v=\p_x h+\p_y g=0.
\end{cases}
\end{align}

Moreover, we rewrite the system \eqref{u_p^0} for $(u_p^0,v_p^0,h_p^0,g_p^0)$ by the new unknowns $(u,v,h,g)$ as
\begin{align}\label{2.11}
\begin{cases}
  &[\{u+u_b+(\overline u_e^0-u_b)\phi'\}\p_x+(v-\overline{u_{ex}^0}\phi)\p_y]u+u\overline{ u_{ex}^0}\phi'+v(\overline u_e^0-u_b)\phi''\\
  &\quad -[(h+\overline h_e^0\phi')\p_x+(g-\overline{h_{ex}^0}\phi)\p_y]h-h\overline{ h_{ex}^0}\phi'-g\overline h_e^0\phi''-\nu\p_y^2 u=r_1,\\
  &[\{u+u_b+(\overline u_e^0-u_b)\phi'\}\p_x+(v-\overline{u_{ex}^0}\phi)\p_y]h+u\overline{ h_{ex}^0}\phi'+v\overline h_e^0\phi''-\kappa\p_y^2 h\\
  &\quad -[(h+\overline h_e^0\phi')\p_x+(g-\overline{h_{ex}^0\phi})\p_y]u-h\overline{ u_{ex}^0}\phi'-g(\overline u_e^0-u_b)\phi''=r_2,
\end{cases}
\end{align}
where $r_1,r_2$ are defined by
\begin{equation}\label{2.12}
\begin{cases}
  r_1=&\overline{p_{ex}^0}[(\phi')^2-\phi\phi''-2]+u_b\overline{u_{ex}^0}[(\phi')^2-\phi'-\phi\phi'']+\nu(\overline u_e^0-u_b)\phi^{(3)},\\
  r_2=&u_b\overline{h_{ex}^0}[(\phi')^2-\phi'+\phi\phi'']+\kappa\overline h_e^0\phi^{(3)},
\end{cases}
\end{equation}
in which we have used the Bernoulli's law in the calculation of \eqref{2.12}.

By the construction of $\phi(y)$, it is easy to get that
\begin{align}\label{2.13}
  r_1(x,y)=-\overline{p_{ex}^0}(x),\quad r_2(x,y)\equiv 0,&\quad y\geq 2R_0,\nonumber \\
  r_1(x,y)=-2\overline{p_{ex}^0}(x),\quad r_2(x,y)\equiv 0,&\quad 0\leq y\leq R_0.
\end{align}
Then for any real number $\lambda\geq 0,~\text{and}~|\alpha|\leq m$, we have
\begin{align}\label{2.14}
  &\|{\y}^\lambda(D^\alpha r_1,D^\alpha r_2)\|_{L^2(\Omega)}\nonumber\\
  \leq&~ C(u_b)+C\sum_{\gamma\leq|\alpha|+1}\|D^\gamma(\overline u_e^0,\overline h_e^0,\overline p_e^0)(x)\|_{L^2(\mathbb{R})}\leq C(u_b)+C M_0.
\end{align}
In addition, there holds that
\begin{equation}\label{2.15}
  \|(u,h)\|_{H_l^m}-CM_0-C(u_b)\leq \|(u_p^0,h_p^0)\|_{H_l^m}\leq \|(u,h)\|_{H_l^m}+CM_0+C(u_b),
\end{equation}
and similarly, for the initial data, it gives
\begin{equation}\label{2.16}
  \|(u_0,h_0)\|_{H_l^m}-CM_0-C(u_b)\leq \|(u_{p,0}^0,h_{p,0}^0)(y)\|_{H_l^m}\leq \|(u_0,h_0)\|_{H_l^m}+CM_0+C(u_b).
\end{equation}

Now, we state the well-posedness result for $(u,v,h,g)$ as follows.
\begin{proposition}\label{prop2.1}Let $m\geq 5$ be an integer, and $l\geq 0$ a real number. Assume that the outer ideal MHD flows $u_e^0,h_e^0$ are smooth and positive, $\p_Y u_e^0,\p_Y h_e^0$ and their derivatives decay exponentially fast to zero at infinity. Furthermore, suppose that there exist some positive constants $\vartheta_0,\eta_0$ and suitable small $\sigma_0 >0$ satisfying
\begin{align}\label{conditionforu_p^0}
\begin{split}
  &\overline u_e^0+u_p^0(0,y)>\overline h_e^0+h_p^0(0,y)\geq \vartheta_0,\\
  &|u_e^0(0,Y)+u_p^0(0,y)|\gg |h_e^0(0,Y)+h_p^0(0,y)|\geq \eta_0,\\
  &|\y^{l+1}\p_y(u_p^0,h_p^0)(0,y)|\leq \frac{1}{2}\sigma_0,\\
  &|\y^{l+1}\p_y^2(u_p^0,h_p^0)(0,y)|\leq \frac{1}{2}\vartheta_0^{-1}.
\end{split}
\end{align}
Then, there exist smooth solutions $(u,v,h,g)$ to problem \eqref{2.10}-\eqref{2.11} in $[0,L_1]\times[0,\infty)$ with small $L_1>0$, such that
\begin{align}\label{2.20}
  \sup_{0\leq x\leq L_1}\|\y^l D^\alpha(u,h)\|_{L_y^2}+\|\y^l\p_yD^\alpha(u,h)\|_{L^2(0,L_1;L^2(0,\infty))}\leq C.
\end{align}
Moreover, for any $(x,y)\in[0,L_1]\times[0,\infty)$, it holds that
\begin{align}\label{estimatesforu_p^0}
\begin{split}
  &\overline u_e^0+u_p^0(x,y)>\overline h_e^0+h_p^0(x,y)\geq \frac{1}{2}\vartheta_0>0,\\
  &|u_e^0(x,Y)+u_p^0(x,y)|\gg |h_e^0(x,Y)+h_p^0(x,y)|\geq \frac{1}{2}\eta_0>0,\\
  &|\y^{l+1}\p_y(u_p^0,h_p^0)(x,y)|\leq\sigma_0,\\
  &|\y^{l+1}\p_y^2(u_p^0,h_p^0)(x,y)|\leq \vartheta_0^{-1}.
\end{split}
\end{align}
\end{proposition}

\begin{remark}\label{remark}
It should be pointed out that the condition \eqref{conditionforu_p^0}$_2$ is not necessary for the well-posedness of the approximate solutions, which will be applied to control the remainder profile in Section \ref{sec3}. Actually, the condition \eqref{conditionforu_p^0}$_2$ implies that
\begin{align*}
\begin{split}
  |u_e^0(0,Y)+u_p^0(0,y)|-|h_e^0(0,Y)+h_p^0(0,y)|\geq \sigma,
\end{split}
\end{align*}
for a positive constant $\sigma\gg0$.  So we can deduce that
\begin{align*}
\begin{split}
  &|u_e^0(x,Y)+u_p^0(x,y)|-|h_e^0(x,Y)+h_p^0(x,y)|\\
  \geq & |u_e^0(0,Y)+u_p^0(0,y)|-\intx |\p_xu_e^0(x,Y)+\p_xu_p^0(x,y)|\\
  &-|h_e^0(0,Y)+h_p^0(0,y)|-\intx |\p_xu_e^0(x,Y)+\p_xh_p^0(x,y)|\\
  \geq &\sigma-L\|(\p_xu_e^0(x,Y)+\p_xu_p^0(x,y),\p_xh_e^0(x,Y)+\p_xh_p^0(x,y))\|_{L^\infty},
\end{split}
\end{align*}
and hence, for sufficiently small $L$, we have $|u_e^0(x,Y)+u_p^0(x,y)|\gg |h_e^0(x,Y)+h_p^0(x,y)|$ as stated in \eqref{estimatesforu_p^0}$_2$, which will be frequently used in Section \ref{sec3}.
\end{remark}

Thanks to the definition in \eqref{2.9}, the boundedness of \eqref{2.15} and \eqref{2.16}, the proposition for the zero-order boundary layer corrector $(u_p^0,v_p^0,h_p^0,g_p^0)$ follows.

\begin{proposition}\label{prop2.2}
Under the assumptions of Proposition \ref{prop2.1}, there exists a small constant $L_1 > 0$, such that problem \eqref{u_p^0} admits local-in-x classical solutions $(u_p^0,v_p^0,h_p^0,g_p^0)$ in $[0,L_1]\times[0,\infty)$ satisfying
\begin{align}\label{2.24}
  \sup_{0\leq x\leq L_1}\|\y^l D^\alpha(u_p^0,v_p^0,h_p^0,g_p^0)\|_{L_y^2}\leq C,\
  \rm{with}\ |\alpha|\leq m.
\end{align}
\end{proposition}

The proof of Proposition \ref{prop2.1} will be completed in three steps, which is left to Appendix \ref{ap2}. First, we will derive the weighted estimates for $D^\alpha(u,h)$ with $|\alpha|\leq m,~D^\alpha=\p_x^\beta\p_y^k,~\beta\leq m-1$. Second, we shall obtain the weighted estimates for $\p_x^\beta(u,h)$ with $\beta=m$, in other words, m-th order tangential derivative. And the final step is devoted to closing the energy estimates.

\subsection{$\eps^{\frac{1}{2}}$-order correctors}\label{sec2.2}
This subsection is devoted to deducing the systems for $\eps^\frac{1}{2}$-order correctors and the formulation of pressure $p_p^2$. First, collecting all terms with a factor $\sqrt\eps$, together with $\sqrt\eps$-order terms from $R^{u,0},R^{h,0}$, we have
\begin{align}\label{2.25}
\begin{cases}
R^{u,1}=&[(u_e^0+u_p^0)\p_x+(v_p^0+v_e^1)\p_y](u_e^1+u_p^1)+[(u_e^1+u_p^1)\p_x+v_p^1\p_y](u_e^0+u_p^0)\\
        &+\p_x(p_e^1+p_p^1)-\nu\p_y^2(u_e^1+u_p^1)+(yu_{px}^0+v_p^0+v_e^1)u_{eY}^0+yv_{eY}^1 u_{py}^0+v_e^0 u_{eY}^1\\
        &-[(h_e^0+h_p^0)\p_x+(g_p^0+g_e^1)\p_y](h_e^1+h_p^1)-[(h_e^1+h_p^1)\p_x+g_p^1\p_y](h_e^0+h_p^0)\\
        &-(yh_{px}^0+g_p^0+g_e^1)h_{eY}^0-yg_{eY}^1 h_{py}^0-g_e^0 h_{eY}^1+\frac{1}{\sqrt\eps}(v_e^0\p_yu_p^1-g_e^0\p_yh_p^1),\\
R^{h,1}=&[(u_e^0+u_p^0)\p_x+(v_p^0+v_e^1)\p_y](h_e^1+h_p^1)+[(u_e^1+u_p^1)\p_x+v_p^1\p_y](h_e^0+h_p^0)\\
        &-\kappa\p_y^2(h_e^1+h_p^1)+(v_p^0+v_e^1)h_{eY}^0+yu_{eY}^0h_{px}^0+yv_{eY}^1h_{py}^0+v_e^0 h_{eY}^1\\
        &-[(h_e^0+h_p^0)\p_x+(g_p^0+g_e^1)\p_y](u_e^1+u_p^1)-[(h_e^1+h_p^1)\p_x+g_p^1\p_y](u_e^0+u_p^0)\\
        &-(g_p^0+g_e^1)u_{eY}^0-yh_{eY}^0u_{px}^0-yg_{eY}^1u_{py}^0-g_e^0 u_{eY}^1+\frac{1}{\sqrt\eps}(v_e^0\p_yh_p^1-g_e^0\p_yu_p^1).
\end{cases}
\end{align}
We stress that the terms for the ideal flows with scaling $Y=\sqrt\eps y$ will be kept when it is hit by the partial derivative $\p_y$, more specifically,
\begin{align*}
   \p_y^2 u_e^1=\eps\p_Y^2 u_e^1,\quad \p_y^2 h_e^1=\eps\p_Y^2 h_e^1, \\
  (v_p^0+v_e^1)\p_y u_e^1=\sqrt\eps (v_p^0+v_e^1)u_{eY}^1,\\
   -(g_p^0+g_e^1)\p_y h_e^1=-\sqrt\eps (g_p^0+g_e^1)h_{eY}^1,\\
   v_p^1u_{ey}^0-g_p^1h_{ey}^0=\sqrt\eps (v_p^1u_{eY}^0-g_p^1h_{eY}^0),\\
   v_p^1h_{ey}^0-g_p^1u_{ey}^0=\sqrt\eps (v_p^1h_{eY}^0-g_p^1u_{eY}^0).
\end{align*}
The leading interior terms $(u_e^1,v_e^1,h_e^1,g_e^1,p_e^1)$ are taken to satisfy
\begin{align}\label{2.26}
\begin{cases}
  &u_e^0 u_{ex}^1+v_e^0 u_{eY}^1+u_e^1 u_{ex}^0+v_e^1 u_{eY}^0+p_{ex}^1\\
   &\quad-h_e^0 h_{ex}^1-g_e^0 h_{eY}^1-h_e^1 h_{ex}^0-g_e^1 h_{eY}^0=0,\\
  &u_e^0 h_{ex}^1+v_e^0 h_{eY}^1+u_e^1 h_{ex}^0+v_e^1 h_{eY}^0\\
   &\quad-h_e^0 u_{ex}^1-g_e^0 u_{eY}^1-h_e^1 u_{ex}^0-g_e^1 u_{eY}^0=0.
\end{cases}
\end{align}
And the boundary-layer terms $(u_p^1,v_p^1,h_p^1,g_p^1,p_p^1)$ are described by the following system,
\begin{align}\label{2.27}
\begin{cases}
    &u^0u_{px}^1 +v^0u_{py}^1 +u_p^1\p_xu^0 +v_p^1\p_yu^0 +p_{px}^1 -\nu\p_y^2u_p^1 \\
  &\quad -h^0h_{px}^1 -g^0h_{py}^1 -h_p^1\p_xh^0 -g_p^1\p_yh^0\\
  &=-[u_p^0\overline{u_{ex}^1} +\overline{u_e^1}u_{px}^0 +(yu_{px}^0+v_p^0)\overline{u_{eY}^0} +y\overline{v_{eY}^1}u_{py}^0\\
  &\quad -h_p^0\overline{h_{ex}^1} -\overline{h_e^1}h_{px}^0 -(yh_{px}^0+g_p^0)\overline{h_{eY}^0} -y\overline{g_{eY}^1}h_{py}^0]\\
  &\quad +\frac{1}{\sqrt\eps}[(\frac{v_e^0}{\sqrt\eps}-y\overline{v_{eY}^0})u_{py}^0+(u_{ex}^0-\overline{u_{ex}^0})u_p^0]\\
  &\quad -\frac{1}{\sqrt\eps}[(\frac{g_e^0}{\sqrt\eps}-y\overline{g_{eY}^0})h_{py}^0+(h_{ex}^0-\overline{h_{ex}^0})h_p^0]
  :=F_p^1,\\
    &u^0h_{px}^1 +v^0h_{py}^1 +u_p^1\p_xh^0 +v_p^1\p_yh^0 -\kappa\p_y^2h_p^1 \\
  &\quad -h^0u_{px}^1 -g^0u_{py}^1 -h_p^1\p_xu^0 -g_p^1\p_yu^0\\
  &=-[u_p^0\overline{h_{ex}^1} +\overline{u_e^1}h_{px}^0 +v_p^0\overline{h_{eY}^0} +yh_{px}^0\overline{u_{eY}^0}+y\overline{v_{eY}^1}h_{py}^0\\
  &\quad -h_p^0\overline{u_{ex}^1} -\overline{h_e^1}u_{px}^0 -g_p^0\overline{u_{eY}^0} -yu_{px}^0\overline{h_{eY}^0} -y\overline{g_{eY}^1}u_{py}^0]\\
  &\quad +\frac{1}{\sqrt\eps}[(\frac{v_e^0}{\sqrt\eps}-y\overline{v_{eY}^0})h_{py}^0+(h_{ex}^0-\overline{h_{ex}^0})u_p^0]\\
  &\quad -\frac{1}{\sqrt\eps}[(\frac{g_e^0}{\sqrt\eps}-y\overline{g_{eY}^0})u_{py}^0+(u_{ex}^0-\overline{u_{ex}^0})h_p^0]
  :=F_p^2,\\
  &p_{py}^1=0,
\end{cases}
\end{align}
where
\begin{align*}
\begin{cases}
  u^0:=\overline u_e^0+u_p^0,\quad v^0:=y\overline{v_{eY}^0}+\overline v_e^1+v_p^0,\\
  h^0:=\overline h_e^0+h_p^0,\quad g^0:=y\overline{g_{eY}^0}+\overline g_e^1+g_p^0.
\end{cases}
\end{align*}
After constructing the above profiles, the errors are reduced to
\begin{align}\label{R_u^1}
\begin{cases}
   R^{u,0}=&E_1-\nu\eps\p_Y^2u_e^0,\\
   R^{h,0}=&E_3-\kappa\eps\p_Y^2h_e^0,\\
    R^{u,1}=
  &\sqrt\eps[(v_p^0+v_e^1)\p_Y u_e^1 -(g_p^0+g_e^1)\p_Y h_e^1] -\nu\eps\p_Y^2 u_e^1 +(u_e^0-\overline u_e^0)u_{px}^1\\
  &+(v_e^1-\overline v_e^1)u_{py}^1 +u_p^1(u_{ex}^0-\overline{u_{ex}^0}) +\sqrt\eps v_p^1u_{eY}^0-(h_e^0-\overline h_e^0)h_{px}^1\\
  &-(g_e^1-\overline g_e^1)h_{py}^1 -h_p^1(h_{ex}^0-\overline{h_{ex}^0}) -\sqrt\eps g_p^1h_{eY}^0 +u_p^0(u_{ex}^1-\overline{ u_{ex}^1})\\
  &+(u_e^1-\overline u_e^1)u_{px}^0 +(yu_{px}^0+v_p^0)(u_{eY}^0-\overline{u_{eY}^0}) +y(v_{eY}^1-\overline{v_{eY}^1})u_{py}^0\\
  &-h_p^0(h_{ex}^1-\overline{h_{ex}^1}) -(h_e^1-\overline h_e^1)h_{px}^0 -(yh_{px}^0+g_p^0)(h_{eY}^0-\overline{h_{eY}^0})\\
  &-y(g_{eY}^1-\overline{g_{eY}^1})h_{py}^0 +(\frac{v_e^0}{\sqrt\eps}-y\overline{v_{eY}^0})u_{py}^1 -(\frac{g_e^0}{\sqrt\eps}-y\overline{g_{eY}^0})h_{py}^1,\\
    R^{h,1}=
  &\sqrt\eps[(v_p^0+v_e^1)\p_Y h_e^1-(g_p^0+g_e^1)\p_Y u_e^1]-\kappa\eps\p_Y^2 h_e^1 +(u_e^0-\overline u_e^0)h_{px}^1\\
  &+(v_e^1-\overline v_e^1)h_{py}^1 +u_p^1(h_{ex}^0-\overline{h_{ex}^0}) +\sqrt\eps v_p^1h_{eY}^0-(h_e^0-\overline h_e^0)u_{px}^1\\
  &-(g_e^1-\overline g_e^1)u_{py}^1 -h_p^1(u_{ex}^0-\overline{u_{ex}^0}) -\sqrt\eps g_p^1u_{eY}^0 +u_p^0(h_{ex}^1-\overline{ h_{ex}^1})\\
  &+(u_e^1-\overline u_e^1)h_{px}^0 +v_p^0(h_{eY}^0-\overline{h_{eY}^0}) +yh_{px}^0(u_{eY}^0-\overline{u_{eY}^0}) +y(v_{eY}^1-\overline{v_{eY}^1})h_{py}^0\\
  &-h_p^0(u_{ex}^1-\overline{u_{ex}^1}) -(h_e^1-\overline h_e^1)u_{px}^0 -g_p^0(u_{eY}^0-\overline{u_{eY}^0}) -y(h_{eY}^0-\overline{h_{eY}^0})u_{px}^0\\
  &-y(g_{eY}^1-\overline{g_{eY}^1})u_{py}^0+(\frac{v_e^0}{\sqrt\eps}-y\overline{v_{eY}^0})h_{py}^1
  -(\frac{g_e^0}{\sqrt\eps}-y\overline{g_{eY}^0})u_{py}^1.\\
\end{cases}
\end{align}

Now we continue to consider the normal components of \eqref{R_app_expansion}. To begin with, the $\mathcal{O}(\frac{1}{\sqrt\eps})$-order of $\eqref{R_app_expansion}_2$ consists of
\begin{align*}
   R^{v,-1}&=u_p^0v_{ex}^0+v_e^0v_{py}^0-h_p^0g_{ex}^0-g_e^0g_{py}^0.
\end{align*}
Then, the next $\mathcal{O}(1)$-order of $\eqref{R_app_expansion}_2$ is composed of
\begin{align*}
    R^{v,0}&=
  (u_e^0+u_p^0)\p_x(v_p^0+v_e^1)+(v_p^0+v_e^1)\p_y(v_p^0+v_e^1)+v_e^0(v_{eY}^1+v_{py}^1)+(v _p^0+v_e^1)v_{eY}^0\\
  &+(u_e^1+u_p^1)v_{ex}^0+v_p^1v_{ey}^0-(h_e^0+h_p^0)\p_x(g_p^0+g_e^1)-(g_p^0+g_e^1)\p_y(g_p^0+g_e^1)-g_p^1g_{ey}^0\\
  &-g_e^0(g_{eY}^1+g_{py}^1)-(g_p^0+g_e^1)g_{eY}^0-(h_e^1+h_p^1)g_{ex}^0+p_{eY}^1+p_{py}^2-\nu\p_y^2(v_p^0+v_e^1).
\end{align*}
Here we note that $(v_p^1,g_p^1)$ will be determined by the construction of $(u_p^1,h_p^1)$ and the divergence-free conditions. Taking the interior profile $(u_e^1,v_e^1,h_e^1,g_e^1,p_e^1)$ to enjoy
\begin{align}\label{u_e^1}
\left\{
\begin{array}{lll}
  &u_e^0v_{ex}^1+v_e^0v_{eY}^1+u_e^1v_{ex}^0+v_e^1v_{eY}^0+p_{eY}^1-h_e^0g_{ex}^1-g_e^0g_{eY}^1-h_e^1g_{ex}^0-g_e^1g_{eY}^0=0,\\
  &u_e^0g_{ex}^1+v_e^0g_{eY}^1+u_e^1g_{ex}^0+v_e^1g_{eY}^0-h_e^0v_{ex}^1-g_e^0v_{eY}^1-h_e^1v_{ex}^0-g_e^1v_{eY}^0=0.
\end{array}
\right.
\end{align}

Finally, the second-order boundary layer pressure $p_p^2$ is taken to be of the following form
\begin{align}\label{p_p^2}
  p_p^2(x,y)=
 &\int_y^\infty\bigg[\frac{1}{\sqrt\eps}\{u_p^0v_{ex}^0+v_e^0v_{py}^0-h_p^0g_{ex}^0-g_e^0g_{py}^0\}+(u_e^0+u_p^0)v_{px}^0+(v_p^0+v_e^1)v_{py}^0\nonumber\\
 &+u_p^0 v_{ex}^1+u_p^1v_{ex}^0+v_e^0v_{py}^1+v_p^0v_{eY}^0-(h_e^0+h_p^0)g_{px}^0-(g_p^0+g_e^1)g_{py}^0-h_p^0 g_{ex}^1\nonumber\\
 &-h_p^1g_{ex}^0-g_e^0g_{py}^1-g_p^0g_{eY}^0-\nu\p_y^2 v_p^0\bigg](x,\theta)d\theta,
\end{align}
%+\sqrt\eps\{v_p^0v_{eY}^1-g_p^0g_{eY}^1+(u_e^1+u_p^1)v_{px}^0\nonumber\\
%&-(h_e^1+h_p^1)g_{px}^0+u_p^1v_{ex}^1-h_p^1g_{ex}^1+(v_p^0+v_e^1)v_{py}^1-(g_p^0+g_e^1)g_{py}^1\}
so the error term $R^{v,0}$ in this order is reduced to
\begin{align}\label{R^v,0}
  R^{v,0}=
 &\sqrt\eps[(v_p^0+v_e^1)v_{eY}^1-(g_p^0+g_e^1)g_{eY}^1+v_p^1v_{eY}^0-g_p^1g_{eY}^0]-\nu\eps\p^2_Y v_e^1.
\end{align}

Consequently, the combination of \eqref{2.26},\eqref{u_e^1} with the divergence-free conditions constitutes the profile equations for the ideal MHD corrector $(u_e^1,v_e^1,h_e^1,g_e^1,p_e^1)$, while the system \eqref{2.27} together with the divergence-free conditions for the MHD boundary layers $(u_p^1,v_p^1,h_p^1,g_p^1,p_p^1)$.

\subsection{The ideal MHD correctors}\label{sec2.3}
Based on the above analysis, to construct the ideal MHD corrector, one should solve the following system for $(u_e^1,v_e^1,h_e^1,g_e^1,p_e^1)$
\begin{align}\label{u_e^1system}
\begin{cases}
  &u_e^0 u_{ex}^1+v_e^0 u_{eY}^1+u_e^1 u_{ex}^0+v_e^1 u_{eY}^0+p_{ex}^1-h_e^0 h_{ex}^1-g_e^0 h_{eY}^1-h_e^1 h_{ex}^0-g_e^1 h_{eY}^0=0,\\
  &u_e^0v_{ex}^1+v_e^0v_{eY}^1+u_e^1v_{ex}^0+v_e^1v_{eY}^0+p_{eY}^1-h_e^0g_{ex}^1-g_e^0g_{eY}^1-h_e^1g_{ex}^0-g_e^1g_{eY}^0=0,\\
  &u_e^0 h_{ex}^1+v_e^0 h_{eY}^1+u_e^1 h_{ex}^0+v_e^1 h_{eY}^0-h_e^0 u_{ex}^1-g_e^0 u_{eY}^1-h_e^1 u_{ex}^0-g_e^1 u_{eY}^0=0,\\
  &u_e^0g_{ex}^1+v_e^0g_{eY}^1+u_e^1g_{ex}^0+v_e^1g_{eY}^0-h_e^0v_{ex}^1-g_e^0v_{eY}^1-h_e^1v_{ex}^0-g_e^1v_{eY}^0=0,\\
  &u_{ex}^1+v_{eY}^1=h_{ex}^1+g_{eY}^1=0,
\end{cases}
\end{align}
with the boundary conditions
\begin{align}\label{u_e^1boundary}
\begin{cases}
  (v_e^1,g_e^1)(x,0)=-(v_p^0,g_p^0)(x,0),\\
  (v_e^1,g_e^1)\rightarrow(0,0),\ \mathrm{as}\ Y\rightarrow\infty.
\end{cases}
\end{align}

On the one hand, using the first two equations of \eqref{u_e^1system} and the divergence-free conditions, we can derive the equations for the vorticity form

\begin{align}\label{rotationofu_e^1}
  &(u_e^0\p_x+v_e^0\p_Y)\omega_1 +(u_e^1\p_x+v_e^1\p_Y)(\p_Y u_e^0 -\p_x v_e^0) \nonumber\\
  &\quad -(h_e^0\p_x+g_e^0\p_Y)\omega_2 -(h_e^1\p_x+g_e^1\p_Y)(\p_Y h_e^0 -\p_x g_e^0)=0,
\end{align}
 where
 \begin{align*}
   \omega_1= \p_Y u_e^1 -\p_x v_e^1,\quad   \omega_2= \p_Y h_e^1 -\p_x g_e^1.
 \end{align*}
Furthermore, by virtue of the divergence-free conditions, there exist stream functions $\Phi,\Psi$ for velocity and magnetic field, respectively, such that
\begin{equation}\label{streamfunctions}
\begin{cases}
  \nabla^\bot\Phi=(\p_Y\Phi,-\p_x\Phi):=(u_e^1,v_e^1),\\
  \nabla^\bot\Psi=(\p_Y\Psi,-\p_x\Psi):=(h_e^1,g_e^1).
\end{cases}
\end{equation}
And note that
\begin{equation*}
  (\omega_1,\omega_2)=(\Delta\Phi,\Delta\Psi),
\end{equation*}
the equation \eqref{rotationofu_e^1} is equivalent to the following equation:
\begin{align}\label{rotationofpsi}
  &(u_e^0\p_x+v_e^0\p_Y)\Delta\Phi -(h_e^0\p_x+g_e^0\p_Y)\Delta\Psi \nonumber\\
  &\quad =-(u_e^1\p_x+v_e^1\p_Y)(\p_Y u_e^0 -\p_x v_e^0) +(h_e^1\p_x+g_e^1\p_Y)(\p_Y h_e^0 -\p_x g_e^0).
\end{align}

On the other hand, we can rewrite the third and fourth equation of \eqref{u_e^1system} as
\begin{align*}\label{totaldifferential}
  \p_Y(v_e^0h_e^1+h_e^0v_e^1-u_e^0g_e^1-u_e^1g_e^0)=0,\\
  \p_x(v_e^0h_e^1+h_e^0v_e^1-u_e^0g_e^1-u_e^1g_e^0)=0,
\end{align*}
in which we have used the divergence-free conditions. So there exists constant $b$, such that
\begin{equation}\label{prerelation0}
  v_e^0h_e^1 -u_e^0g_e^1=g_e^0 u_e^1 -h_e^0v_e^1 +b,
\end{equation}
or equivalently,
\begin{equation}\label{relationpsi0}
  (u_e^0\p_x+v_e^0\p_Y)\Psi=(h_e^0\p_x+g_e^0\p_Y)\Phi+b.
\end{equation}
Let $Y\to \infty$ in \eqref{prerelation0}, we have
$$b=0,$$
where we have used the fact that $(v^0_e,g^0_e,v^1_e,g^1_e)\to (0,0,0,0)$ as $Y\to \infty$. Therefore, we get
\begin{equation}\label{prerelation}
  v_e^0h_e^1 -u_e^0g_e^1=g_e^0 u_e^1 -h_e^0v_e^1,
\end{equation}
or equivalently, in the formulation of the stream functions:
\begin{equation}\label{relationpsi}
  (u_e^0\p_x+v_e^0\p_Y)\Psi=(h_e^0\p_x+g_e^0\p_Y)\Phi.
\end{equation}

Performing a similar calculation as we did for the equation \eqref{relationpsi},  we can deduce from the third and the fourth equation in system \eqref{u_e^0} for the ideal MHD flows $(u_e^0,v_e^0,h_e^0,g_e^0)$ that
\begin{equation}\label{pre_ratio}
  v_e^0h_e^0 -g_e^0u_e^0=0,
\end{equation}
where the zero boundary conditions $(v^0_e,g^0_e)\to (0,0)$ as $Y\to \infty$ have been applied. It implies that there exists a function $k(x,Y)$ satisfying
\begin{equation}\label{ratio}
  h_e^0=k(x,Y)u_e^0,\quad g_e^0=k(x,Y)v_e^0.
\end{equation}
In addition, we have $0<k<1$, by virtue of the assumption \eqref{condition_for_idealMHD1} stated in the main theorem of this paper.

Putting \eqref{ratio} into \eqref{relationpsi}, we obtain that
\begin{equation}\label{partialk1}
  (u_e^0\p_x+v_e^0\p_Y)\Psi=k(u_e^0\p_x+v_e^0\p_Y)\Phi.
\end{equation}
Moreover, using the divergence-free conditions of velocity field and magnetic field, it gives
\begin{equation}\label{partialk2}
  (u_e^0\p_x+v_e^0\p_Y)k(x,Y)=\p_xh_e^0+\p_Yg_e^0-k(\p_x u_e^0+\p_Y v_e^0)=0.
\end{equation}
Therefore, adding the identity \eqref{partialk2} into the right hand side of equation \eqref{partialk1}, we can deduce a linear first-order partial differential equation read as
\begin{equation}\label{relation}
  (\p_x+\frac{v_e^0}{u_e^0}\p_Y)f=0,
\end{equation}
in which we denote $f:=\Psi-k\Phi$. Define the characteristic curve of \eqref{relation} as the following ordinary differential equation
\begin{equation}\label{character}
  \frac{dY}{dx}=\frac{v_e^0}{u_e^0}(x,Y),
\end{equation}
then the equation \eqref{relation} becomes
\begin{equation*}
  \frac{df}{dx}=0,
\end{equation*}
with given data $f_0:=f(0,Y)=\Psi_0-k_0\Phi_0$. According to characteristic method and local well-posedness theory of ODE, we get
\begin{equation}\label{streamrelation}
  \Psi=k\Phi+F(u_e^0,v_e^0,f_0),
\end{equation}
in which $F$ is a function determined by the data $u_e^0,v_e^0,f_0$. The equality \eqref{streamrelation} gives the relationship between two stream functions, and thus the following high order terms of $\Psi$ can be expressed by $\Phi$ as
\begin{align}\label{psiterms}
\begin{split}
  &\p_x\Psi=\p_xk\cdot\Phi+k\p_x\Phi+\p_xF,\\
  &\p_Y\Psi=\p_Yk\cdot\Phi+k\p_Y\Phi+\p_YF,\\
  &\Delta\Psi=k\Delta\Phi+\Phi\Delta k+2\nabla k\cdot\nabla \Phi+\Delta F.
\end{split}
\end{align}

Hence, plugging \eqref{psiterms} into second-order vorticity equation \eqref{rotationofpsi} with definition \eqref{streamfunctions}, using the identity \eqref{partialk2} and the positivity of $u_e^0$, we can conclude that the stream function $\Phi$ enjoys the following equation
\begin{align}\label{Hphi}
  (\p_x+\frac{v_e^0}{u_e^0}\p_Y)G(\Phi)=H(\Phi),
\end{align}
where $G(\Phi)$ and $H(\Phi)$ are defined by
$$G(\Phi):=-(1-k^2)\Delta\Phi+k\Delta k\cdot\Phi+2k\nabla k\cdot\nabla \Phi+k\Delta F,$$
 and
\begin{align*}
  H(\Phi):=
  &-\frac{1}{u_e^0}\bigg[(\p_Yk\cdot\Phi+k\p_Y\Phi+\p_YF)\p_x(\p_Y h_e^0 -\p_x g_e^0)\\
  &-(\p_xk\cdot\Phi+k\p_x\Phi+\p_xF)\p_Y(\p_Y h_e^0 -\p_x g_e^0)\\
  &-(\p_Y\Phi\p_x-\p_x\Phi\p_Y)(\p_Y u_e^0 -\p_x v_e^0)\bigg].
\end{align*}
Define the characteristic curve of \eqref{Hphi} the same as \eqref{character}, then the equation \eqref{Hphi} becomes
\begin{equation*}
  \frac{dG(\Phi)}{dx}=H(\Phi).
\end{equation*}
And thus, with given data $G(\Phi)|_{x=0}=G_0(\Phi_0(Y))$, the problem in consideration turns to the following equation
\begin{align}\label{Gphi}
  G(\Phi(x,Y))=\intx H(\Phi(s,Y))ds+G_0(\Phi_0(Y)),
\end{align}
with the given boundary conditions:
\begin{equation}\label{varphiboundary}
\begin{cases}
  \Phi|_{Y=0}=1+\int_0^x v_p^0(s,0)\rm{ds},\\
  \Phi(0,Y)=\Phi_0(Y),~~\Phi(L,Y)=\Phi_L(Y).
\end{cases}
\end{equation}
 We take the data $\Phi_0(Y)$ and $\Phi_L(Y)$ to be sufficiently smooth and decay exponentially fast at infinity. In addition, we assume the compatibility conditions $\Phi_0(0)=1$ and $\Phi_L(0)=1+\int_0^L v_p^0(s,0)\rm{ds}$. Note that the first condition in \eqref{varphiboundary} is equivalent to $\Phi_x(x,0)=-v_e^1(x,0)=v_p^0(x,0)$, where the general constant have been selected to 1, without loss of generality. To solve the problem \eqref{Gphi}-\eqref{varphiboundary}, we suppose that the data are in the well-prepared sense defined as follow (see also \cite{Iyernonshear}).
\begin{definition}[Well-prepared boundary data]\label{wellprepared}
 We call the boundary data are well-prepared up to order two, if the boundary value of $\Phi_{YY}(x,Y)$ satisfying
\begin{align*}
 \Phi_{YY}(x,0)=
 &-\Phi_{xx}(x,0)+\frac{1}{1-k^2(x,0)}(k\Delta k\cdot\Phi+2k\nabla k\cdot\nabla \Phi+k\Delta F)(x,0)\\
 &-\frac{1}{1-k^2(x,0)}\bigg\{\intx H(\Phi(s,0))ds+G_0(\Phi_0(0))\bigg\},
\end{align*}
with the following compatibility conditions
\[ \Phi_{YY}(x,0)|_{x=0}=\p_{YY}\Phi_0(0), \quad \Phi_{YY}(x,0)|_{x=L}=\p_{YY}\Phi_L(0).\]
Repeating the above procedure, we achieve the generalization to order $k$.
\end{definition}

Rewrite \eqref{Gphi} as follows,
\begin{align}\label{phieq}
\begin{split}
  -\Delta\Phi=&\mathcal{F}(\Phi),
\end{split}
\end{align}
with the source term
\begin{align*}
  \mathcal{F}(\Phi)
  :=&-\frac{k\Delta k}{1-k^2}\cdot\Phi-\frac{2k\nabla k}{1-k^2}\cdot\nabla\Phi
  -\frac{k}{1-k^2}\Delta F\\
  &+\frac{1}{1-k^2}\int_0^x H(\Phi(s,Y))ds+\frac{1}{1-k^2}G_0(\Phi_0(Y)).
  \end{align*}
Recall the facts that
\begin{align}\label{assumption_on_k}
\begin{split}
  &0<\|k\|_{L^\infty}<<1,\\
 &\|\nabla^m k\|_{L^2}\leq C \textrm{\ for\ arbitrary\ }m.
\end{split}
\end{align}
We can apply the contraction mapping principle to determine $\Phi$. Indeed, for any $\tilde{\Phi}$, by the standard theory of elliptic system, there exists a unique solution to the following equation
$$-\Delta\Phi=\mathcal{F}(\tilde{\Phi}),$$
which produces a solution mapping $\mathbf{T}:\tilde\Phi\mapsto \Phi$ for $\tilde\Phi,\Phi\in X:=\{\Phi|\|Y^n\Phi\|_{H^m}<C(n,m)\}$. The contraction of the mapping follows from the smallness of $k$ and $L>0$. This gives the fixed point $\mathbf{T}(\Phi)=\Phi$, which is the solution to the original problem \eqref{phieq}. Moreover, the magnetic fields can be also determined.

With the sketch, our attentions will be paid on the estimates for $\Phi$.
To begin with, let us introduce a function $B_\Phi$ defined by
\begin{align*}
  B_\Phi(x,Y)
  =&(1-\frac{x}{L})\frac{\Phi_{0}(Y)}{\Phi(0,0)}(1+\intx v_p^0(s,0)ds)+\frac{x}{L}\frac{\Phi_{L}(Y)}{\Phi(L,0)}(1+\intx v_p^0(s,0)ds)\\
  =&(1-\frac{x}{L})\Phi_{0}(Y)(1+\intx v_p^0(s,0)ds)+\frac{x}{L}\frac{\Phi_{L}(Y)}{\Phi(L,0)}(1+\intx v_p^0(s,0)ds)\\
  =&\Phi_{0}(Y)(1+\intx v_p^0(s,0)ds)+\frac{x}{L}(\frac{\Phi_{L}(Y)}{\Phi(L,0)}-\Phi_{0}(Y))(1+\intx v_p^0(s,0)ds).
\end{align*}
in the case when $\Phi(L,0)$ is nonzero, otherwise, for the case that $\Phi(L,0)=0$, we will take $\frac{x}{L}[\Phi_L(Y)+(1+\intx v_p(s,0)ds)]$ as the second term, instead of the ratio term $\frac{x}{L}\frac{\Phi_{L}(Y)}{\Phi(L,0)}(1+\intx v_p(s,0)ds)$. Thanks to the compatibility conditions at the corner, it is easy to get that $B_\Phi$ satisfies the boundary conditions \eqref{varphiboundary}.
In addition, if we suppose that $|Y^n\p_Y^m(\Phi_{L}(Y)-\Phi_{0}(Y))|\leq CL$, then it yields that $Y^nB_\Phi\in H^m$ for arbitrary $n,m$. And we introduce the function $E_\Phi(x,Y)$ satisfying
\begin{align}\label{Ephi}
\begin{split}
  E_\Phi
  :=&-\Delta B_\Phi+\frac{k\Delta k}{1-k^2}\cdot B_\Phi+\frac{2k\nabla k}{1-k^2}\cdot\nabla B_\Phi+\frac{k}{1-k^2}\Delta F\\
  &-\frac{1}{1-k^2}G_0(\Phi_0(Y))-\frac{1}{1-k^2}\intx H(B_\Phi(s,Y))ds.
\end{split}
\end{align}
Then $E_\Phi$ is sufficiently smooth and enjoys the following weighted estimate
\begin{equation}\label{Ephiestimate}
  \|Y^nE_\Phi\|_{H^m}\leq C,\quad \mathrm{for}\ n,m\geq 0.
\end{equation}

Introduce function $\Phi^*$ as
\begin{equation*}
  \Phi=\Phi^*+ B_\Phi,
\end{equation*}
then the function $\Phi^*$ solves the following elliptic problem
\begin{align}\label{Phistareq}
\begin{split}
  -\Delta\Phi^*
  =&-\frac{k\Delta k}{1-k^2}\cdot\Phi^*-\frac{2k\nabla k}{1-k^2}\cdot\nabla\Phi^*\\
  &+\frac{1}{1-k^2}\intx H(\Phi^*(s,Y))ds-E_\Phi,
\end{split}
\end{align}
with the homogenous boundary conditions $\Phi^*|_{\p\Omega}=0$.

In what follows, we will derive the weighted $H^k$ estimate for the equation \eqref{Phistareq} of $\Phi^*$. Let us perform the $H^1$ estimate at first. Multiplying the equation \eqref{Phistareq} by $\Phi^*$ and integrating by parts, it gives
\begin{align}
\begin{split}
\iint|\nabla\Phi^*|^2
  =&-\iint\frac{k\Delta k}{1-k^2}\cdot|\Phi^*|^2-\iint\frac{2k\nabla k}{1-k^2}\cdot\nabla\Phi^*\cdot\Phi^*\\
  &+\iint\bigg[\frac{1}{1-k^2}\intx H(\Phi^*(s,Y))ds\cdot\Phi^*\bigg]-\iint E_\Phi\cdot\Phi^*\\
  \leq& CL\|\nabla\Phi^*\|_{L^2}^2+\|E_\Phi\|_{L^2}^2,
\end{split}
\end{align}
where the first term has been estimated as follows and others are similar:
\begin{align}\label{H^1pre}
\begin{split}
  -\iint\frac{k\Delta k}{1-k^2}\cdot|\Phi^*|^2
  =&-\iint\frac{k\Delta k}{1-k^2}\cdot(\intx\p_x\Phi^*)^2\\
  \leq&\iint |\frac{k\Delta k}{1-k^2}|\cdot L(\int_0^L|\p_x\Phi^*|^2)\\
  \leq&|\frac{k\Delta k}{1-k^2}|\cdot L^2\cdot\|\p_x\Phi^*\|_{L^2}^2.
\end{split}
\end{align}
Thus, using the smallness of $L$, we have
\begin{equation}\label{H^1}
  \|\Phi^*\|_{H^1}^2\leq \|E_\Phi\|_{L^2}^2\leq C,
\end{equation}
where the $L^2$ norm of $\Phi^*$ is bounded by $\|\nabla\Phi^*\|_{L^2}$ via performing the similar arguments of \eqref{H^1pre}.

Next, to obtain the $H^2$ estimate, applying $\partial_Y $ on \eqref{Phistareq} to get
\begin{align}\label{H^2equ}
\begin{split}
  -\Delta\p_Y\Phi^*+\p_Y(\frac{k\Delta k}{1-k^2}\cdot\Phi^*)+\p_Y(\frac{2k\nabla k}{1-k^2}\cdot\nabla\Phi^*)
  =\p_Y(\frac{\intx H(\Phi^*(s,Y))ds}{1-k^2})-\p_YE_\Phi.
\end{split}
\end{align}
Multiplying the above equation by $\p_Y\Phi^*$ and integrating in $[0,L]\times [0,\infty)$, the left-hand side is reduced to
\begin{align}\label{H^2.left}
\begin{split}
  &\iint\bigg[-\Delta\p_Y\Phi^*+\p_Y\left(\frac{k\Delta k}{1-k^2}\cdot\Phi^*\right)+\p_Y\left(\frac{2k\nabla k}{1-k^2}\cdot\nabla\Phi^*\right)\bigg]\cdot\p_Y\Phi^*\\
  =&\iint|\nabla\p_Y\Phi^*|^2-\iint\frac{k\Delta k}{1-k^2}\cdot\Phi^*\cdot\p_{YY}\Phi^*-\iint\frac{2k\nabla k}{1-k^2}\cdot\nabla\Phi^*\cdot\p_{YY}\Phi^*\\
  &-\int_0^L\left(-\p_Y^2\Phi^*+\frac{2k\p_Y k}{1-k^2}\cdot\p_Y\Phi^*\right)\cdot\p_Y\Phi^*\bigg|_{Y=0}dx.
\end{split}
\end{align}
Meanwhile, the right-hand side reads as
\begin{align}\label{H^2.right}
\begin{split}
  &\iint\p_Y\left(\frac{\intx H(\Phi^*(s,Y))ds}{1-k^2}\right)\cdot\p_Y\Phi^*
  -\p_YE_\Phi\cdot\p_Y\Phi^*\\
  =&-\iint\frac{\intx H(\Phi^*(s,Y))ds}{1-k^2}\cdot\p_{YY}\Phi^*-\iint E_\Phi\cdot\p_{YY}\Phi^*\\
  &-\int_0^L\left(\frac{\intx H(\Phi^*(s,Y))ds}{1-k^2}+ E_\Phi\right)\cdot\p_Y\Phi^*|_{Y=0}dx.
\end{split}
\end{align}
By taking $Y=0$ in the equation \eqref{Phistareq}, the boundary terms in the above two equalities \eqref{H^2.left} \eqref{H^2.right} vanish. Applying \textit{H\"{o}lder inequality}, we can conclude that
\begin{equation}\label{H^2pre}
  \|\nabla\p_Y\Phi^*\|_{L^2}^2
  \leq\bigg(\frac{|k\Delta k|+L}{1-k^2}\cdot\|\Phi^*\|_{L^2}+\frac{|2k\nabla k|+L}{1-k^2}\cdot\|\nabla\Phi^*\|_{L^2}+\|E_\Phi\|_{L^2}\bigg)\|\p_Y^2\Phi^*\|_{L^2}.
\end{equation}
Using \textit{Y\"{o}ung inequality}, the $H^1$ estimate for $\Phi^*$ and the estimate for $E_\Phi$, we have
\begin{equation}\label{H^2}
  \|\nabla\p_Y\Phi^*\|_{L^2}^2\leq C.
\end{equation}
In addition, the $L^2$ estimate of $\p_{xx}\Phi^*$ yields by using the equation \eqref{H^2equ} and estimate \eqref{H^2}. So we get the estimates of $\Phi^*$ in $H^2$ norm.

Let us turn to consider the weighted estimates. For any $n\geq 1$, we consider the elliptic problem for $Y^n\Phi^*$ as follows:
\begin{align}\label{weightedH^2}
\begin{split}
  -\Delta(Y^n\Phi^*)
  =&-2\p_Y(Y^n)\p_Y\Phi^*-\p_Y^2(Y^n)\Phi^*-Y^n\frac{k\Delta k}{1-k^2}\cdot\Phi^*-Y^n\frac{2k\nabla k}{1-k^2}\cdot\nabla\Phi^*\\
  &+\frac{Y^n}{1-k^2}\intx H(\Phi^*(s,Y))ds-Y^nE_\Phi,
\end{split}
\end{align}
 with the homogenous boundary conditions. By induction, assume that $Y^{n-1}\Phi^*$ is uniformly bounded in $H^2$ norm, and then the terms in the right hand side of the elliptic problem \eqref{weightedH^2} is uniformly bounded in $H^1$ norm. The similar argument applied above for the unweighted norm yields that $\|Y^n\Phi^*\|_{H^2}\leq C$ for any $n\geq 1$.

Furthermore, we derive higher regularity estimates. To derive the $H^3$ estimates, we consider the following elliptic problem for $\p_Y\Phi^*$,
\begin{align}\label{H^3equ}
\begin{split}
  &-\Delta\p_Y\Phi^*+\p_Y(\frac{k\Delta k}{1-k^2}\cdot\Phi^*)+\p_Y(\frac{2k\nabla k}{1-k^2}\cdot\nabla\Phi^*)\\
  =&\p_Y(\frac{\intx H(\Phi^*(s,Y))ds}{1-k^2})-\p_YE_\Phi.
\end{split}
\end{align}
Testing the \eqref{H^3equ} by $-\Delta\p_Y\Phi^*$, and integrating in $[0,L]\times[0,\infty)$, we can obtain the boundedness for $\|\Delta\p_{Y}\Phi^*\|_{L^2}$. Next, we test the following equation
\begin{align}\label{H^4equ2}
\begin{split}
  &-\Delta\p_{x}\Phi^*+\p_{x}(\frac{k\Delta k}{1-k^2}\cdot\Phi^*)+\p_{x}(\frac{2k\nabla k}{1-k^2}\cdot\nabla\Phi^*)\\
  =&\p_{x}(\frac{\intx H(\Phi^*(s,Y))ds}{1-k^2})-\p_{x}E_\Phi
\end{split}
\end{align}
by function $-\Delta\p_{x}\Phi^*$ to achieve the estimate for $\|\Delta\p_{x}\Phi^*\|_{L^2}$. Then it gives the $H^3$ estimate for $\Phi^*$.

To estimate the $H^4$ norm, consider the following elliptic problem for $\p_{YY}\Phi^*$,
\begin{align}\label{H^4equ1}
\begin{split}
  &-\Delta\p_{YY}\Phi^*+\p_{YY}(\frac{k\Delta k}{1-k^2}\cdot\Phi^*)+\p_{YY}(\frac{2k\nabla k}{1-k^2}\cdot\nabla\Phi^*)\\
  =&\p_{YY}(\frac{\intx H(\Phi^*(s,Y))ds}{1-k^2})-\p_{YY}E_\Phi.
\end{split}
\end{align}
Testing the \eqref{H^4equ1} by $-\Delta\p_{YY}\Phi^*$, and integrating in $[0,L]\times[0,\infty)$, we can obtain the estimate for $\|\Delta\p_{YY}\Phi^*\|_{L^2}$. Next, the estimate for $\|\Delta\p_{xY}\Phi^*\|_{L^2}$ can be achieved through testing the following equation
\begin{align}\label{H^4equ2}
\begin{split}
  &-\Delta\p_{xY}\Phi^*+\p_{xY}(\frac{k\Delta k}{1-k^2}\cdot\Phi^*)+\p_{xY}(\frac{2k\nabla k}{1-k^2}\cdot\nabla\Phi^*)\\
  =&\p_{xY}(\frac{\intx H(\Phi^*(s,Y))ds}{1-k^2})-\p_{xY}E_\Phi
\end{split}
\end{align}
by function $-\Delta\p_{xY}\Phi^*$. Finally, combining the following equation
\begin{align}\label{H^4equ3}
\begin{split}
  &-\Delta\p_{xx}\Phi^*+\p_{xx}(\frac{k\Delta k}{1-k^2}\cdot\Phi^*)+\p_{xx}(\frac{2k\nabla k}{1-k^2}\cdot\nabla\Phi^*)\\
  =&\p_{xx}(\frac{\intx H(\Phi^*(s,Y))ds}{1-k^2})-\p_{xx}E_\Phi\\
\end{split}
\end{align}
with the estimate of $\|\Delta\p_{YY}\Phi^*\|_{L^2}$, the $L^2$ estimate of $\p_x^4\Phi^*$ follows, so we accomplish the $H^4$ estimate for $\Phi^*$. In a similar argument, any $H^m$ norm estimate for $\Phi^*$ can be achieved.

Therefore, based on the above analysis, we have the following lemma.
%%%%%%%%%%%%%%%%%%%%%%
\begin{lemma}\label{lemma2.1}
Under the assumption of well-prepared boundary data and the hypotheses on $k(x,Y)$ stated in \eqref{assumption_on_k}, the problem \eqref{Gphi}-\eqref{varphiboundary} admits a solution $\Phi$ and the following estimate holds
\begin{equation}\label{estimateforvarphi}
  \|Y^n\Phi\|_{H^m}\leq C(n,m).
\end{equation}
\end{lemma}
Combining \eqref{estimateforvarphi} with the relationship \eqref{streamrelation} between two stream functions, we have
\begin{lemma}\label{lemma2.2}
Under the assumption of well-prepared boundary data and the hypotheses on $k(x,Y)$ stated in \eqref{assumption_on_k}, the problem \eqref{rotationofpsi} and \eqref{relationpsi} admits a solution $(\Phi,\Psi)$ which satisfies the following estimate
\begin{equation}\label{estimateforpsi}
  \|Y^n(\Phi,\Psi)\|_{H^m}\leq C(n,m).
\end{equation}
\end{lemma}

Therefore, according to the definition of \eqref{streamfunctions}, with the estimate of $\Phi,\Psi$ in hands, we are ready to give the estimates for $(u_e^1,v_e^1,h_e^1,g_e^1)$.
\begin{proposition}\label{prop2.3}
There exists a solution $(u_e^1,v_e^1,h_e^1,g_e^1)$ to the problem \eqref{u_e^1system}-\eqref{u_e^1boundary} which satisfies the following estimate
\begin{equation}\label{estimateforu_e^1}
  \|Y^n(u_e^1,v_e^1,h_e^1,g_e^1)\|_{H^m}\leq C.
\end{equation}
where the constant $C$ is depending on $n,m$.
\end{proposition}

\subsection{$\sqrt\eps$-order MHD boundary layer correctors}\label{sec2.4}
In this subsection, we will construct the  MHD boundary layer correctors $(u_p^1,v_p^1,h_p^1,g_p^1,p_p^1)$ by solving \eqref{2.27}. For simplicity, we drop the superscript $1$ of the solutions.

The system for MHD boundary layer correctors $(u_p,v_p,h_p,g_p,p_p)$ is described by
\begin{align}\label{u_p}
\begin{cases}
    &u^0u_{px} +v^0u_{py} +u_p\p_xu^0 +v_p\p_yu^0 +p_{px} -\nu\p_y^2u_p \\
  &\quad\quad -h^0h_{px} -g^0h_{py} -h_p\p_xh^0 -g_p\p_yh^0
  =F_p^1,\\
    &u^0h_{px} +v^0h_{py} +u_p\p_xh^0 +v_p\p_yh^0 -\kappa\p_y^2h_p \\
  &\quad\quad -h^0u_{px} -g^0u_{py} -h_p\p_xu^0 -g_p\p_yu^0
  =F_p^2,\\
  &p_{py}=0,
\end{cases}
\end{align}
where $F_p^1,F_p^2$ are defined in \eqref{2.27}, together with  the following boundary conditions
\begin{align}\label{upboundary}
\begin{cases}
  (u_p,h_p)(0,y)=(u_{p,0},h_{p,0})(y),\\
  (u_p,\p_y h_p)(x,0)=-(\overline{u_e^1},\overline{\p_Y h_e^0})(x),\\
  (v_p,g_p)(x,0)=(0,0),\\
  (u_p,h_p)\rightarrow (0,0),~\text{as}~ y\rightarrow \infty.
\end{cases}
\end{align}
To solve the problem, we will also consider the divergence-free conditions
\begin{equation}\label{updivergence}
  u_{px}+v_{py}=h_{px}+g_{py}=0,
\end{equation}
here we note that $(v_p,g_p)$ will be constructed by $(v_p,g_p)=-\int_0^y (u_{px},h_{px})(x,z)dz$.

Moreover, the second equation in \eqref{u_p} can be rewritten as the following form
\begin{align}\label{u_pprestream}
\begin{split}
  &\p_y[-u^0g_p -g^0u_p +v^0h_p +h^0v_p] -\kappa\p_y^2h_p\\
  &=\p_y\bigg[-y\overline{h_{eY}^0}v_p^0 +y\overline{u_{eY}^0}g_p^0 -y\overline{v_{eY}^1}h_p^0 +y\overline{g_{eY}^1}u_p^0 +\overline{u_e^1}g_p^0 -\overline{h_e^1}v_p^0\\
  &\quad\quad +\frac{1}{\sqrt\eps}(\frac{v_e^0}{\sqrt\eps}h_p^0-\frac{g_e^0}{\sqrt\eps}u_p^0+y\overline{u_{ex}^0}h_p^0-y\overline{h_{ex}^0}u_p^0)\bigg].
\end{split}
\end{align}
Thanks to the divergence-free conditions, there exists a stream function $\tilde\psi$ satisfying
\begin{equation}\label{u_pstream}
  (h_p,g_p)=(\p_y\tilde\psi,-\p_x\tilde\psi),\quad {\rm{with}}~~\tilde\psi|_{y=0}=0.
\end{equation}
By virtue of \eqref{u_pprestream}, using the boundary conditions and the definition of $\tilde\psi$, we have
\begin{align}\label{tildepsi}
  &u^0\p_x\tilde\psi -g^0u_p +v^0\p_y\tilde\psi +h^0v_p -\kappa\p_y^2\tilde\psi\nonumber\\
  =&-y\overline{h_{eY}^0}v_p^0 +y\overline{u_{eY}^0}g_p^0 -y\overline{v_{eY}^1}h_p^0 +y\overline{g_{eY}^1}u_p^0 +\overline{u_e^1}g_p^0 -\overline{h_e^1}v_p^0 +\overline{u_e^1g_e^1} -\overline{h_e^1v_e^1} +\kappa\overline{h_{eY}^0}\nonumber\\
  &+\frac{1}{\sqrt\eps}(\frac{v_e^0}{\sqrt\eps}h_p^0-\frac{g_e^0}{\sqrt\eps}u_p^0+y\overline{u_{ex}^0}h_p^0-y\overline{h_{ex}^0}u_p^0).
\end{align}

Here, we state the well-posedness theory of $(u_p,v_p,h_p,g_p)$ which shall be established by using a similar energy estimate method to that of $(u_p^0,v_p^0,h_p^0,g_p^0)$.
\begin{proposition}\label{prop2.4}
There exists a solution $(u_p,v_p,h_p,g_p)$ to the problem \eqref{u_p}-\eqref{updivergence} in $[0,L_2]\times[0,+\infty)$ with $0< L_2\leq L_1$ which satisfies the following estimate
\begin{equation}\label{estimateforu_p}
  \sup_{0\leq x\leq L_2}\|y^m\p_x^k(u_p,h_p)\|_{L_y^2} +\|y^m\p_x^k(u_{py},h_{py})\|_{L^2}+\|(v_p,g_p)\|_{L^\infty}\leq C.
\end{equation}
where the constant $C$ is depending on k,m.
\end{proposition}
\begin{proof}
For simplicity, we give here an outline about the application of the energy method in Appendix \ref{ap2} without details. First, we will get the weighted estimates for $D^\alpha(u_p,h_p)$ with $|\alpha|\leq m,~D^\alpha=\p_x^\beta\p_y^k,~\beta\leq m-1$. Second, we shall obtain the weighted estimates for $\p_x^\beta(u_p,h_p)$ with $\beta=m$ via introducing the new quantities:
\begin{align*}
  u_p^\beta:=\p_x^\beta u_p-\frac{\p_y u_p^0}{\bar h_e^0+h_p^0}\p_x^{\beta}\tilde\psi,\quad
  h_p^\beta:=\p_x^\beta h_p-\frac{\p_y h_p^0}{\bar h_e^0+h_p^0}\p_x^{\beta}\tilde\psi.
\end{align*}
According to the system \eqref{u_p} for $u_p$ and \eqref{tildepsi} for $\tilde\psi$, we can deduce the system for $(u_p^\beta,h_p^\beta)$, in which the tough terms involving $\p_x^\beta(v_p,g_p)$ are cancelled. So the weighted $L^2$-estimates of $(u_p^\beta,h_p^\beta)$ follows. And then we can receive the desired weighted $L^2$-estimates of $(\p_x^\beta u_p,\p_x^\beta h_p)$ by proving the $L^2$-norm equivalence between $(u_p^\beta,h_p^\beta)$ and $(\p_x^\beta u_p,\p_x^\beta h_p)$, so as to close the energy estimates.
\end{proof}

Therefore, we are ready to get estimates on every error term mentioned above by using the fact that the ideal MHD flows are evaluated at $(x,Y)$, while the boundary layers profiles are at $(x,y)$. We can obtain that
\begin{align}\label{E134}
  \|E_1,E_3\|_{L^2}\lesssim\eps^{\frac{3}{4}},\quad \|E_4\|_{L^2}\lesssim\eps^{\frac{1}{4}}.
\end{align}
and
\begin{align}\label{estimateforR^v,g,0}
  \|R^{v,0},R^{g,0},R^{u,1},R^{h,1}\|_{L^2}\lesssim \eps^{\frac{1}{4}}.
\end{align}
Indeed, the terms in $E_1,E_3,E_4,R^{v,0},R^{g,0},R^{u,1},R^{h,1}$ can be handled similarly to the following terms
\begin{align*}
    \|\eps\p_x u_p^0\int_0^y \int_y^\theta \p_Y^2 u_e^0(\sqrt\eps\tau)d\tau d\theta\|_{L^2}
  &\leq\eps\|\p_x u_p^0\cdot y^2\|_{ L^\infty} \|\p_Y^2 u_e^0(\sqrt\eps\cdot)\|_{ L^2} \lesssim \eps^{\frac{3}{4}},\\
    \|\sqrt\eps u_p^0\int_0^y\p_Yg_{ex}^1(\sqrt\eps\tau)d\tau\|_{L^2}
  &\leq\sqrt\eps\| u_p^0\cdot y\|_{ L^\infty} \|\p_Y g_{ex}^1(\sqrt\eps\cdot)\|_{ L^2} \lesssim \eps^{\frac{1}{4}},\\
    \|\sqrt\eps g_p^1h_{eY}^0\|_{L^2}
  &\leq\sqrt\eps\|g_p^1\|_{ L^\infty}\|h_{eY}^0(\sqrt\eps\cdot)\|_{ L^2} \lesssim \eps^{\frac{1}{4}},\\
    \|y(g_{eY}^1-\overline{g_{eY}^1})h_{py}^0\|_{L^2}
  &\leq\sqrt\eps\|y^2h_{py}^0\|_{ L^\infty}\|\p_Y^2g_e^1(\sqrt\eps\cdot)\|_{ L^2} \lesssim \eps^{\frac{1}{4}},\\
    \|(\frac{v_e^0}{\sqrt\eps}-y\overline{v_{eY}^0})u_{py}^1\|_{L^2}
  &\leq\|y^2u_{py}^1\|_{ L^\infty}\|\sqrt\eps\p_Y^2v_e^0(\sqrt\eps\cdot)\|_{ L^2} \lesssim \eps^{\frac{1}{4}},
\end{align*}
where we have used \textit{Proposition \ref{prop2.1}--\ref{prop2.4}, Hardy inequality} and the boundary condition of $v_e^0$ on $\{Y=0\}$.

\subsection{Construction of approximate solutions}\label{subsec2.5}
This subsection is to construct the approximate solutions for system \eqref{1.1}. To this end, we define a cut-off function $\chi(Y)$ supported in $[0,1]$
\begin{align*}
  \chi(Y)=
\begin{cases}
 1,\quad Y\in[0,1],\\
 0,\quad Y\in[2,+\infty),
\end{cases}
\end{align*}
and a smooth boundary corrector $\rho(x,Y)$ with compact support
\begin{equation*}
  \rho(x,Y)=-\overline{\p_Yh_e^1(x)}\cdot Y\chi(Y)
\end{equation*}
satisfying
\begin{equation*}
  \p_Y\rho(x,Y)|_{Y=0}=-\overline{\p_Yh_e^1(x)}.
\end{equation*}
Let $(h_e^1,g_e^1)$ be constructed as in the previous subsection, introducing
\begin{equation}\label{tildeh_e^1}
\begin{cases}
\tilde h_e^1(x,Y)=h_e^1(x,Y)+\rho(x,Y),\\
\tilde g_e^1(x,Y)=g_e^1(x,Y)-\int_0^Y\p_x\rho(x,s)ds,
\end{cases}
\end{equation}
so that the boundary value of $\p_Yh_e^1(x,Y)$ can be cancelled on $\{y=0\}$, which ensures the boundary condition $\p_yh_{app}(x,y)|_{y=0}=0$ for approximate solution. Since $(\tilde h_e^1,\tilde g_e^1)$ satisfy the divergence-free condition as well, we will still denote it by $(h_e^1,g_e^1)$ in the approximate solutions expansion for convenience.

Next, we introduce the boundary layer correctors which will be used in the boundary layer expansion. Let $(u_p,v_p,h_p,g_p)$ be constructed as in the previous subsection, define
\begin{equation}\label{cutoff}
\begin{cases}
  (u_p^1,h_p^1)(x,y)=\chi(\sqrt\eps y)(u_p,h_p)+\sqrt\eps\chi'(\sqrt\eps y)\int_0^y (u_p,h_p)(x,s)ds,\\
  (v_p^1,g_p^1)(x,y)=\chi(\sqrt\eps y)(v_p,g_p).
\end{cases}
\end{equation}
Clearly, $(u_p^1,v_p^1,h_p^1,g_p^1)$ is a divergence-free vector field, that is
\begin{equation}\label{}
  u_{px}^1+v_{py}^1=h_{px}^1+g_{py}^1=0.
\end{equation}

Using the estimates of $(u_p,h_p)$ in Proposition \ref{prop2.4}, we have
\begin{equation*}
    |\sqrt\eps\chi'(\sqrt\eps y)\int_0^y(u_p,h_p)(x,s)ds|
  \leq \sqrt\eps y|\chi'(\sqrt\eps y)|\cdot\|(u_p,h_p)\|_{ L^\infty}\leq C,
\end{equation*}
and it follows that
\begin{align}\label{estimateforu_p^1}
  \sup_{0\leq x\leq L_2}\|y^m\p_x^k(u_p^1,h_p^1)\|_{L_y^2} +\|y^m\p_x^k(u_{py}^1,h_{py}^1)\|_{L^2}+\|(v_p^1,g_p^1)\|_{L^\infty}\leq C.
\end{align}

Additionally, the new error in $\mathcal{O}(\sqrt\eps)$-order created by the cut-off layer is
\begin{equation}\nonumber
\begin{aligned}
  R^{u,1}_p=
  &-(1-\chi)F_p^1+(u^0 \p_x+u_x^0 +v^0 \p_y-\nu\p_y^2)\left(\sqrt\eps\chi'\int_0^y u_p(x,s)ds\right)\nonumber\\
  &-(h^0 \p_x+h_x^0 +g^0 \p_y)\left(\sqrt\eps\chi'\int_0^y h_p(x,s)ds\right)-\sqrt\eps\chi'h_pg^0\nonumber\\
  &-2\sqrt\eps\nu\chi'u_{py}+u_p(\sqrt\eps\chi'v^0-\nu\eps\chi''),\nonumber\\
    R^{h,1}_p=
  &-(1-\chi)F_p^2+(u^0 \p_x-u_x^0+v^0 \p_y-\kappa\p_y^2)\left(\sqrt\eps\chi'\int_0^y h_p(x,s)ds\right)\nonumber\\
  &-(h^0 \p_x-h_x^0+g^0 \p_y)\left(\sqrt\eps\chi'\int_0^y u_p(x,s)ds\right)-\sqrt\eps\chi'u_p g^0\nonumber\\
  &-2\sqrt\eps\kappa\chi' h_{py}+h_p(\sqrt\eps\chi'v^0-\kappa\eps\chi''),
\end{aligned}
\end{equation}

which can be estimated as
\begin{equation}\label{estimateforR^u,1_p}
  \|(R^{u,1}_p,R^{h,1}_p)\|_{ L^2}\lesssim \eps^{\frac{1}{4}}.
\end{equation}

Indeed, thanks to the property that the zero-order boundary layer correctors is rapidly decaying as $y\rightarrow\infty$, when $\sqrt\eps y\geq 1$, the terms $F_p^1, F_p^2$ with coefficient $1-\chi$ are of order $\eps^n$ for large enough $n\geq 0$. And the terms involved with $\sqrt\eps\chi'$ are bounded by
\begin{align*}
 \sqrt\eps\|\chi'(\sqrt\eps y)\|_{L^2}\leq \eps^{1/4},
\end{align*}
the boundedness is also satisfied for the $\eps\chi''$ terms since $\eps$ is sufficiently small. These points, taken together, lead to the summary \eqref{estimateforR^u,1_p}.

These new error terms $R^{u,1}_p,R^{h,1}_p$ would contribute into the error term $R^u_{app}, R^h_{app}$, which are defined by
\begin{equation}\label{R^u_app}
\begin{aligned}
  R^u_{app}:=&\tilde R^{u,2}+\sqrt\eps R^{u,1}+\sqrt\eps R^{u,1}_p+\eps p_{px}^2,\nonumber\\
  R^h_{app}:=&\tilde R^{h,2}+\sqrt\eps R^{h,1}+\sqrt\eps R^{h,1}_p,
\end{aligned}
\end{equation}
where
\begin{equation*}\label{R^u}
\begin{aligned}
    \tilde R^{u,2}&=
   -\eps\nu(\Delta u_e^0+\p_x^2 u_p^0)+\eps^{\frac{3}{2}}\left[v_p^1u_{eY}^1-g_p^1h_{eY}^1-\nu(\p_x^2 u_e^1+\p_x^2 u_p^1)\right]+E_1\\
  &\quad +\eps\left[(u_e^1+u_p^1)\p_x(u_e^1+u_p^1)-(h_e^1+h_p^1)\p_x(h_e^1+h_p^1)+v_p^1u_{py}^1-g_p^1h_{py}^1\right],\\
    \tilde R^{h,2}&=
  -\eps\kappa(\Delta h_e^0+\p_x^2 h_p^0)+\eps^{\frac{3}{2}}\left[v_p^1h_{eY}^1-g_p^1u_{eY}^1-\kappa(\p_x^2 h_e^1+\p_x^2 h_p^1)\right]+E_3\\
  &\quad +\eps\left[(u_e^1+u_p^1)\p_x(h_e^1+h_p^1)-(h_e^1+h_p^1)\p_x(u_e^1+u_p^1)+v_p^1h_{py}^1-g_p^1u_{py}^1\right],
\end{aligned}
\end{equation*}
by collecting the error terms from $R^{u,0},R^{h,0}$ and some higher order terms. The estimates on $\tilde R^{u,2}, \tilde R^{h,2}$ can be obtained similarly to the estimates on \eqref{E134} and \eqref{estimateforR^v,g,0}, that is
\begin{equation}\label{estimateforR^u,2}
  \|(\tilde R^{u,2}, \tilde R^{h,2})\|_{ L^2}\lesssim \eps^{\frac{3}{4}}.
\end{equation}

 And recalling the definition of $p_p^2$, we have
\begin{align*}
  p_{px}^2(x,y)=
  &\int_y^\infty\bigg[\frac{1}{\sqrt\eps}\{u_p^0v_{exx}^0+v_e^0v_{pxy}^0-h_p^0g_{exx}^0-g_e^0g_{pxy}^0\}-\nu\p_y^2 v_{px}^0
  +(u_e^0+u_p^0)v_{pxx}^0\nonumber\\
  &+(v_p^0+v_e^1)v_{pxy}^0+u_p^0 v_{exx}^1+u_p^1v_{exx}^0+v_e^0v_{pxy}^1+v_p^0v_{exY}^0-(h_e^0+h_p^0)g_{pxx}^0\nonumber\\
  &-(g_p^0+g_e^1)g_{pxy}^0-h_p^0 g_{exx}^1-h_p^1g_{exx}^0-g_e^0g_{pxy}^1-g_p^0g_{exY}^0\bigg](x,\theta)d\theta,
\end{align*}
where the divergence-free conditions for $(u_p^i,v_p^i)$ and $(h_p^i,g_p^i)$ have been used. For the $\mathcal{O}(\frac{1}{\sqrt\eps})$ terms, for any $n\geq 2$, there holds
\begin{align*}
     \int_y^\infty\frac{v_e^0}{\sqrt\eps}v_{pxy}^0
  &\leq C\y^{-(n-2)}\left\|\frac{v_e^0}{Y}\right\|_{L^\infty}\|\y^nv_{pxy}^0\|_{L^2(0,\infty)},\\
     \int_y^\infty\frac{v_{exx}^0}{\sqrt\eps}u_p^0
  &=\int_y^\infty\frac{-\int_0^Yu_{exx}^0}{\sqrt\eps y}yu_p^0\\
  &\leq C\y^{-(n-2)}\eps^{-\frac{1}{4}}\|u_{exx}^0\|_{L^2(0,\infty)}\|\y^nu_p^0\|_{L^\infty},
\end{align*}
where we have used \textit{Hardy inequality}. Similar arguments can be applied to achieve the estimate for the terms $\frac{v_e^0}{\sqrt\eps}g_{pxy}^0$ and $\frac{v_{exx}^0}{\sqrt\eps}h_p^0$. The other terms can be handled similarly to the estimates on \eqref{E134} and \eqref{estimateforR^v,g,0}, so it gives
\begin{equation}\label{ppx2}
  \|p_{px}^2\|_{L^2}\lesssim \eps^{-\frac{1}{4}}.
\end{equation}
Therefore, collecting \eqref{estimateforR^v,g,0}, \eqref{estimateforR^u,1_p}-\eqref{ppx2}, we obtain
\begin{equation}\label{R^u,h_app}
  \|R_{app}^u,R_{app}^h\|_{L^2}\lesssim \eps^{\frac{3}{4}}.
\end{equation}

On the other hand, error terms for the normal components of \eqref{R_app_expansion} is reduced to
\begin{align*}
   R^v_{app}=
  &R^{v,0}+\sqrt\eps[\{(u_e^0+u_p^0)+\sqrt\eps(u_e^1+u_p^1)\}\p_x+(v_p^0+v_e^1+\sqrt\eps v_p^1)\p_y]v_p^1\\
  &-\sqrt\eps[\{(h_e^0+h_p^0)+\sqrt\eps(h_e^1+h_p^1)\}\p_x+(g_p^0+g_e^1+\sqrt\eps g_p^1)\p_y]g_p^1\\
  &+\sqrt\eps[(u_e^1+u_p^1)\p_x+v_p^1\p_y](v_p^0+v_e^1)-\sqrt\eps[(h_e^1+h_p^1)\p_x+g_p^1\p_y](g_p^0+g_e^1)\\
  &-\nu\sqrt\eps[\Delta v_e^0+\Delta_\eps v_p^1+\sqrt\eps\p_x^2(v_e^1+ v_p^0)],\\
   R^g_{app}=
  &R^{g,0}+\sqrt\eps[\{(u_e^0+u_p^0)+\sqrt\eps(u_e^1+u_p^1)\}\p_x+(v_p^0+v_e^1+\sqrt\eps v_p^1)\p_y]g_p^1\\
  &-\sqrt\eps[\{(h_e^0+h_p^0)+\sqrt\eps(h_e^1+h_p^1)\}\p_x+(g_p^0+g_e^1+\sqrt\eps g_p^1)\p_y]v_p^1\\
  &+\sqrt\eps[(u_e^1+u_p^1)\p_x+v_p^1\p_y](g_p^0+g_e^1)-\sqrt\eps[(h_e^1+h_p^1)\p_x+g_p^1\p_y](v_p^0+v_e^1)\\
  &-\kappa\sqrt\eps[\Delta g_e^0+\Delta_\eps g_p^1+\sqrt\eps\p_x^2 (g_e^1+g_p^0)].
\end{align*}
Recalling the estimate of $R^{v,0},R^{g,0}$ in \eqref{estimateforR^v,g,0}, it gives
\begin{equation}\label{R^v,g_app}
  \|R_{app}^v,R_{app}^g\|_{L^2}\lesssim \eps^{\frac{1}{4}}.
\end{equation}

Therefore, according to the above estimates on each profile for the error terms, we come to the following conclusion:
\begin{proposition} Suppose the assumption in Theorem \ref{th1.1} holds, then there exist approximate solutions $(u_{app},v_{app},h_{app},g_{app})$ satisfying
\begin{align}\label{R_app}
    \|R_{app}^u,R_{app}^h,\sqrt\eps(R_{app}^v,R_{app}^g)\|_{L^2}+\|\y\p_y\{R_{app}^u,R_{app}^h,\sqrt\eps(R_{app}^v,R_{app}^g)\}\|_{L^2}
  \lesssim \eps^{\frac{3}{4}}.
\end{align}
\end{proposition}

\section{Proof of the main theorem}\label{sec3}
\subsection{The remainder terms}
Since the approximate solutions have been constructed as above, now we are on a position to derive the estimates for the remainder terms.
To this end, we first deduce the system for the remainder terms read as follows, denote the approximate solutions in expansion \eqref{expansion} by $(u_s,v_s,h_s,g_s)$ for simplification,
\begin{align}\label{3.1}
\begin{cases}
  u_s:=u_{app}=u_e^0(x,Y)+u_p^0(x,y)+\sqrt{\eps}[u_e^1(x,Y)+u_p^1(x,y)],\\
  v_s:=v_{app}=\frac{v_e^0}{\sqrt\eps}(x,Y)+v_p^0(x,y)+v_e^1(x,Y)+\sqrt{\eps}v_p^1(x,y),\\
  h_s:=h_{app}=h_e^0(x,Y)+h_p^0(x,y)+\sqrt{\eps}[(h_e^1(x,Y)+h_p^1(x,y)],\\
  g_s:=g_{app}=\frac{g_e^0}{\sqrt{\eps}}(x,Y)+g_p^0(x,y)+g_e^1(x,Y)+\sqrt{\eps}g_p^1(x,y).
\end{cases}
\end{align}
Then we have the equations for the remainders $(u^\eps,v^\eps,h^\eps,g^\eps,p^\eps)$ read as
\begin{align}\label{u_epssystem}
\begin{cases}
    &(u_s\p_x u^\eps+u^\eps\p_x u_s+v_s\p_y u^\eps+v^\eps\p_y u_s) +p^\eps_x-\nu\Delta_\eps u^\eps \\
  &\quad\quad\quad -(h_s\p_x h^\eps+h^\eps\p_x h_s+g_s\p_y h^\eps+g^\eps\p_y h_s)=R_1(u^\eps,v^\eps,h^\eps,g^\eps),\\
    &(u_s\p_x v^\eps+u^\eps\p_x v_s+v_s\p_y v^\eps+v^\eps\p_y v_s) +\frac{p^\eps_y}{\eps}-\nu\Delta_\eps v^\eps \\
  &\quad\quad\quad -(h_s\p_x g^\eps+h^\eps\p_x g_s+g_s\p_y g^\eps+g^\eps\p_y g_s)=R_2(u^\eps,v^\eps,h^\eps,g^\eps),\\
    &(u_s\p_x h^\eps+u^\eps\p_x h_s+v_s\p_y h^\eps+v^\eps\p_y h_s) -\kappa\Delta_\eps h^\eps \\
  &\quad\quad\quad -(h_s\p_x u^\eps+h^\eps\p_x u_s+g_s\p_y u^\eps+g^\eps\p_y u_s)=R_3(u^\eps,v^\eps,h^\eps,g^\eps),\\
    &(u_s\p_x g^\eps+u^\eps\p_x g_s+v_s\p_y g^\eps+v^\eps\p_y g_s) -\kappa\Delta_\eps g^\eps\\
  &\quad\quad\quad -(h_s\p_x v^\eps+h^\eps\p_x v_s+g_s\p_y v^\eps+g^\eps\p_y v_s)=R_4(u^\eps,v^\eps,h^\eps,g^\eps),\\
  &\p_x u^\eps+\p_y v^\eps=\p_x h^\eps+\p_y g^\eps=0.
\end{cases}
\end{align}
where the source term $R_i(i=1,2,3,4)$ are given by
\begin{align*}
\begin{cases}
    R_1
  &:=\eps^{-\frac{1}{2}-\gamma}R_{app}^u-\eps^{\frac{1}{2}+\gamma}(u^\eps\p_x u^\eps +v^\eps\p_y u^\eps -h^\eps\p_x h^\eps -g^\eps\p_y h^\eps)\\
  &:=\eps^{-\frac{1}{2}-\gamma}R_{app}^u-N^u(u^\eps,v^\eps,h^\eps,g^\eps),\\
    R_2
  &:=\eps^{-\frac{1}{2}-\gamma}R_{app}^v-\eps^{\frac{1}{2}+\gamma}(u^\eps\p_x v^\eps +v^\eps\p_y v^\eps -h^\eps\p_x g^\eps -g^\eps\p_y g^\eps)\\
  &:=\eps^{-\frac{1}{2}-\gamma}R_{app}^v-N^v(u^\eps,v^\eps,h^\eps,g^\eps),\\
    R_3
  &:=\eps^{-\frac{1}{2}-\gamma}R_{app}^h-\eps^{\frac{1}{2}+\gamma}(u^\eps\p_x h^\eps +v^\eps\p_y h^\eps -h^\eps\p_x u^\eps -g^\eps\p_y u^\eps)\\
  &:=\eps^{-\frac{1}{2}-\gamma}R_{app}^h-N^h(u^\eps,v^\eps,h^\eps,g^\eps),\\
    R_4
  &:=\eps^{-\frac{1}{2}-\gamma}R_{app}^g-\eps^{\frac{1}{2}+\gamma}(u^\eps\p_x g^\eps +v^\eps\p_y g^\eps -h^\eps\p_x v^\eps -g^\eps\p_y v^\eps)\\
  &:=\eps^{-\frac{1}{2}-\gamma}R_{app}^g-N^g(u^\eps,v^\eps,h^\eps,g^\eps).
\end{cases}
\end{align*}
And we take the following boundary conditions into consideration:
\begin{align}\label{u_epsboundary}
  \begin{cases}
  (u^\eps ,v^\eps ,\p_y h^\eps ,g^\eps)|_{y=0}=(0,0,0,0),\quad
  (u^\eps,v^\eps,h^\eps,g^\eps)|_{x=0}=(0,0,0,0),\\
  p^\eps-2\nu\eps\p_x u^\eps|_{x=L}=0,\quad
  \p_yu^\eps+\nu\eps\p_x v^{\epsilon}|_{x=L}=0,\quad
  (h^\eps,\p_x g^\eps)|_{x=L}=0.
\end{cases}
\end{align}

Thanks to the above constructed profiles of $(u_s,v_s,h_s,g_s)$, we have the following boundedness which will be used frequently throughout this section:
\begin{align}\label{summary_estimates}
  \|y^j\p_x^i\p_y^j(u_s,h_s),\p_x^i(v_s,g_s)\|_{L^\infty}\lesssim 1,
\end{align}
where $i=0,1,2$ and $j=0,1,2$. Furthermore, by using the assumption \eqref{condition_for_idealMHD6} imposed on $u_e^0,h_e^0$, the estimates of $u_p^0,u_e^1,u_p^1,h_p^0,h_e^1,h_p^1$ and the smallness of $\eps$, it gives
\begin{equation}\label{yu_sy}
  \|y\p_y(u_s,h_s)\|_{L^\infty}<C\sigma_0,\quad {\rm for}\ x\in[0,L],
\end{equation}
for sufficiently small constant $\sigma_0$ and some small $0< L\ll 1$.
In addition, by using the inequality \eqref{estimatesforu_p^0}$_2$, the smallness of $\varepsilon$ and the boundedness of the first-order profiles, the following strict positivity holds
\begin{equation}\label{u_s_positivity}
  u_s\gtrsim 1.
\end{equation}

\subsection{Linear stability estimates}
For the sake of solving the nonlinear system \eqref{u_epssystem} for $(u^\eps,v^\eps,h^\eps,g^\eps)$, we first consider the following linearized equations
\begin{align}\label{linearizedu_s}
\begin{cases}
  -\nu\Delta_\eps u^\eps+S_u(u^\eps,v^\eps)-S_h(h^\eps,g^\eps)+p_x^\eps=f_1,\\
  -\nu\Delta_\eps v^\eps+S_v(u^\eps,v^\eps)-S_g(h^\eps,g^\eps)+\frac{p_y^\eps}{\eps }=f_2,\\
  -\kappa\Delta_\eps h^\eps+K_h(h^\eps,g^\eps)-K_u(u^\eps,v^\eps)=f_3,\\
  -\kappa\Delta_\eps g^\eps+K_g(h^\eps,g^\eps)-K_v(u^\eps,v^\eps)=f_4,\\
  \p_x u^\eps+\p_y v^\eps=\p_x h^\eps+\p_y g^\eps=0,
\end{cases}
\end{align}
where $f_1,f_2,f_3,f_4$ are given functions in $L^2$, and
\begin{align}\label{S_u}
\begin{cases}
  &S_u(u^\eps,v^\eps)=u_s\p_x u^\eps+u^\eps\p_x u_s+v_s\p_y u^\eps+v^\eps\p_y u_s,\\
  &S_h(h^\eps,g^\eps)=h_s\p_x h^\eps+h^\eps\p_x h_s+g_s\p_y h^\eps+g^\eps\p_y h_s,\\
  &S_v(u^\eps,v^\eps)=u_s\p_x v^\eps+u^\eps\p_x v_s+v_s\p_y v^\eps+v^\eps\p_y v_s,\\
  &S_g(h^\eps,g^\eps)=h_s\p_x g^\eps+h^\eps\p_x g_s+g_s\p_y g^\eps+g^\eps\p_y g_s,\\
  &K_h(h^\eps,g^\eps)=u_s\p_x h^\eps-h^\eps\p_x u_s+v_s\p_y h^\eps-g^\eps\p_y u_s,\\
  &K_u(u^\eps,v^\eps)=h_s\p_x u^\eps-u^\eps\p_x h_s+g_s\p_y u^\eps-v^\eps\p_y h_s,\\
  &K_g(h^\eps,g^\eps)=u_s\p_x g^\eps-h^\eps\p_x v_s+v_s\p_y g^\eps-g^\eps\p_y v_s,\\
  &K_v(u^\eps,v^\eps)=h_s\p_x v^\eps-u^\eps\p_x g_s+g_s\p_y v^\eps-v^\eps\p_y g_s,
\end{cases}
\end{align}
together with the boundary conditions \eqref{u_epsboundary}.

In this subsection, we shall prove the following \textit{Proposition \ref{prop3.1}}:
\begin{proposition}\label{prop3.1} Consider solutions $[u^\eps,v^\eps,h^\eps,g^\eps]\in\mathcal{X}$ to linearized problem \eqref{linearizedu_s} with boundary conditions \eqref{u_epsboundary}, then it satisfies the following estimate:
\begin{equation}\label{estimateforX1}
  \|(u^\eps,v^\eps,h^\eps,g^\eps)\|_{X_1}^2+\|(u^\eps,v^\eps,h^\eps,g^\eps)\|_{B}^2\lesssim \|(f_1,f_3)\|_{ L^2}+\sqrt\eps\|(f_2,f_4)\|_{ L^2}+\mathcal{R},
\end{equation}
where
\begin{align}\label{X_1norm}
\begin{split}
  \|(u^\eps,v^\eps,h^\eps,g^\eps)\|_{X_1}
  &:=\| \{u^\eps_y,h^\eps_y,\sqrt\eps(u^\eps_x,h^\eps_x)\}\cdot y \|_{ L^2}+\|v^\eps_y,g^\eps_y,\sqrt\eps(v^\eps_x,g^\eps_x)\|_{ L^2}\\
  &\quad +\| \{u^\eps_{yy},h^\eps_{yy},\sqrt\eps (u^\eps_{xy},h^\eps_{xy}),\eps(u^\eps_x,h^\eps_x)\}\cdot y \|_{ L^2},
\end{split}
\end{align}
and
\begin{align*}
  \mathcal{R}:=&\iint \p_yf_1\p_y\left\{\frac{u^\eps wy^2}{u_s}\right\}
  -\iint\eps\p_yf_2\p_x\left\{\frac{u^\eps wy^2}{u_s}\right\}\\
  &+\iint\p_yf_3\p_y\left\{\frac{h^\eps wy^2 }{u_s}\right\}
  -\iint\eps\p_yf_4\p_x\left\{\frac{h^\eps wy^2 }{u_s}\right\}.
\end{align*}
\end{proposition}

The proof of Proposition \ref{prop3.1} will be completed by three lemmas. In the proof, the following Poincar\'{e} type inequalities will be applied frequently
\begin{align}\label{embedding}
  \| u^\eps\|_{ L^2}^2\leq L\| u^\eps_x\|_{ L^2}^2,\quad \| h^\eps\|_{ L^2}^2\leq L\| h^\eps_x\|_{ L^2}^2.
\end{align}
\begin{lemma}\label{basicestimatelemma}{\rm (Basic energy estimates)} Let $[u^\eps,v^\eps,h^\eps,g^\eps]\in\mathcal{X}$ be the solutions to linearized system \eqref{linearizedu_s} with boundary conditions \eqref{u_epsboundary}, then the following estimate holds
\begin{align}\label{basicestimate}
\begin{split}
  &\nu\|\nabla_\eps u^\eps\|_{ L^2}^2+ \kappa\|\nabla_\eps h^\eps\|_{ L^2}^2 +\int_{x=L}u_s\cdot(|u^\eps|^2+\eps|v^\eps|^2+\eps|g^\eps|^2) \\
  &\quad\quad \lesssim L\|(\nabla_\eps v^\eps,\nabla_\eps g^\eps)\|_{ L^2}^2 +\|(f_1,f_3)\|_{ L^2}^2+\eps\|(f_2,f_4)\|_{ L^2}^2,
\end{split}
\end{align}
\end{lemma}
\begin{proof}
This lemma can be obtained by standard energy arguments. Indeed, multiplying each equation in system \eqref{linearizedu_s} by $\{u^\eps,\eps v^\eps,h^\eps,\eps g^\eps\}$ respectively and following the arguments as in \cite{DLX}, the lemma follows. Here we omit the details, see Lemma 3.1 in \cite{DLX} for instance.
\end{proof}

Before giving the following two lemmas, we remind the readers of a frequently used estimate that is
\begin{equation*}
  \bigg\|\frac{h_s}{u_s}\bigg\|_{L^\infty}\ll1,
\end{equation*}
for $(x,y)\in[0,L]\times[0,+\infty)$ with $0<L\ll1$, which can be proved by using the positivity of $u_s$ in \eqref{u_s_positivity}, the second estimate in \eqref{estimatesforu_p^0} and the estimates for $h_p^0,h_e^1,h_p^1$ with the smallness of $\eps$.

\begin{lemma}\label{positivitylemma}{\rm (Positivity estimates)} Consider solutions $[u^\eps,v^\eps,h^\eps,g^\eps]\in\mathcal{X}$ to linearized problem \eqref{linearizedu_s}  with boundary conditions \eqref{u_epsboundary}, suppose that $\|\frac{h_s}{u_s}\|_{L^\infty}\ll1$, $\|y\p_y (u_s,h_s)\|_{ L^\infty}<C\sigma_0$ uniform in $0<L\ll1$ and the normal velocity enjoys $\|\frac{v_e^0}{Y}\|_{L^\infty}\ll 1$, then the following estimate holds
\begin{align}\label{positivityestimate}
\begin{split}
  &\quad \|\nabla_\eps v^\eps\|_{ L^2}^2+\|\nabla_\eps g^\eps\|_{ L^2}^2+\eps\int_{x=L}\frac{|v^\eps_y|^2}{u_s}\\
  &\leq C\|(u^\eps_y,h^\eps_y)\|_{ L^2}^2+\mathcal{O}(v_e^0)\cdot\|\{\nabla_\eps u^\eps,\nabla_\eps h^\eps\}\cdot y\|_{ L^2}^2+\|(f_1,f_3)\|_{ L^2}^2+\eps\|(f_2,f_4)\|_{ L^2}^2\\
  &\quad +\left(\left\|\frac{h_s}{u_s}\right\|_{L^\infty}+L+\sqrt\eps+\mathcal{O}(v_e^0)+\|y\p_y(u_s,h_s)\|_{ L^\infty}\right)\cdot\|(\nabla_\eps v^\eps,\nabla_\eps g^\eps)\|_{ L^2}^2,
\end{split}
\end{align}
where $\mathcal{O}(v_e^0)$ stands for some small constant under the assumption of $\|\frac{v_e^0}{Y}\|_{L^\infty}\ll 1$.
\end{lemma}
\begin{proof}
Applying operator $\big[\p_y(\frac{v^\eps}{u_s}),-\eps\p_x(\frac{v^\eps}{u_s}),\p_y(\frac{g^\eps}{u_s}),-\eps\p_x(\frac{g^\eps}{u_s})\big]$  to the system \eqref{linearizedu_s}, integrating them in $[0,L]\times[0,\infty)$ and adding them up, we have
\begin{align}\label{postivityestimate1}
\begin{split}
  &\iint \p_y(\frac{v^\eps}{u_s})(S_u(u^\eps,v^\eps)+\p_xp^\eps-\nu\Delta_\eps u^\eps)
  +\iint\p_y(\frac{g^\eps}{u_s})(K_h(h^\eps,g^\eps)-\kappa\Delta_\eps h^\eps)\\
  &-\iint\eps\p_x(\frac{v^\eps}{u_s})(S_v(u^\eps,v^\eps)+\frac{\p_y p^\eps}{\eps}-\nu\Delta_\eps v^\eps)
   -\iint\eps\p_x(\frac{g^\eps}{u_s})(K_g(h^\eps,g^\eps)-\kappa\Delta_\eps g^\eps)\\
  &\quad +\iint [\p_y(\frac{v^\eps}{u_s})S_h(h^\eps,g^\eps)-\eps\p_x(\frac{v^\eps}{u_s})S_g(h^\eps,g^\eps)]\\
  &\quad +\iint [\p_y(\frac{g^\eps}{u_s})K_u(u^\eps,v^\eps)-\eps\p_x(\frac{g^\eps}{u_s})K_v(u^\eps,v^\eps)]\\
  &=\iint \p_y(\frac{v^\eps}{u_s})\cdot f_1-\eps\p_x(\frac{v^\eps}{u_s})\cdot f_2+\p_y(\frac{g^\eps}{u_s})\cdot f_3-\eps\p_x(\frac{g^\eps}{u_s})\cdot f_4.
\end{split}
\end{align}

Firstly,  by virtue of \eqref{embedding} and $\textit{Y\"{o}ung inequality}$, using boundary conditions \eqref{u_epsboundary}, we get
\begin{align}\label{postivityestimate2}
\begin{split}
  &\iint \p_y(\frac{v^\eps}{u_s})(S_u(u^\eps,v^\eps)+\p_xp^\eps-\nu\Delta_\eps u^\eps)
  +\iint\p_y(\frac{g^\eps}{u_s})( K_h(h^\eps,g^\eps)-\kappa\Delta_\eps h^\eps)\\
  &-\iint\eps\p_x(\frac{v^\eps}{u_s})(S_v(u^\eps,v^\eps)+\frac{\p_y p^\eps}{\eps}-\nu\Delta_\eps v^\eps)
  -\iint\eps\p_x(\frac{g^\eps}{u_s})( K_g(h^\eps,g^\eps)-\kappa\Delta_\eps g^\eps)\\
  \lesssim& -\|\nabla_\eps v^\eps\|_{ L^2}^2-\|\nabla_\eps g^\eps\|_{ L^2}^2-\eps\int_{x=L}\frac{|v^\eps_y|^2}{u_s}
  +\mathcal{O}(v_e^0)\cdot\|\{ \nabla_\eps u^\eps,\nabla_\eps h^\eps\}\cdot y\|_{ L^2}^2\\
  &\quad+\|(u^\eps_y,h^\eps_y)\|_{ L^2}^2+\iint \frac{|g^\eps|^2\cdot |\p_yu_s|^2}{u_s^2}.
\end{split}
\end{align}
And the last term in \eqref{postivityestimate2} can be estimated as
\begin{equation}\label{postivityestimate5}
  \iint \frac{|g^\eps|^2\cdot |\p_yu_s|^2}{u_s^2}\leq \|y\p_yu_s\|_{L^\infty}^2\|\p_yg^\eps\|_{L^2}^2,
\end{equation}
which will be absorbed into $\|\nabla_\eps g\|_{ L^2}^2$ on the left hand side, since the coefficient $\|y\p_y(u_s,h_s)\|_{L^\infty}$ is small by using the estimate \eqref{yu_sy}.

Secondly, for the terms in the right hand side, using \textit{Y\"{o}ung inequality}, \textit{Hardy inequality} and \eqref{yu_sy} \eqref{u_s_positivity}, it is direct to obtain that
\begin{align}\label{postivityestimate3}
\begin{split}
  &\iint \p_y(\frac{v^\eps}{u_s})\cdot f_1-\eps\p_x(\frac{v^\eps}{u_s})\cdot f_2+\p_y(\frac{g^\eps}{u_s})\cdot f_3-\eps\p_x(\frac{g^\eps}{u_s})\cdot f_4\\
  \leq& (\|f_1\|_{ L^2}+\sqrt\eps\|f_2\|_{ L^2})\cdot\|\nabla_\eps v^\eps\|_{ L^2} +(\|f_3\|_{ L^2}+\sqrt\eps\|f_4\|_{ L^2})\cdot\|\nabla_\eps g^\eps\|_{ L^2}.
\end{split}
\end{align}

Finally, it remains to estimate the following terms
\begin{align}\label{positivityestimate4}
\begin{split}
  &\quad \iint [\p_y(\frac{v^\eps}{u_s})S_h(h^\eps,g^\eps)-\eps\p_x(\frac{v^\eps}{u_s})S_g(h^\eps,g^\eps)]\\
  &\quad +\iint [\p_y(\frac{g^\eps}{u_s})K_u(u^\eps,v^\eps)-\eps\p_x(\frac{g^\eps}{u_s}) K_v(u^\eps,v^\eps)]\\
  &:=K_1+K_2+K_3+K_4.
\end{split}
\end{align}
It should be noted that
\begin{align}\label{g_e^0/Y}
  \bigg\|\frac{g_e^0}{Y}\bigg\|_{L^\infty}
  =\bigg\|\frac{v_e^0}{Yu_e^0}\cdot h_e^0\bigg\|_{L^\infty}
  \leq \bigg\|\frac{v_e^0}{Y}\bigg\|_{L^\infty} \left\|\frac{h_e^0}{u_e^0}\right\|_{L^\infty}
   \ll 1,
\end{align}
where we have used equality \eqref{pre_ratio}, the hypothesis \eqref{condition_for_idealMHD1} and the smallness of $\|\frac{v_e^0}{Y}\|_{L^\infty}$.

Now we give the detail estimates for each term in $K_1$:
\begin{align*}
    \iint \p_y(\frac{v^\eps}{u_s})h_sh^\eps_x
  &=\iint(\frac{\p_yv^\eps}{u_s}-\frac{y\p_y u_s}{u_s^2}\cdot\frac{v^\eps}{y})\cdot h_sh^\eps_x\\
  &\leq \bigg\|\frac{h_s}{u_s}\bigg\|_{L^\infty}\cdot(1+\|y\p_y u_s\|_{ L^\infty})\cdot (\|v^\eps_y\|_{ L^2}^2+\|h^\eps_x\|_{ L^2}^2),\\\nonumber
    \iint \p_y(\frac{v^\eps}{u_s})h^\eps\p_xh_s
  &=\iint(\frac{\p_yv^\eps}{u_s}-\frac{y\p_y u_s}{u_s^2}\cdot\frac{v^\eps}{y})\cdot h^\eps\cdot\p_xh_s\\
  &\leq \|\p_xh_s\|_{L^\infty}(1+\|y\p_y u_s\|_{ L^\infty})\cdot L(\|v^\eps_y\|_{ L^2}^2+\|h^\eps_x\|_{ L^2}^2),\\\nonumber
    \iint \p_y(\frac{v^\eps}{u_s})g_sh^\eps_y
  &=\iint \p_y(\frac{v^\eps}{u_s})\cdot\frac{g_e^0}{\sqrt\eps}\cdot\p_y h^\eps
    +\iint \p_y(\frac{v^\eps}{u_s})\cdot(g_p^0+g_e^1+\sqrt\eps g_p^1)\cdot\p_y h^\eps\\
  &=\iint \frac{\p_yv^\eps}{u_s}\cdot\frac{g_e^0}{\sqrt\eps y}\cdot y\p_y h^\eps
    -\iint \frac{y\p_y u_s}{u_s^2}\cdot\frac{v^\eps}{y}\cdot\frac{g_e^0}{\sqrt\eps y}\cdot y\p_y h^\eps\\
  &\quad +\iint \frac{\p_yv^\eps }{u_s}(g_s-\frac{g_e^0}{\sqrt\eps})\cdot\p_y h^\eps
    -\iint \frac{y\p_y u_s}{u_s^2}\cdot\frac{v^\eps}{y}(g_s-\frac{g_e^0}{\sqrt\eps})\cdot \p_y h^\eps\\
  &\leq C\|v^\eps_y\|_{ L^2}(\bigg\| \frac{g_e^0}{Y} \bigg\|_{ L^\infty}\|yh^\eps_y\|_{ L^2}+\bigg\| g_s-\frac{g_e^0}{\sqrt\eps} \bigg\|_{ L^\infty}\|h^\eps_y\|_{ L^2})\\
  &\leq \mathcal{O}(v_e^0)(\|v^\eps_y\|_{ L^2}^2+\|yh^\eps_y\|_{ L^2}^2)+\delta\|v^\eps_y\|_{ L^2}^2+N_\delta\|h^\eps_y\|_{ L^2}^2,\\\nonumber
    \iint \p_y(\frac{v^\eps}{u_s})g^\eps\p_yh_s
  &=\iint(\frac{\p_yv^\eps}{u_s}-\frac{y\p_y u_s}{u_s^2}\cdot\frac{v^\eps}{y})\cdot g^\eps\cdot\p_yh_s\\
  &\leq \|y\p_yh_s\|_{L^\infty}(1+\|y\p_y u_s\|_{ L^\infty})\cdot (\|v^\eps_y\|_{ L^2}^2+\|g^\eps_y\|_{ L^2}^2),
\end{align*}
Hence, for arbitrary small constant $\delta$, we have
\begin{align}\label{K_1}
  K_1
  &\leq (\left\|\frac{h_s}{u_s}\right\|_{L^\infty}+L+\mathcal{O}(v_e^0)+\|y\p_yh_s\|_{ L^\infty})\cdot\|(v^\eps_y,g^\eps_y)\|_{ L^2}^2\nonumber\\
  &\quad +\delta \|v^\eps_y\|_{ L^2}^2+\mathcal{O}(v_e^0)\|h^\eps_y\cdot y\|_{ L^2}^2+N_\delta\|h^\eps_y\|_{ L^2}^2.
\end{align}
By a similar argument, the other terms $K_2 \sim K_4$ can be estimated as
\begin{align}
\label{K_2}
    K_2
  &\leq \left(\left\|\frac{h_s}{u_s}\right\|_{L^\infty}+L+\mathcal{O}(v_e^0)+\sqrt\eps\right)\cdot
  \left(\|\nabla_\eps (v^\eps,g^\eps)\|_{ L^2}^2+\|\sqrt\eps g^\eps_y\cdot y\|_{ L^2}^2\right),\\
\label{K_3}
    K_3
  &\leq \left(\left\|\frac{h_s}{u_s}\right\|_{L^\infty}+L+\mathcal{O}(v_e^0)+\|y\p_yh_s\|_{ L^\infty}\right)\cdot \|(v^\eps_y,g^\eps_y)\|_{ L^2}^2+\delta\|g^\eps_y\|_{ L^2}^2\nonumber\\
  &\quad+\mathcal{O}(v_e^0)\|u^\eps_y\cdot y\|_{ L^2}^2+N_\delta\|u^\eps_y\|_{ L^2}^2,\\
\label{K_4}
    K_4
  &\leq \left(\left\|\frac{h_s}{u_s}\right\|_{L^\infty}+L+\mathcal{O}(v_e^0)+\sqrt\eps\right)\|\nabla_\eps (v^\eps,g^\eps)\|_{ L^2}^2.
\end{align}

Therefore, putting \eqref{K_1}-\eqref{K_4} into \eqref{positivityestimate4}, chosing sufficiently small $\delta$ and combining \eqref{postivityestimate2}-\eqref{postivityestimate3}, we can obtain the estimate \eqref{positivityestimate}.
\end{proof}

\begin{lemma}\label{weightedestimatelemma}{\rm (Weighted estimates)} Consider solutions $[u^\eps,v^\eps,h^\eps,g^\eps]\in\mathcal{X}$ to linearized problem \eqref{linearizedu_s} with boundary conditions \eqref{u_epsboundary}, under the assumption that $\|\frac{h_s}{u_s}\|_{L^\infty}\ll1$, $\|y\p_y (u_s,h_s)\|_{ L^\infty}<C\sigma_0$ uniform in $0<L\ll1$ and the normal velocity enjoys $\|\frac{v_e^0}{Y}\|_{L^\infty}\ll 1$, then the following estimate holds
\begin{align}\label{weightedestimate}
\begin{split}
   &\| \{u^\eps_{yy},h^\eps_{yy},\sqrt\eps (u^\eps_{xy},h^\eps_{xy}),\eps(u^\eps_{xx},h^\eps_{xx})\}\cdot y \|_{ L^2}^2+\| \{u^\eps_y,h^\eps_y,\sqrt\eps(u^\eps_x,h^\eps_x)\}\cdot y \|_{ L^2}^2\\
   &\quad +\| \{u^\eps_y,\sqrt\eps(u^\eps_x,h^\eps_x)\}\cdot y \|_{ L^2(x=L)}^2
   \leq \|u^\eps_y,h^\eps_y\|_{ L^2}^2 +\|v^\eps_y,g^\eps_y,\sqrt\eps(v^\eps_x,g^\eps_x)\|_{ L^2}^2+\mathcal{R},
\end{split}
\end{align}
where
\begin{align*}
  \mathcal{R}:=&\iint \p_yf_1\p_y\left\{\frac{u^\eps wy^2}{u_s}\right\}
  -\iint\eps\p_yf_2\p_x\left\{\frac{u^\eps wy^2}{u_s}\right\}\\
  &+\iint\p_yf_3\p_y\left\{\frac{h^\eps wy^2}{u_s}\right\}
  -\iint\eps\p_yf_4\p_x\left\{\frac{h^\eps wy^2}{u_s}\right\}.
\end{align*}
\end{lemma}
\begin{proof}
Taking partial $y$ derivative of system \eqref{linearizedu_s}, we have
\begin{align}\label{p_ylinearizedu_s}
\begin{cases}
  -\nu\Delta_\eps u^\eps_y+\p_yS_u(u^\eps,v^\eps)-\p_yS_h(h^\eps,g^\eps)+p^\eps_{xy}=\p_yf_1,\\
  -\nu\Delta_\eps  v^\eps_y+\p_yS_v(u^\eps,v^\eps)-\p_yS_g(h,g)+\frac{p^\eps_{yy}}{\eps }=\p_yf_2,\\
  -\kappa\Delta_\eps  h^\eps_y+\p_yK_h(h^\eps,g^\eps)-\p_yK_u(u^\eps,v^\eps)=\p_yf_3,\\
  -\kappa\Delta_\eps g^\eps_y+\p_yK_g(h^\eps,g^\eps)-\p_yK_v(u^\eps,v^\eps)=\p_yf_4,
\end{cases}
\end{align}
where
\begin{align}\label{p_yS_u}
\begin{cases}
  &\p_yS_u(u^\eps,v^\eps)=u_su^\eps_{xy}+v_su^\eps_{yy}+u^\eps\p_{xy} u_s+v^\eps\p_{yy} u_s,\\
  &\p_yS_h(h^\eps,g^\eps)=h_sh^\eps_{xy}+g_sh^\eps_{yy}+h^\eps\p_{xy} h_s+g^\eps\p_{yy} h_s,\\
  &\p_yS_v(u^\eps,v^\eps)=u_sv^\eps_{xy}+v_sv^\eps_{yy}+\p_yu_sv^\eps_x+u^\eps_y\p_x v_s\\
   &\qquad\qquad\qquad\quad +u^\eps\p_{xy}v_s+2\p_yv_sv^\eps_y+v^\eps\p_{yy} v_s,\\
  &\p_yS_g(h^\eps,g^\eps)=h_sg^\eps_{xy}+g_sg^\eps_{yy}+\p_yh_sg^\eps_x+h^\eps_y\p_x g_s\\
   &\qquad\qquad\qquad\quad +h^\eps\p_{xy}g_s+2\p_yg_sg^\eps_y+g^\eps\p_{yy} g_s,\\
  &\p_yK_h(h^\eps,g^\eps)=u_sh^\eps_{xy}+v_sh^\eps_{yy}+\p_yu_sh^\eps_x-h^\eps_y\p_xu_s\\
   &\qquad\qquad\qquad\quad -h^\eps\p_{xy}u_s+\p_yv_sh^\eps_y-g^\eps_y\p_yu_s-g^\eps\p_{yy}u_s,\\
&\p_yK_u(u^\eps,v^\eps)=h_su^\eps_{xy}+g_su^\eps_{yy}+\p_yh_su^\eps_x-u^\eps_y\p_x h_s\\
   &\qquad\qquad\qquad\quad -u^\eps\p_{xy}h_s+\p_yg_su^\eps_y-v^\eps_y\p_yh_s-v^\eps\p_{yy}h_s,\\
  &\p_yK_g(h^\eps,g^\eps)=u_sg^\eps_{xy}+v_sg^\eps_{yy}+\p_yu_s\p_xg^\eps-h^\eps\p_{xy}v_s \\
   &\qquad\qquad\qquad\quad -h^\eps_y\p_xv_s+\p_yv_sg^\eps_y-g^\eps_y\p_yv_s-g^\eps\p_{yy}v_s,\\
  &\p_yK_v(u^\eps,v^\eps)=h_sv^\eps_{xy}+g_sv^\eps_{yy}+\p_yh_s\p_xv^\eps-u^\eps\p_{xy}g_s\\
   &\qquad\qquad\qquad\quad -u^\eps_y\p_xg_s+\p_yg_sv^\eps_y-v^\eps_y\p_yg_s-v^\eps\p_{yy}g_s.
\end{cases}
\end{align}

Applying operator $\big[\p_y\{\frac{u^\eps wy^2}{u_s}\},-\eps\p_x\{\frac{u^\eps wy^2}{u_s}\},\p_y\{\frac{h^\eps wy^2}{u_s}\},-\eps\p_x\{\frac{h^\eps wy^2}{u_s}\}\big]$ to system \eqref{p_ylinearizedu_s} with $w(x)=1-x$, and integrating them in $[0,L]\times[0,\infty)$. Next, we estimate each term as follows.

{\bf{Step 1}}: Positive profile terms.
At first, by virtue of \eqref{embedding} and \textit{Y\"{o}ung inequality}, it is direct to get that
\begin{align}\label{positive1.1}
\begin{split}
  &\iint (u_su^\eps_{xy}+v_su^\eps_{yy})\cdot\p_y\left\{\frac{u^\eps wy^2}{u_s}\right\}
  -(u_sv^\eps_{xy}+v_sv^\eps_{yy})\cdot\eps\p_x\left\{\frac{u^\eps wy^2}{u_s}\right\}\\
  &\qquad +(u_sh^\eps_{xy}+v_sh^\eps_{yy})\cdot\p_y\left\{\frac{h^\eps wy^2}{u_s}\right\}
  -(u_sg^\eps_{xy}+v_sg^\eps_{yy})\cdot\eps\p_x\left\{\frac{h^\eps wy^2}{u_s}\right\}\\
  &\geq\iint \frac{y^2}{2}(|u^\eps_y|^2+|h^\eps_y|^2+\eps |v^\eps_y|^2+\eps |g^\eps_y|^2)+\int_{x=L}\frac{y^2 }{2}(|u^\eps_y|^2+\eps |v^\eps_y|^2+\eps |g^\eps_y|^2)\\
  &\qquad -\|(u^\eps_x,u^\eps_y,h^\eps_x,h^\eps_y)\|_{ L^2}^2.
\end{split}
\end{align}
Second, the following terms can be rewritten and estimated as
\begin{align}\label{positive1.2}
\begin{split}
  &\iint -h_sh^\eps_{xy}\cdot\p_y\left\{\frac{u^\eps wy^2}{u_s}\right\}
  +\iint\eps h_sg^\eps_{xy}\p_x\left\{\frac{u^\eps wy^2}{u_s}\right\}\\
  &\quad -\iint h_su^\eps_{xy}\cdot\p_y\left\{\frac{h^\eps wy^2}{u_s}\right\}
  -\iint\eps h_sv^\eps_{xy}\p_x\left\{\frac{h^\eps wy^2}{u_s}\right\}\\
  &=\iint \frac{h_s}{u_s}h^\eps_y\cdot(2u^\eps_xwy-u^\eps_yy^2-2u^\eps y)-\iint h_sh^\eps_y\frac{(u_y^\eps wy^2+2u^\eps wy)\p_x u_s}{u_s^2}\\
  &\quad +\iint \frac{\p_xh_s}{u_s}h^\eps_y\cdot(u^\eps_ywy^2+2u^\eps wy)-\int_{x=L}\frac{h_s}{u_s}h^\eps_y(u^\eps_yy^2+2u^\eps y)(1-L)\\
  &\quad -\eps\iint \frac{\p_xh_s}{u_s}g^\eps_y\cdot(u^\eps_xwy^2-u^\eps y^2)
  +\eps\iint\frac{h_sg_y^\eps y^2(u_x^\eps w-u^\eps)\p_xu_s}{u_s^2} \\
  &\quad +3\eps\iint \frac{h_s}{u_s}g^\eps_yu^\eps_xy^2+\eps\int_{x=L}\frac{h_s}{u_s}g^\eps_y\cdot (u^\eps_x(1-L)-u^\eps)y^2+\iint \frac{h_s}{u_s}u^\eps_y\cdot 2 y(h_x^\eps w-h^\eps)\\
  &\quad +\iint h_su_y^\eps\frac{-h^\eps wy\p_x u_s}{u_s^2}+\iint \p_xh_su_y^\eps\frac{2h^\eps wy}{u_s}-\int_{x=L}\frac{h_s}{u_s}u_y^\eps\cdot 2h^\eps y(1-L)\\
  &\quad +\eps\iint \frac{1}{u_s}v^\eps_yh^\eps y^2\cdot(\p_xh_s-\frac{h_s\p_xu_s}{u_s})-\eps\int_{x=L}\frac{h_s}{u_s}v^\eps_yh^\eps y^2+\iint h_sh^\eps_{xy}\frac{u^\eps wy^2\p_yu_s}{u_s^2}\\
  &\quad -\eps\iint h_sg^\eps_{xy}\frac{u^\eps wy^2\p_xu_s}{u_s^2}+\iint h_su^\eps_{xy}\frac{h^\eps wy^2\p_yu_s}{u_s^2}-\eps\iint h_sv^\eps_{xy}\frac{h^\eps wy^2\p_xu_s}{u_s^2}\\
  &\leq C\|(u^\eps_x,u^\eps_y,h^\eps_x,h^\eps_y)\|_{ L^2}^2+(L+\bigg\|\frac{h_s}{u_s}\bigg\|_{L^\infty}+\delta+\sqrt\eps)\cdot
  \|\{u^\eps_y,h^\eps_y,\sqrt\eps(u^\eps_x,h^\eps_x)\}\cdot y\|_{ L^2}^2\\
  &\quad +C\|\sqrt\eps (v^\eps_x,g^\eps_x)\|_{L^2}^2,
\end{split}
\end{align}
where we have used the boundary conditions \eqref{u_epsboundary} and $\|y\p_x(u_s,h_s)\|_{L^\infty}\lesssim 1$. In a similar argument, the following estimate holds
\begin{align}\label{positive1.3}
\begin{split}
&\iint -g_sh^\eps_{yy}\p_y(\frac{u^\eps wy^2}{u_s})+\eps g_sg^\eps_{yy}\p_x(\frac{u^\eps wy^2}{u_s})
  -g_su^\eps_{yy}\p_y(\frac{h^\eps wy^2}{u_s})+\eps g_sv^\eps_{yy}\p_x(\frac{h^\eps wy^2}{u_s})\\
  &=\iint \frac{\p_yg_s}{u_s}h^\eps_y\cdot(u^\eps_ywy^2+2u^\eps wy)+\iint \frac{g_s}{u_s}h^\eps_y\cdot(6u_y^\eps wy+2u^\eps w)\\
  &\quad -\iint \frac{g_s}{u_s^2}h_y^\eps(u_y^\eps wy^2+2u^\eps wy)\cdot\p_yu_s+\eps\iint\frac{g_s}{u_s^2}g_y^\eps(2u_x^\eps wy^2-u^\eps y^2)\cdot\p_yu_s\\
  &\quad -\eps\iint \frac{\p_yg_s}{u_s}g^\eps_y\cdot(u^\eps_xwy^2-u^\eps y^2)-\eps\iint\frac{g_s}{u_s}g^\eps_y\cdot(2u^\eps_xwy-u^\eps_yy^2-2u^\eps y)\\
  &\quad +\iint\frac{\p_yg_s}{u_s}u^\eps_y\cdot 2h^\eps wy+\iint \frac{g_s}{u_s}u^\eps_y\cdot2h^\eps w-\iint\frac{g_s}{u_s^2}u_y^\eps\cdot 2h^\eps wy\cdot\p_y u_s\\
  &\quad +\eps\iint \frac{\p_yg_s}{u_s}v^\eps_yh^\eps y^2+\eps\iint \frac{g_s}{u_s}v^\eps_y(h^\eps_yy^2+2h^\eps y)-\eps\iint\frac{g_s}{u_s^2}v_y^\eps h^\eps y^2\p_yu_s\\
  &\quad +\iint g_s(u^\eps h^\eps_{yy}+h^\eps u^\eps_{yy})\frac{wy^2\p_yu_s}{u_s^2}-\eps\iint g_s(u^\eps g^\eps_{yy}+h^\eps v^\eps_{yy})\frac{wy^2\p_xu_s}{u_s^2}\\
  &\leq (\mathcal{O}(v_e^0)+L+\sqrt\eps+\delta)\|\{u^\eps_y,h^\eps_y,\sqrt\eps(u^\eps_x,h^\eps_x)\}\cdot y\|_{ L^2}^2+\|(u^\eps_y,v^\eps_y,h^\eps_y,g^\eps_y)\|_{ L^2}^2.
\end{split}
\end{align}

{\bf{Step 2}}: Remaining terms. First, by virtue of \textit{H$\ddot{o}$lder inequality}, \textit{Cauchy inequality}, the inequality \eqref{embedding} and the estimate \eqref{summary_estimates}, we have
\begin{align}\label{positive2.1}
\begin{split}
  &\iint \eps(\p_yu_sv^\eps_x+u^\eps_y\p_xv_s+u^\eps\p_{xy}v_s+2\p_yv_sv^\eps_y+v^\eps\p_{yy}v_s)\p_x(\frac{u^\eps wy^2}{u_s})\\
  &\quad +\iint (u^\eps\p_{xy}u_s+v^\eps\p_{yy}u_s)\p_y(\frac{u^\eps wy^2}{u_s})\\
  &\lesssim (L+\sqrt\eps+\delta)\|\{u^\eps_y,\sqrt\eps u^\eps_x\}\cdot y\|_{L^2}^2+N_\delta\|(u^\eps_x,\sqrt\eps v^\eps_x)\|_{L^2}^2.
\end{split}
\end{align}
Similarly, collecting the remaining terms in $\p_y(S_u,S_h,S_v,S_g,K_h,K_u,K_g,K_v)$ stated in \eqref{p_yS_u}, we can conclude the estimate as follows:
\begin{align}\label{positive2.2}
\begin{split}
  {\rm{Remaining~~~terms}}&\leq (L+\delta+\sqrt\eps)\|\{u^\eps_y,h^\eps_y,\sqrt\eps(u^\eps_x,h^\eps_x)\}\cdot y\|_{ L^2}^2\\
  &\quad +N_\delta\|u^\eps_y,v^\eps_y,h^\eps_y,g^\eps_y,\sqrt\eps\{v^\eps_x,g^\eps_x\}\|_{ L^2}^2.
\end{split}
\end{align}

{\bf{Step 3:}} Vorticity terms. For the vorticity terms, it can be rewritten as
\begin{align}\label{positive3.1}
\begin{split}
  &-\nu\iint\Delta_\eps u^\eps_y\p_y(\frac{u^\eps wy^2}{u_s})+
  \nu\eps\iint\Delta_\eps v^\eps_y\p_x(\frac{u^\eps wy^2}{u_s})\\
  &-\kappa\iint\Delta_\eps h^\eps_y\p_y(\frac{h^\eps wy^2}{u_s})+
  \kappa\eps\iint\Delta_\eps g_y\p_x(\frac{h^\eps wy^2}{u_s})\\
  =&-\nu\iint\left\{u^\eps_{yyy}+2\eps u^\eps_{xxy}+\eps v^\eps_{xyy}\right\}\cdot\p_y(\frac{u^\eps wy^2}{u_s})\\
  &-\kappa\iint\left\{h^\eps_{yyy}+2\eps h^\eps_{xxy}+\eps g^\eps_{xyy}\right\}\cdot\p_y(\frac{h^\eps wy^2}{u_s})\\
  &+\nu\eps \iint \left\{2v^\eps_{yyy}+\eps\p_x(u^\eps_{yy}+\eps v^\eps_{xy})\right\}\cdot\p_x(\frac{u^\eps wy^2}{u_s})\\
  &+\kappa \eps\iint\left\{2g^\eps_{yyy}+\eps\p_x(h^\eps_{yy}+\eps g^\eps_{xy})\right\}\cdot\p_x(\frac{h^\eps wy^2}{u_s}).
\end{split}
\end{align}

It remains to estimate each term on the right hand side, which can be bounded as follows:
\begin{align}\label{positive3.2}
  &-\iint\nu u^\eps_{yyy}\cdot\p_y(\frac{u^\eps wy^2}{u_s})-\iint\kappa h^\eps_{yyy}\cdot\p_y(\frac{h^\eps wy^2}{u_s})\nonumber\\
  \gtrsim& \iint\frac{\nu |u^\eps_{yy}|^2y^2w}{u_s}+\iint\frac{\kappa |h^\eps_{yy}|^2y^2w}{u_s}-\|u^\eps_y,h^\eps_y\|_{L^2}^2-\delta\|\{u^\eps_{yy},h^\eps_{yy}\}\cdot y\|_{L^2}^2,
\end{align}

\begin{align}\label{positive3.3}
  &-\iint2\nu\eps u^\eps_{xxy}\cdot\p_y(\frac{u^\eps wy^2}{u_s})-\iint2\kappa\eps h^\eps_{xxy}\cdot\p_y(\frac{h^\eps wy^2}{u_s})\nonumber\\
  \gtrsim& \iint\frac{2\nu |u^\eps_{xy}|^2y^2w}{u_s}+\iint\frac{\kappa |h^\eps_{xy}|^2y^2w}{u_s}-\int_{x=L}2\nu\eps u^\eps_{xy}\cdot\p_y(\frac{u^\eps wy^2}{u_s})-\|u^\eps_x,h^\eps_x\|_{L^2}^2\nonumber\\
  &-(L+\sqrt\eps)\|\{\sqrt\eps(u^\eps_{xy},h^\eps_{xy}),u^\eps_{y},h^\eps_{y}\}\cdot y\|_{L^2}^2,
\end{align}

\begin{align}\label{positive3.4}
  &-\iint\nu\eps v^\eps_{xyy}\cdot\p_y(\frac{u^\eps wy^2}{u_s})-\iint\kappa\eps g^\eps_{xyy}\cdot\p_y(\frac{h^\eps wy^2}{u_s})\nonumber\\
  \gtrsim& -\iint\frac{\nu\eps u^\eps_{xx}u^\eps_{yy}y^2w}{u_s}-\iint\frac{\kappa\eps h^\eps_{xx}h^\eps_{yy}y^2w}{u_s}-\|u^\eps_y,h^\eps_y,\sqrt\eps\{v^\eps_x,g^\eps_x
  \}\|_{L^2}^2\nonumber\\
  &-(\delta+\sqrt\eps)\|u^\eps_{yy},h^\eps_{yy},\eps\{u^\eps_{xx},h^\eps_{xx}\}\cdot y\|_{L^2}^2,
\end{align}
and
\begin{align}\label{positive3.5}
  &\iint\nu\eps \p_x(u^\eps_{yy}+\eps v^\eps_{xy})\cdot\p_x(\frac{u^\eps wy^2}{u_s})-\iint\kappa\eps \p_x(h^\eps_{yy}+\eps g^\eps_{xy})\cdot\p_x(\frac{h^\eps wy^2}{u_s})\nonumber\\
  \gtrsim& -\iint\frac{\nu\eps u^\eps_{xx}u^\eps_{yy}y^2w}{u_s}-\iint\frac{\kappa\eps h^\eps_{xx}h^\eps_{yy}y^2w}{u_s}-
  \|u^\eps_x,h^\eps_x,u^\eps_y,h^\eps_y\|_{L^2}^2\nonumber\\
  &-(\delta+\sqrt\eps)\|\{\sqrt\eps(u^\eps_x,h^\eps_x),u^\eps_{yy},h^\eps_{yy},\eps(u^\eps_{xx},h^\eps_{xx})\}\cdot y\|_{L^2}^2.
\end{align}

As for the remainder terms, we get
\begin{align}\label{positive3.6}
  &\iint2\nu\eps v^\eps_{yyy}\cdot\p_x(\frac{u^\eps wy^2}{u_s})-\iint2\kappa\eps g^\eps_{yyy}\cdot\p_x(\frac{h^\eps wy^2}{u_s})\nonumber\\
  \gtrsim& \iint2\eps |v^\eps_{yy}|^2\cdot\frac{y^2w}{u_s}+\iint2\eps |g^\eps_{yy}|^2\cdot\frac{y^2w}{u_s}-\|v^\eps_y,g^\eps_y\|_{L^2}^2\nonumber\\
  &-(\delta+\sqrt\eps+L)\|\{u^\eps_y,h^\eps_y,\sqrt\eps(u^\eps_{xy},h^\eps_{xy})\}\cdot y\|_{L^2}^2.
\end{align}
Summarizing the above estimates up and following the similar arguments as in \cite{Iyernonshear} (see Section 4 of \cite{Iyernonshear} for details), we can obtain that
\begin{align}\label{positive3.7}
  \iint&\bigg(-\nu\Delta_\eps u^\eps_y\p_y(\frac{u^\eps wy^2}{u_s})-\nu\eps\Delta_\eps v^\eps_y\p_x(\frac{u^\eps wy^2}{u_s})\nonumber\\
  &-\kappa\Delta_\eps h^\eps_y\p_y(\frac{h^\eps wy^2}{u_s})-\kappa\eps\Delta_\eps g^\eps_y\p_x(\frac{h^\eps wy^2}{u_s})\bigg)dxdy\nonumber\\
  \gtrsim& -\int_{x=L}2\nu\eps u^\eps_{xy}\p_y(\frac{u^\eps wy^2}{u_s})+\|\{u^\eps_{yy},\sqrt\eps u^\eps_{xy},\eps u^\eps_{xx}\}\cdot y\|_{L^2}^2-\mathcal{O}(RHS)\nonumber\\
  &\quad +\|\{h^\eps_{yy},\sqrt\eps h^\eps_{xy},\eps h^\eps_{xx}\}\cdot y\|_{L^2}^2-(L+\sqrt\eps+\delta)\mathcal{O}(LHS),
\end{align}
where $\mathcal{O}(RHS)$ and $\mathcal{O}(LHS)$ are the simplifications of the terms in the right and left hand side of estimate \eqref{weightedestimate}.

In addition, for the pressure terms, by integrating by parts, we get
\begin{align}\label{positive3.8}
  \iint p^\eps_{xy}\cdot\p_y(\frac{u^\eps wy^2}{u_s})-\iint p^\eps_{yy}\cdot\p_x(\frac{u^\eps wy^2}{u_s})=\int_{x=L}p^\eps_y\cdot\p_y(\frac{u^\eps wy^2}{u_s}).
\end{align}

Consequently, combining the above estimates \eqref{positive1.1}-\eqref{positive1.3}, \eqref{positive2.2}, \eqref{positive3.7} and \eqref{positive3.8} together, choosing sufficiently small $\delta>0$ and using boundary conditions in \eqref{u_epsboundary}, the desired estimate \eqref{weightedestimate} follows.
\end{proof}

\subsection{Proof of the main theorem}
To prove the main theorem, before turning back to the nonlinear system \eqref{u_epssystem}, we still need to obtain uniform estimates for the linearized system \eqref{linearizedu_s}.
\begin{lemma}\label{L^inftylemmea1} Consider solutions $[u^\eps,v^\eps,h^\eps,g^\eps]\in\mathcal{X}$ to linearized problem \eqref{linearizedu_s} with boundary conditions \eqref{u_epsboundary}, then it satisfies the following uniform estimate:
\begin{align}\label{L^inftyesimates1}
\begin{split}
  \eps^{\frac{\gamma}{4}}\|(u^\eps,h^\eps,\sqrt\eps v^\eps,\sqrt\eps g^\eps)\|_{ L^\infty(\bar\Omega)}
  \leq& C(\gamma,L)\{\|\nabla_\eps (u^\eps, h^\eps,\sqrt\eps v^\eps,\sqrt\eps g^\eps)\|_{ L^2}\\
  &+\sqrt\eps\|(f_2,f_4)\|_{ L^2} +\|(S_u,S_h,K_h,K_u)\|_{ L^2}\\
  &+\|(f_1,f_3)\|_{ L^2}+\sqrt\eps\|(S_v,S_g,K_g,K_v)\|_{ L^2}\}.
\end{split}
\end{align}
\end{lemma}
\begin{proof}
The estimate \eqref{L^inftyesimates1} follows from the fact that $(u^\eps,v^\eps)$ are solutions to Stokes problem
\begin{align*}
\begin{cases}
  &-\nu\Delta_\eps u^\eps+p^\eps_x=f_1-S_u+S_h:=\tilde f_1,\\
  &-\nu\Delta_\eps v^\eps+\frac{p^\eps_y}{\eps }=f_2-S_v+S_g:=\tilde f_2,\\
  &\p_x u^\eps+\p_y v^\eps=0,
\end{cases}
\end{align*}
while $(h^\eps,g^\eps)$ solve a Possion problem
\begin{align*}
\begin{cases}
  &-\kappa\Delta_\eps h^\eps=f_3-K_h+K_u:=\tilde f_3,\\
  &-\kappa\Delta_\eps g^\eps=f_4-K_g+K_v:=\tilde f_4,\\
  &\p_x h^\eps+\p_y g^\eps=0,
\end{cases}
\end{align*}
with boundary conditions \eqref{u_epsboundary}. Hence, applying a similar argument in Lemma 4.1 of \cite{GN17} to Stokes problem, using the standard theory for elliptic problem in \cite{Trudinger}, we can establish the existence and the desired estimate. Moreover, since the domain $\Omega$ in our consideration is local Lipschitz, it ensures that the estimate is up to boundary $L^\infty(\bar\Omega)$, see \cite{Adams} for details.
\end{proof}
Now, let us go back to the original system \eqref{u_epssystem}, and for any $(\tilde u,\tilde v,\tilde h,\tilde g)\in \mathcal{X}$, we consider the linearized system as follows
\begin{align}\label{u_epslinear}
\begin{cases}
  &-\nu\Delta_\eps u^\eps+S_u(u^\eps,v^\eps)-S_h(h^\eps,g^\eps)+p^\eps_x=R_1(\tilde u,\tilde v,\tilde h,\tilde g),\\
  &-\nu\Delta_\eps v^\eps+S_v(u^\eps,v^\eps)-S_g(h^\eps,g^\eps)+\frac{p^\eps_y}{\eps }=R_2(\tilde u,\tilde v,\tilde h,\tilde g),\\
  &-\kappa\Delta_\eps h^\eps+K_h(h^\eps,g^\eps)-K_u(u^\eps,v^\eps)=R_3(\tilde u,\tilde v,\tilde h,\tilde g),\\
  &-\kappa\Delta_\eps g^\eps+K_g(h^\eps,g^\eps)-K_v(u^\eps,v^\eps)=R_4(\tilde u,\tilde v,\tilde h,\tilde g),\\
  &\p_x u^\eps+\p_y v^\eps=\p_x h^\eps+\p_y g^\eps=0,
\end{cases}
\end{align}
where
\begin{align*}
\begin{cases}
  &R_1(\tilde u,\tilde v,\tilde h,\tilde g)=\eps^{-\frac{1}{2}-\gamma}R_{app}^u-N^u(\tilde u,\tilde v,\tilde h,\tilde g),\\
  &R_2(\tilde u,\tilde v,\tilde h,\tilde g)=\eps^{-\frac{1}{2}-\gamma}R_{app}^v-N^v(\tilde u,\tilde v,\tilde h,\tilde g),\\
  &R_3(\tilde u,\tilde v,\tilde h,\tilde g)=\eps^{-\frac{1}{2}-\gamma}R_{app}^h-N^h(\tilde u,\tilde v,\tilde h,\tilde g),\\
  &R_4(\tilde u,\tilde v,\tilde h,\tilde g)=\eps^{-\frac{1}{2}-\gamma}R_{app}^g-N^g(\tilde u,\tilde v,\tilde h,\tilde g)).
\end{cases}
\end{align*}
\begin{lemma}\label{L^inftylemmea2}For any $(u^\eps,v^\eps,h^\eps,g^\eps)\in \mathcal{X}$, we have the following estimate:
\begin{equation}\label{L^inftyesimates2}
\begin{aligned}
  &\|\nabla_\eps (u^\eps, h^\eps,\sqrt\eps v^\eps,\sqrt\eps g^\eps)\|_{ L^2}+\|(S_u,S_h,K_h,K_u)\|_{ L^2}\\
  &+\sqrt\eps\|(S_v,S_g,K_g,K_v)\|_{ L^2}+\|(R_1,R_3,\sqrt\eps R_2,\sqrt\eps R_4)(\tilde u,\tilde v,\tilde h,\tilde g)\|_{L^2}\\
  \lesssim &1+\|(u^\eps,v^\eps,h^\eps,g^\eps)\|_{X_1}+\|(\tilde u,\tilde v,\tilde h,\tilde g)\|_{\mathcal X}^2.
\end{aligned}
\end{equation}
\end{lemma}
\begin{proof}
First, using Lemma \ref{basicestimatelemma} and the definition of $X_1$-norm in \eqref{X_1norm}, we have
\begin{align}\label{L^inftyestimate2.1}
  \|\nabla_\eps (u^\eps, h^\eps,\sqrt\eps v^\eps,\sqrt\eps g^\eps)\|_{ L^2}\lesssim 1+\|(u^\eps,v^\eps,h^\eps,g^\eps)\|_{X_1}.
\end{align}

Next, it follows from estimate \eqref{R_app} and the assumption $0<\gamma<\frac{1}{4}$ that
\begin{align}\label{L^inftyestimate2.2}
  \|\eps^{-\frac{1}{2}-\gamma}\{R_{app}^u,R_{app}^h,\sqrt\eps(R_{app}^v,R_{app}^g)\}\|_{ L^2}\lesssim 1.
\end{align}

In addition, by virtue of \eqref{summary_estimates} and \eqref{L^inftyestimate2.1}, we have
\begin{align*}
  \|K_u\|_{ L^2}
  \leq& \|h_s,\p_xh_s,\frac{g_e^0}{Y},g_s-\frac{g_e^0}{\sqrt\eps},y\p_yu_s\|_{ L^\infty}\cdot\|(u^\eps_x,u^\eps_y,u^\eps_y\cdot y,v^\eps_y)\|_{ L^2}\nonumber\\
  \lesssim& 1+\|(u^\eps,v^\eps,h^\eps,g^\eps)\|_{X_1},
\end{align*}
and the following terms can be handled similarly
\begin{align}\label{L^inftyestimate2.3}
  \|(S_u,S_h,K_h,K_u)\|_{ L^2}+\sqrt\eps\|(S_v,S_g,K_g,K_v)\|_{ L^2}\lesssim 1+\|(u^\eps,v^\eps,h^\eps,g^\eps)\|_{X_1}.
\end{align}

Finally, it remains to estimate $(N^u,N^h,\sqrt\eps N^v,\sqrt\eps N^g)(\tilde u,\tilde v,\tilde h,\tilde g)$. By the definitions, we get that
\begin{align*}
    \|(N^u,N^h)(\tilde u,\tilde v,\tilde h,\tilde g)\|_{ L^2}
  \leq& \|\eps^{\frac{1}{2}+\gamma}(\tilde u\tilde u_x+\tilde v\tilde u_y-\tilde h\tilde h_x-\tilde g\tilde h_y)\|_{ L^2}\nonumber\\
  & +\|\eps^{\frac{1}{2}+\gamma}(\tilde  u\tilde  h_x+\tilde v\tilde h_y-\tilde  h\tilde u_x-\tilde  g\tilde u_y)\|_{ L^2}\nonumber\\
  \leq& \eps^{\frac{\gamma}{2}} \|\eps^{\frac{\gamma}{2}}(\tilde u,\tilde h,\sqrt\eps\tilde v,\sqrt\eps\tilde g)\|_{ L^\infty}\|(\tilde  u_x,\tilde h_x,\tilde u_y,\tilde h_y)\|_{ L^2},\nonumber\\
    \|\sqrt\eps(N^v,N^g)(\tilde u,\tilde v,\tilde h,\tilde g)\|_{ L^2}
  \leq& \|\eps^{1+\gamma}(\tilde u\tilde v_x+\tilde v\tilde v_y-\tilde h\tilde g_x-\tilde g\tilde g_y)\|_{ L^2}\nonumber\\
  & +\|\eps^{1+\gamma}(\tilde u\tilde g_x+\tilde v\tilde g_y-\tilde h\tilde v_x-\tilde g\tilde v_y)\|_{ L^2}\nonumber\\
  \leq& \eps^{\frac{\gamma}{2}+\frac{1}{2}} \|\eps^{\frac{\gamma}{2}}(\tilde u,\tilde h)\|_{ L^\infty}\|\sqrt\eps(\tilde v_x,\tilde g_x)\|_{ L^2}\nonumber\\
  &+\eps^{\frac{\gamma}{2}+\frac{1}{2}}\|\eps^{\frac{\gamma}{2}+\frac{1}{2}}(\tilde  v,\tilde  g)\|_{ L^\infty}\|(\tilde  v_y,\tilde  g_y)\|_{ L^2},
\end{align*}
which imply that
\begin{align}\label{L^inftyestimate2.4}
  \|(N^u,N^h,\sqrt\eps N^v,\sqrt\eps N^g)(\tilde u,\tilde v,\tilde h,\tilde g)\|_{ L^2}
  \lesssim \|(\tilde u,\tilde v,\tilde h,\tilde g)\|_{\mathcal X}^2+1.
\end{align}

Combining the above estimates together, we complete the proof.
\end{proof}
Therefore, the above two Lemmas can be summarized into the following uniform estimate:
\begin{proposition}\label{L^inftylemmea}
Let $[u^\eps,v^\eps,h^\eps,g^\eps]\in\mathcal{X}$ be the solutions to problem \eqref{u_epslinear} with boundary conditions \eqref{u_epsboundary}, then it holds that
\begin{align}\label{L^inftyesimates}
  &\eps^{\frac{\gamma}{2}}\|(u^\eps,h^\eps,\sqrt\eps v^\eps,\sqrt\eps g^\eps)\|_{ L^\infty}\nonumber\\
  \leq&\eps^{\frac{\gamma}{4}}[1+\|(u^\eps,v^\eps,h^\eps,g^\eps)\|_{X_1}+\|(\tilde u,\tilde v,\tilde h,\tilde g)\|_{\mathcal X}^2].
\end{align}
\end{proposition}

Together with Proposition \ref{prop3.1}, we can conclude that
\begin{corollary}\label{X-normlemmea}
Let $[u^\eps,v^\eps,h^\eps,g^\eps]\in\mathcal{X}$ be the solutions to problem \eqref{u_epslinear} with boundary conditions \eqref{u_epsboundary}, then there holds that
\begin{align}\label{X-normesimates}
  \|(u^\eps,v^\eps,h^\eps,g^\eps)\|_{ \mathcal{X}}^2
  \leq \eps^{\frac{\gamma}{2}}+\mathcal{R}+\eps^{\frac{\gamma}{2}}\|(\tilde u,\tilde v,\tilde h,\tilde g)\|_{\mathcal X}^4,
\end{align}
where
\begin{align*}
  \mathcal{R}:=&\iint \bigg(\p_yR_1\p_y\left\{\frac{u^\eps wy^2}{u_s}\right\}
  -\eps\p_yR_2\p_x\left\{\frac{u^\eps wy^2}{u_s}\right\}\\
  &+\p_yR_3\p_y\left\{\frac{h^\eps wy^2}{u_s}\right\}
  -\eps\p_yR_4\p_x\left\{\frac{h^\eps wy^2}{u_s}\right\}\bigg),
\end{align*}
with $w(x)=1-x$.
\end{corollary}

It remains to handle the term $\mathcal{R}$, which will be achieved by the following two lemmas: the first lemma is devoted to the nonlinear terms $(N^u,N^v,N^h,N^g)$, while the second is for the remainder terms from approximate solutions $(R^u_{app},R^v_{app},R^h_{app},R^g_{app})$.
\begin{lemma}\label{p_yN^ulemmea}Consider solutions $[u^\eps,v^\eps,h^\eps,g^\eps],[\tilde u,\tilde v,\tilde h,\tilde g]\in\mathcal{X}$ with boundary conditions \eqref{u_epsboundary}, then the following estimate holds:
\begin{align}\label{p_yN^uesimates}
\begin{split}
  \iint& \p_yN^u\p_y\left\{\frac{u^\eps wy^2}{u_s}\right\}
  -\iint \eps\p_yN^v\p_x\left\{\frac{u^\eps wy^2}{u_s}\right\}\\
  & +\iint \p_yN^h\p_y\left\{\frac{h^\eps wy^2}{u_s}\right\}
  -\iint \eps\p_yN^g\p_x\left\{\frac{h^\eps wy^2}{u_s}\right\}\\
  \lesssim&~~ \eps^{\frac{\gamma}{2}}\|(u^\eps,v^\eps,h^\eps,g^\eps)\|_{\mathcal X}^2
  +\eps^{\frac{\gamma}{2}}\|(\tilde u,\tilde v,\tilde h,\tilde g)\|_{\mathcal X}^4.
\end{split}
\end{align}
\end{lemma}
\begin{proof}
  According to the definition of $(N^u,N^v,N^h,N^g)$, we have
\begin{align*}
  \p_yN^u(\tilde u,\tilde v,\tilde h,\tilde g)=&\eps^{\frac{1}{2}+\gamma}(\tilde u\tilde u_{xy}+\tilde v\tilde u_{yy}-\tilde h\tilde h_{xy}-\tilde g\tilde h_{yy}),\\
  \eps\p_yN^v(\tilde u,\tilde v,\tilde h,\tilde g)=&\eps^{\frac{3}{2}+\gamma}( \tilde u_y\tilde v_{x} +\tilde u\tilde v_{xy}+\tilde v_y^2+\tilde v\tilde v_{yy}- \tilde h_y\tilde g_x -\tilde h\tilde g_{xy}-\tilde g_y^2-\tilde g\tilde g_{yy}),\\
  \p_yN^h(\tilde u,\tilde v,\tilde h,\tilde g)=&\eps^{\frac{1}{2}+\gamma}(\tilde u_y\tilde h_x+\tilde u\tilde h_{xy}+\tilde v_y\tilde h_y+\tilde v\tilde h_{yy}-\tilde h_y\tilde u_x-\tilde h\tilde u_{xy}-\tilde g_y\tilde u_y-\tilde g\tilde u_{yy}),\\
  \eps\p_yN^g(\tilde u,\tilde v,\tilde h,\tilde g)=&\eps^{\frac{3}{2}+\gamma}(\tilde u_y\tilde g_x+\tilde u\tilde g_{xy}+\tilde v_y\tilde g_y+\tilde v\tilde g_{yy}-\tilde h_y\tilde v_x-\tilde h\tilde v_{xy}-\tilde g_y\tilde v_y-\tilde g\tilde v_{yy}).
\end{align*}

Firstly, applying a similar argument stated in \cite{Iyernonshear}, we can achieve that
\begin{align}\label{p_yN^u1}
  &\iint \p_yN^u\p_y\left\{\frac{u^\eps wy^2}{u_s}\right\}-\iint\eps\p_yN^v\p_x\left\{\frac{u^\eps wy^2}{u_s}\right\}\nonumber\\
  \leq &\eps^{\frac{\gamma}{2}}\|(\tilde u,\tilde v,\tilde h,\tilde g)\|_{\mathcal X}^2\|(u^\eps,v^\eps,h^\eps,g^\eps)\|_{\mathcal X}.
\end{align}

Secondly, we are now concerned with $\p_yN^h$, for the first and the third term, by virtue of \textit{H\"{o}lder inequality} and \eqref{embedding}, it yields that
\begin{align}\label{p_yN^h1}
  &\iint \eps^{\frac{1}{2}+\gamma}(\tilde u_y\tilde h_x)\p_y(\frac{h^\eps wy^2}{u_s})
   +\iint \eps^{\frac{1}{2}+\gamma}(\tilde v_y\tilde h_y)\p_y(\frac{h^\eps wy^2}{u_s})\nonumber\\
  =&\iint \eps^{\frac{1}{2}+\gamma}(\tilde u_y\tilde h_x-\tilde u_x\tilde h_y)w\cdot(\frac{h^\eps_yy^2}{u_s}+\frac{2h^\eps y}{u_s}
   -\frac{h^\eps y^2\p_yu_s}{u_s^2})\nonumber\\
  =&-\iint \eps^{\frac{1}{2}+\gamma}\frac{\tilde hy^2}{u_s}(\tilde u_{xy}\cdot h^\eps_yw+\tilde u_y\cdot h^\eps_{xy}w-\tilde u_y\cdot h^\eps_y-\frac{\tilde u_y\cdot h^\eps_yw\p_xu_s}{u_s})\nonumber\\
  &+\iint \eps^{\frac{1}{2}+\gamma}\frac{\tilde uy^2}{u_s}(\tilde h_{xy} wh^\eps_y-\tilde h_y h^\eps_y+\tilde h_y wh^\eps_{xy}-\frac{\tilde h_y wh^\eps_y\p_xu^\eps_s}{u_s})\nonumber\\
  &+\iint \eps^{\frac{1}{2}+\gamma}(\tilde u_y\tilde h_x-\tilde u_x\tilde h_y)w \cdot(\frac{2h^\eps y}{u_s}-\frac{h^\eps y^2\p_yu_s}{u_s^2})\nonumber\\
  \leq& \eps^{\frac{\gamma}{2}}\|\eps^{\frac{\gamma}{2}}\tilde h\|_{L^\infty}\bigg[\|\sqrt\eps\tilde u_{xy}\cdot y\|_{L^2}\|h^\eps_y\cdot y\|_{L^2}+\|\tilde u_y\cdot y\|_{L^2}\|\sqrt\eps h^\eps_{xy}\cdot y\|_{L^2}\nonumber\\
  &+\|\tilde u_y\cdot y\|_{L^2}\|h^\eps_y\cdot y\|_{L^2}\bigg]+\eps^{\frac{\gamma}{2}}\|\eps^{\frac{\gamma}{2}} h^\eps\|_{L^\infty}\|\sqrt\eps\tilde u_{xy}\cdot y\|_{L^2}\|\tilde g_y \|_{L^2}\nonumber\\
  &+\eps^{\frac{\gamma}{2}}\|\eps^{\frac{\gamma}{2}}\tilde u\|_{L^\infty}\bigg[(1+2L)\|\sqrt\eps\tilde h_{xy}\cdot y\|_{L^2}\|h^\eps_y\cdot y\|_{L^2}+\|\tilde h_y\cdot y\|_{L^2}\|\sqrt\eps h^\eps_{xy}\cdot y\|_{L^2}\bigg]\nonumber\\
  &+\eps^{\frac{\gamma}{2}}\|\eps^{\frac{\gamma}{2}}h^\eps\|_{L^\infty}\|\tilde u_x\|_{L^2}\|\sqrt\eps\tilde h^\eps_{xy}\cdot y\|_{L^2}\nonumber\\
  \leq& \eps^{\frac{\gamma}{2}}\|(\tilde u,\tilde v,\tilde h,\tilde g)\|_{\mathcal X}^2\|(u^\eps,v^\eps,h^\eps,g^\eps)\|_{\mathcal X}.
\end{align}

For the second and fourth term in $\p_yN^h$, we get
\begin{align}\label{p_yN^h2}
  &\iint \eps^{\frac{1}{2}+\gamma}(\tilde u\tilde h_{xy})\p_y(\frac{h^\eps wy^2}{u_s})+ \iint \eps^{\frac{1}{2}+\gamma}(\tilde v\tilde h_{yy})\p_y(\frac{h^\eps wy^2}{u_s})\nonumber\\
  =& \iint \eps^{\frac{1}{2}+\gamma}(\tilde u\tilde h_{xy}+\tilde v\tilde h_{yy})\cdot \frac{1}{u_s}(h^\eps_yy^2w+ 2h^\eps yw-\frac{h^\eps y^2w\p_yu_s}{u_s})\nonumber\\
  \leq&  \eps^{\frac{\gamma}{2}}\|\eps^{\frac{\gamma}{2}}\tilde u\|_{L^\infty}(\|\sqrt\eps\tilde h_{xy}\cdot y\|_{L^2}\|h^\eps_y\cdot y\|_{L^2}+L \|\sqrt\eps\tilde h_{xy}\cdot y\|_{L^2}\|g^\eps_y\|_{L^2})\nonumber\\
  &+ \eps^{\frac{\gamma}{2}}\|\eps^{\frac{\gamma}{2}+\frac{1}{2}}\tilde v\|_{L^\infty}(\|\tilde h_{yy}\cdot y\|_{L^2}\|h^\eps_y\cdot y\|_{L^2}+L \|\tilde h_{yy}\cdot y\|_{L^2}\|g^\eps_y\|_{L^2})\nonumber\\
  \leq& \eps^{\frac{\gamma}{2}}\|(\tilde u,\tilde v,\tilde h,\tilde g)\|_{\mathcal X}^2\|(u^\eps,v^\eps,h^\eps,g^\eps)\|_{\mathcal X}.
\end{align}
Taking the above four terms as examples, the last four terms in $\p_yN^h$ can be handled similarly.

Finally, we turn to the terms in $\p_yN^g$. For the first term
\begin{align}\label{p_yN^g1}
  &\iint\eps^{\frac{3}{2}+\gamma}(\tilde u_y\tilde g_x)\p_x(\frac{h^\eps wy^2}{u_s})\nonumber\\
  =&\iint\eps^{\frac{3}{2}+\gamma}\tilde u_y\tilde g_xy^2(\frac{h^\eps_xw}{u_s}-\frac{h^\eps}{u_s}-\frac{h^\eps w\p_xu_s}{u_s^2})\nonumber\\
  =&-\eps^{\frac{3}{2}+\gamma}\iint\frac{\tilde gy^2}{u_s}(\tilde u_{xy}h^\eps_xw+\tilde u_y h^\eps_{xx}w-\tilde u_y h^\eps_x-\frac{\tilde u_y h^\eps_xw\p_xu_s}{u_s})+\frac{\tilde u_y\tilde g_xy^2 h^\eps w\p_xu_s}{u_s^2}\nonumber\\
  &+\int_{x=L}\eps^{\frac{3}{2}+\gamma}\frac{\tilde u_y\tilde gh^\eps_xy^2(1-L)}{u_s}+\iint\eps^{\frac{3}{2}+\gamma}\frac{\tilde gy^2}{u_s}(\tilde u_{xy}h^\eps+\tilde u_yh^\eps_x-\frac{\tilde u_yh^\eps\p_xu_s}{u_s})\nonumber\\
  \leq& \eps^{\frac{\gamma}{2}}\|\eps^{\frac{\gamma}{2}+\frac{1}{2}}\tilde g\|_{L^\infty}\bigg[(1+L)\|\sqrt\eps\tilde u_{xy}\cdot y\|_{L^2}\|\sqrt\eps h^\eps_x\cdot y\|_{L^2}+\|\tilde u_y\cdot y\|_{L^2}\|\eps h^\eps_{xx}\cdot y\|_{L^2}\nonumber\\
  &+(3+L)\|\tilde u_y\cdot y\|_{L^2}\|\sqrt\eps h^\eps_x\cdot y\|_{L^2}+\eps^{\frac{1}{2}}\|\tilde u_y\cdot y\|_{L^2(x=L)}\|\sqrt\eps h^\eps_x\cdot y\|_{L^2(x=L)}\bigg]\nonumber\\
  \leq& \eps^{\frac{\gamma}{2}}\|(\tilde u,\tilde v,\tilde h,\tilde g)\|_{\mathcal X}^2\|(u^\eps,v^\eps,h^\eps,g^\eps)\|_{\mathcal X},
\end{align}
where we have used \eqref{L^inftyesimates1} to control $\tilde g$ on the boundary $\{x=L\}$.

The third term can be handled in a similar way
\begin{align}\label{p_yN^g2}
  &\iint\eps^{\frac{3}{2}+\gamma}(\tilde v_y\tilde g_y)\p_x(h^\eps wy^2)\nonumber\\
  =&\iint\eps^{\frac{3}{2}+\gamma}(\tilde v_y\tilde g_y)y^2(\frac{h^\eps_xw}{u_s}-\frac{h^\eps}{u_s}-\frac{h^\eps w\p_xu_s}{u_s^2})\nonumber\\
  =&-\iint\eps^{\frac{3}{2}+\gamma}\frac{\tilde v}{u_s}(\tilde g_{yy}h^\eps_xwy^2+\tilde g_yh^\eps_{xy}wy^2+2\tilde g_yh^\eps_xwy-\frac{\tilde g_yh^\eps_xwy^2\p_yu_s}{u_s})\nonumber\\
  &-\iint\eps^{\frac{3}{2}+\gamma}\tilde v_y\tilde g_yy^2(\frac{h^\eps}{u_s}+\frac{h^\eps w\p_xu_s}{u_s^2})\nonumber\\
  \leq& \eps^{\frac{\gamma}{2}}\|\eps^{\frac{\gamma}{2}+\frac{1}{2}}\tilde v\|_{L^\infty}\bigg[\|\sqrt\eps\tilde h_{xy}\cdot y\|_{L^2}\|\sqrt\eps h^\eps_x\cdot y\|_{L^2}+\|\sqrt\eps\tilde h_x\cdot y\|_{L^2}\|\sqrt\eps h^\eps_{xy}\cdot y\|_{L^2}\nonumber\\
  &+\eps^{\frac{1}{2}}\|\tilde g_y\|_{L^2}\|\sqrt\eps h^\eps_x\cdot y\|_{L^2}\bigg]+ \eps^{\frac{\gamma}{2}+\frac{1}{2}}\|\eps^{\frac{\gamma}{2}}h^\eps\|_{L^\infty}\|\sqrt\eps\tilde u_x\cdot y\|_{L^2}\|\sqrt\eps\tilde h_x\cdot y\|_{L^2}\nonumber\\
  \leq& \eps^{\frac{\gamma}{2}}\|(\tilde u,\tilde v,\tilde h,\tilde g)\|_{\mathcal X}^2\|(u^\eps,v^\eps,h^\eps,g^\eps)\|_{\mathcal X},
\end{align}

For the second and fourth terms, we have
\begin{align}\label{p_yN^g3}
  &\iint\eps^{\frac{3}{2}+\gamma}(\tilde u\tilde g_{xy})\p_x(\frac{h^\eps wy^2}{u_s})+ \iint \eps^{\frac{3}{2}+\gamma}(\tilde v\tilde g_{yy})\p_x(\frac{h^\eps wy^2}{u_s})\nonumber\\
  =& \iint\eps^{\frac{3}{2}+\gamma}(\tilde u\tilde g_{xy}+\tilde v\tilde g_{yy})\cdot
  y^2(\frac{h^\eps_xw}{u_s}-\frac{h^\eps}{u_s}-\frac{h^\eps w\p_xu_s}{u_s^2})\nonumber\\
  \leq&  \eps^{\frac{\gamma}{2}}(1+2L)\|\sqrt\eps h^\eps_x\cdot y\|_{L^2}\bigg(\|\eps^{\frac{\gamma}{2}}\tilde u\|_{L^\infty}\|\eps\tilde h_{xx}\cdot y\|_{L^2}+\|\eps^{\frac{\gamma}{2}+\frac{1}{2}}\tilde v\|_{L^\infty}\|\sqrt\eps\tilde h_{xy}\cdot y\|_{L^2}\bigg)\nonumber\\
  \leq& \eps^{\frac{\gamma}{2}}\|(\tilde u,\tilde v,\tilde h,\tilde g)\|_{\mathcal X}^2\|(u^\eps,v^\eps,h^\eps,g^\eps)\|_{\mathcal X}.
\end{align}
With the above four inequalities in hand,  the rest terms in $\p_yN^g$ can be handled similarly. The proof is completed.
\end{proof}

\begin{lemma}\label{p_yR_applemmea}Consider solutions $[u^\eps,v^\eps,h^\eps,g^\eps]\in\mathcal{X}$ with boundary conditions \eqref{u_epsboundary}, then the following estimate holds:
\begin{align}\label{p_yR_appesimate}
\begin{split}
  &\iint \p_yR^u_{app}\p_y\left\{\frac{u^\eps wy^2}{u_s}\right\}
  -\eps\iint\p_yR^v_{app}\p_x\left\{\frac{u^\eps wy^2}{u_s}\right\}\\
  &+\iint\p_yR^h_{app}\p_y\left\{\frac{h^\eps wy^2}{u_s}\right\}
  -\eps\iint\p_yR^g_{app}\p_x\left\{\frac{h^\eps wy^2}{u_s}\right\}\\
  \lesssim& \eps^{\frac{3}{4}}\|(u^\eps,v^\eps,h^\eps,g^\eps)\|_{\mathcal X}.
\end{split}
\end{align}
\end{lemma}
\begin{proof}
Recalling the estimate \eqref{R_app}, we have
\begin{align}\label{yp_yR_app}
  \|\y\p_y\{R_{app}^u,R_{app}^h,\sqrt\eps(R_{app}^v,R_{app}^g)\}\|_{L^2}
  \lesssim \eps^{\frac{3}{4}}.
\end{align}

For the tangential components, using \textit{H\"{o}lder inequality} and \eqref{yp_yR_app}, we get
\begin{align}\label{p_yR^u_app}
\begin{split}
  \iint &\p_yR^u_{app}\p_y\left\{\frac{u^\eps wy^2}{u_s}\right\}+\iint\p_yR^h_{app}\p_y\left\{\frac{h^\eps wy^2}{u_s}\right\}\\
    &\leq\|\y\p_yR_{app}^u\|_{L^2}(\|u^\eps_y\cdot y\|_{L^2}+L\|u^\eps_x\|_{L^2})\\
      &\quad +\|\y\p_yR_{app}^h\|_{L^2}(\|h^\eps_y\cdot y\|_{L^2}+L\|h^\eps_x\|_{L^2})\\
  &\lesssim\eps^{\frac{3}{4}}\|(u^\eps,v^\eps,h^\eps,g^\eps)\|_{\mathcal{X}}.
\end{split}
\end{align}

Similarly, for the normal components, it gives
\begin{align}\label{p_yR^v_app}
\begin{split}
  \eps\iint& \p_yR^v_{app}\p_x\left\{\frac{u^\eps wy^2}{u_s}\right\}+\eps\iint\p_yR^g_{app}\p_x\left\{\frac{h^\eps wy^2}{u_s}\right\}\\
  &\lesssim\|\y\sqrt\eps\p_yR_{app}^v\|_{L^2}\|\sqrt\eps u^\eps_x\cdot y\|_{L^2}+\|\y\sqrt\eps\p_yR_{app}^g\|_{L^2}\|\sqrt\eps h^\eps_x\cdot y\|_{L^2}\\
  &\lesssim\eps^{\frac{3}{4}}\|(u^\eps,v^\eps,h^\eps,g^\eps)\|_{\mathcal{X}}.
\end{split}
\end{align}
Then the above two inequalities bring us the desired estimate \eqref{p_yR_appesimate}.
\end{proof}

Therefore, collecting the estimates in Lemma \ref{p_yN^ulemmea} and Lemma \ref{p_yR_applemmea}, using the \textit{Young's inequality}, we have
\begin{corollary}\label{Rlemma}
For $\mathcal{R}$ defined in Corollary \ref{X-normlemmea}, the following estimate holds
\begin{equation}\label{Restimate}
  \mathcal{R}\lesssim  \eps^{\frac{1}{4}-\gamma}+\eps^{\gamma_0}\|(u^\eps,v^\eps,h^\eps,g^\eps)\|_{\mathcal X}^2+\eps^{\frac{\gamma}{2}}\|(\tilde u,\tilde v,\tilde h,\tilde g)\|_{\mathcal X}^4,
\end{equation}
where $\gamma_0=\min\{\frac{1}{4}-\gamma,\frac{\gamma}{2}\}$.
\end{corollary}

Substituting \eqref{Restimate} into \eqref{X-normesimates}, and absorbing the term $\|(u^\eps,v^\eps,h^\eps,g^\eps)\|_{\mathcal X}$ into the left hand side, we have
\begin{theorem}\label{th3.1}
Consider solutions $[u^\eps,v^\eps,h^\eps,g^\eps]\in\mathcal{X}$ to system \eqref{u_epslinear} with boundary conditions \eqref{u_epsboundary}, then it satisfy the following estimate:
\begin{align}\label{X-estimate}
  \|(u^\eps,v^\eps,h^\eps,g^\eps)\|_{\mathcal X}^2 \lesssim \eps^{\gamma_0}+\eps^{\frac{\gamma}{2}}\|(\tilde u,\tilde v,\tilde h,\tilde g)\|_{\mathcal X}^4.
\end{align}
\end{theorem}
\begin{proof}[Proof of Theorem \ref{th1.1}]
The estimate \eqref{X-estimate} shows that for any sufficiently small $\eps$, there exists a operator $A:[\tilde u,\tilde v,\tilde h,\tilde g]\mapsto [u^\eps,v^\eps,h^\eps,g^\eps]$ for system \eqref{u_epslinear}, mapping an ball $B:=\{\|(u^\eps,v^\eps,h^\eps,g^\eps)\|_{\mathcal X}\leq 4C(u_s,v_s,h_s,g_s):=K\}$ into itself. In addition, for any solutions $[\tilde u_1,\tilde v_1,\tilde h_1,\tilde g _1]$ $[\tilde u_2,\tilde v_2,\tilde h_2,\tilde g_2]\in \mathcal{X}$ to system \eqref{u_epslinear}, we have
\begin{align}\label{contractionmap}
\begin{split}
  &\|(u^\eps_1-u^\eps_2,v^\eps_1-v^\eps_2,h^\eps_1-h^\eps_2,g^\eps_1-g^\eps_2)\|_{\mathcal X} \\
  &\lesssim 2\eps^{\frac{\gamma}{2}}KC(L,u_s,v_s,h_s,g_s)\|(\tilde u_1-\tilde u_2,\tilde v_1-\tilde v_2,\tilde h_1-\tilde h_2,\tilde g_1-\tilde g_2)\|_{\mathcal X}.
\end{split}
\end{align}
By virtue of contraction fixed-point theorem and the argument stated in \cite{Iyernonshear} Appendix B, we can obtain the existence of solution in $\mathcal{X}$ to system \eqref{u_epssystem}, and thus the main result Theorem \ref{th1.1} is proved.
\end{proof}
\appendix
\renewcommand{\appendixname}{Appendix~\Alph{section}}
\section{Leading order ideal MHD layer}\label{ap1}
In this appendix, we are going to presribe the leading order ideal MHD profile $[u_e^0,v_e^0,h_e^0,g_e^0,p_e^0]$ by verifying that there exist nonshear flows to 2D steady incompressible ideal MHD equations \eqref{u_e^0} satisfying the assumptions \eqref{condition_for_idealMHD1}-\eqref{M_0}.
\begin{proposition}
There exist ideal MHD flows $[u_e^0,v_e^0,h_e^0,g_e^0,p_e^0](X,Y)$ satisfying assumptions \eqref{condition_for_idealMHD1}-\eqref{M_0}.
\end{proposition}
\begin{proof}
We first prescribe the shear flows $[U_0(Y),0,H_0(Y),0]$ to 2D steady incompressible ideal MHD equations \eqref{u_e^0} constructed as in \cite{DLX} satisfying the assumptions as follows:
 \begin{align}
\label{aa1}
  &0< c_0\leq  H_0(Y)\ll U_0(Y) \leq C_0<\infty,\\
\label{aa2}
  &U_0,H_0~ \mathrm{smooth, with\ rapidly\ decaying\ derivatives},\\
\label{aa5}
  &\|\langle Y\rangle\p_Y (U_0,H_0)\|_{ L^\infty}<\delta_0,\ \mathrm{for\  suitably \ small \ }  \delta_0>0,\\
\label{aa6}
  &\|Y^k \p_Y^m (U_0,H_0)\|_{ L^\infty}<\infty,\ \mathrm{for\  sufficiently \ large \ }  k\geq0, m\geq1.
\end{align}
Such shear flows have stream functions $\phi_0(Y)=\int_0^YU_0(z){\rm d}z$ and $\psi_0(Y)=\int_0^YH_0(z){\rm d}z$ which enjoy the following asymptotics:
\begin{align}\label{aa5}
\begin{split}
  \phi_0|_{Y=0}=0,\quad \phi_0|_{X=0}=\phi_0|_{X=L}=\phi_0(Y),\quad \lim_{Y\rightarrow\infty}\frac{\phi_0}{Y}=U_\infty\in(c_0,C_0),\\
  \psi_0|_{Y=0}=0,\quad \psi_0|_{X=0}=\psi_0|_{X=L}=\psi_0(Y),\quad \lim_{Y\rightarrow\infty}\frac{\psi_0}{Y}=H_\infty\in(c_0,C_0).
\end{split}
\end{align}

The second step is devoted to constructing an nonshear solution $(\tilde u_e^0,\tilde v_e^0,\tilde h_e^0,\tilde g_e^0,\tilde p_e^0)$ to the ideal MHD equations \eqref{u_e^0}. Above all, we can rewrite the third and fourth equation of \eqref{u_e^0} as
\begin{align*}
  \nabla_{X,Y}(v_e^0h_e^0 -g_e^0u_e^0)=0,
\end{align*}
in which we have used the divergence free conditions have been used. So there exists a constant $b$, such that
\begin{equation*}
  v_e^0h_e^0 -g_e^0u_e^0=b,
\end{equation*}
Take $Y=0$, by virtue of the zero boundary conditions for $(v_e^0,g_e^0)$, we have $b=0$, that is to say,
\begin{equation}\label{pre_ratio}
  v_e^0h_e^0 -g_e^0u_e^0=0.
\end{equation}
It implies that there exists a scalar function $k_1(X,Y)$ satisfying
\begin{equation}\label{appratio}
  h_e^0=k_1(X,Y)u_e^0,\quad g_e^0=k_1(X,Y)v_e^0.
\end{equation}

There is a classical observation that any functions $\phi(X,Y),\psi(X,Y)$ satisfying
\begin{align}\label{laplace}
  -\Delta\phi=\tilde f(\phi),\quad-\Delta\psi=\tilde g(\psi),\quad \nabla\psi=k_1\nabla\phi,
\end{align}
with the boundary conditions
\begin{align}\label{streambc}
\begin{split}
  &\phi(0,Y)=A_0(Y),\quad\phi(L,Y)=A_L(Y),\quad\phi(x,0)=0,\\
  &\psi(0,Y)=B_0(Y),\quad\psi(L,Y)=B_L(Y),\quad\psi(x,0)=0,\\
  &\frac{\phi(X,Y)}{Y}\xrightarrow{Y\rightarrow\infty} U_\infty,
   \quad\frac{\psi(X,Y)}{Y}\xrightarrow{Y\rightarrow\infty} H_\infty.
\end{split}
\end{align}
produce solutions to the ideal MHD equations \eqref{u_e^0} by setting
\begin{align}\label{solution}
\begin{split}
  &(\tilde u_e^0,\tilde v_e^0)=(\p_Y\phi,-\p_X\phi),\quad (\tilde h_e^0,\tilde g_e^0)=(\p_Y\psi,-\p_X\psi),\\
  &\tilde p_e^0=-\frac{1}{2}|\nabla\phi|^2+\frac{1}{2}|\nabla\psi|^2-\tilde F(\phi)+\tilde G(\psi) \ \mathrm{with} \ \tilde F'=\tilde f,\ \tilde G'=\tilde g.
\end{split}
\end{align}
In fact, according to the setting of $(\tilde u_e^0,\tilde v_e^0,\tilde h_e^0,\tilde g_e^0,\tilde p_e^0)$ in \eqref{solution} and the system \eqref{laplace}, we have
\begin{align*}
\begin{split}
  &\nabla(-\frac{1}{2}|\nabla\phi|^2)=\nabla(-\frac{1}{2}|\mathbf{U}|^2)
   =\begin{pmatrix} -\tilde u_e^0\p_X\tilde u_e^0-\tilde v_e^0\p_X\tilde v_e^0\\ -\tilde u_e^0\p_Y\tilde u_e^0-\tilde v_e^0\p_Y\tilde v_e^0  \end{pmatrix},\\
\end{split}
\end{align*}
in which $\mathbf{U}:=(\tilde u_e^0,\tilde v_e^0)$. And
\begin{align*}
\begin{split}
  &\nabla(-\tilde F(\phi))=-\tilde F'\cdot\nabla\phi=-\tilde f\cdot\begin{pmatrix} -\tilde v_e^0 \\ \tilde u_e^0\end{pmatrix}
   =\begin{pmatrix} \tilde v_e^0\p_X\tilde v_e^0-\tilde v_e^0\p_Y\tilde u_e^0\\ -\tilde u_e^0\p_X\tilde v_e^0+\tilde u_e^0\p_Y\tilde u_e^0  \end{pmatrix}.
\end{split}
\end{align*}
Thus, it follows that
\begin{align*}
  \nabla(-\frac{1}{2}|\nabla\phi|^2- \tilde F(\phi))
  =\begin{pmatrix} -\tilde u_e^0\p_X\tilde u_e^0-\tilde v_e^0\p_X\tilde u_e^0\\ -\tilde u_e^0\p_Y\tilde v_e^0-\tilde v_e^0\p_Y\tilde v_e^0  \end{pmatrix}.
\end{align*}
Operating a similar computation, it gives
\begin{align*}
  \nabla(\frac{1}{2}|\nabla\psi|^2+\tilde G(\phi))
  =\begin{pmatrix} \tilde h_e^0\p_X\tilde h_e^0+\tilde g_e^0\p_X\tilde h_e^0\\ \tilde h_e^0\p_Y\tilde g_e^0+\tilde g_e^0\p_Y\tilde g_e^0  \end{pmatrix},
\end{align*}
Combining the above two equalities, we can achieve that
\begin{align*}
  \nabla\tilde p_e^0
  =\begin{pmatrix} -\tilde u_e^0\p_X\tilde u_e^0-\tilde v_e^0\p_X\tilde u_e^0+\tilde h_e^0\p_X\tilde h_e^0+\tilde g_e^0\p_X\tilde h_e^0
  \\ -\tilde u_e^0\p_Y\tilde v_e^0-\tilde v_e^0\p_Y\tilde v_e^0+\tilde h_e^0\p_Y\tilde g_e^0+\tilde g_e^0\p_Y\tilde g_e^0  \end{pmatrix},
\end{align*}
which implies that any function $\phi,\psi$ satisfying \eqref{laplace}-\eqref{solution} produce solutions to system \eqref{u_e^0}.

Furthermore, to produce nonshear flows $(\tilde u_e^0,\tilde v_e^0,\tilde h_e^0,\tilde g_e^0,\tilde p_e^0)$ for the ideal MHD equations satisfying the assumptions \eqref{condition_for_idealMHD3}-\eqref{M_0}, we will assume the following conditions on $\tilde f,\tilde g$ and the boundary data $A_0,A_L,B_0,B_L$:
\begin{align*}
\begin{cases}
  &0\leq \tilde f,\tilde g\leq \delta\ll1,\\
  &|\p^k\tilde f(x+a)|\leq|\p^k\tilde f(x)|\ \mathrm{for}\ any\ a\geq 0,\\
  &|\p^k\tilde g(x+a)|\leq|\p^k\tilde g(x)|\ \mathrm{for}\ any\ a\geq 0,\\
  &\tilde f,\tilde g\in C^\infty(\mathbb{R}), \mathrm{rapidly\ decaying\ in\ its\ argument},\\
  &\tilde f,\tilde g,A_0,A_L,B_0,B_L \ \mathrm{supported\ in\ a\ neighborhood\ away\ from\ 0},\\
  &0\leq A_0,A_L,B_0,B_L\leq\delta\times L^{10},\\
  &|\p_Y^k\{A_0,A_L,B_0,B_L\}|\leq\delta\times L^{10},\\
  &A_0,A_L,B_0,B_L\in C^\infty(\mathbb{R}_+), \ \mathrm{rapidly\ decaying\ in\ its\ argument},\\
  &A_0\neq A_L, B_0\neq B_L.
\end{cases}
\end{align*}
Note that the property $A_0\neq A_L, B_0\neq B_L$ plays a key role in creating the $x$-dependence to produce the nonshear flows, for if $A_0=A_L, B_0=B_L$, the equations \eqref{laplace} can be solved for $\phi,\psi$ as just functions of $Y$, which creates another shear flows. For $0<L\ll\delta\ll 1$, the boundedness of $(\tilde u_e^0,\tilde h_e^0)$ and the properties \eqref{condition_for_idealMHD3}-\eqref{M_0} for the nonshear solutions $[\tilde u_e^0,\tilde v_e^0,\tilde h_e^0,\tilde g_e^0,\tilde p_e^0]$ will be verified easily, the readers could refer to the paper \cite{Iyernonshear} in pages 1685-1686 for more details.

Finally, we come to construct the nonshear flows $[u_e^0,v_e^0,h_e^0,g_e^0,p_e^0]$ in our consideration in the expansion \eqref{expansion} satisfying the assumptions \eqref{condition_for_idealMHD1}-\eqref{M_0} by defining
\begin{equation}
  [u_e^0,v_e^0,h_e^0,g_e^0,p_e^0]:=[U_0(Y),0,H_0(Y),0,0]+\delta[\tilde u_e^0,\tilde v_e^0,\tilde h_e^0,\tilde g_e^0,\tilde p_e^0],
\end{equation}
with suitably small constant $\delta$, in which the shear flows $[U_0(Y),0,H_0(Y),0,0]$ and the nonshear flows $[\tilde u_e^0,\tilde v_e^0,\tilde h_e^0,\tilde g_e^0,\tilde p_e^0]$ are prescribed as above. Note that the crucial condition \eqref{condition_for_idealMHD1} can be easily verified by using the condition \eqref{aa1} imposed on $[U_0(Y),0,H_0(Y),0]$ and the smallness of $\delta$ with the boundedness of $(\tilde u_e^0,\tilde h_e^0)$.
\end{proof}

\section{The Well-posedness and estimates for the MHD leading order boundary layer system \eqref{u_p^0}}\label{ap2}
In this section, we focus on proving \textit{Proposition \ref{prop2.1}} by obtaining \textit{a priori estimates} of the MHD boundary layer system \eqref{2.10}-\eqref{2.11} for $(u,v,h,g)$, and then the well-posedness theory for leading order boundary layer corrector $(u_p^0,v_p^0,h_p^0,g_p^0)$ can be directly deduced.
\begin{proposition}\label{propA.1}
$(Weighted~~estimates~~ for~~ D^\alpha(u,h) ~~with~~ |\alpha|\leq m)$ Let $m\geq 5$ be an integer, and $l\in\mathbb{R}$ with $l\geq 0$. Suppose that there exist some positive constants $\vartheta_0$ and suitably small $\sigma_0$ such that
\begin{align}
\label{A.1}
  &\overline u_e^0+u_p^0(0,y)>\overline h_e^0+h_p^0(0,y)\geq \vartheta_0>0,\\
\label{A.2}
  &|\y^{l+1}\p_y(\overline u_e^0+u_p^0,\overline h_e^0+h_p^0)(0,y)|\leq \frac{1}{2}\sigma_0,\\
\label{A.3}
  &|\y^{l+1}\p_y^2(\overline u_e^0+u_p^0,\overline h_e^0+h_p^0)(0,y)|\leq \frac{1}{2}\vartheta_0^{-1},
\end{align}
uniform in y. And also, the hypotheses for $(u_e^0,h_e^0,p_e^0)$ in Proposition \ref{prop2.1} hold. Then there exist classical solutions $(u,v,h,g)$ to problem \eqref{2.10}-\eqref{2.11} in $[0,L]\times(0,+\infty)$ with small $L>0$ satisfying
\begin{align*}
  (u,h)\in L^\infty(0,L;H_l^m(0,+\infty)),\quad (\p_y u,\p_y h)\in L^2(0,L;H_l^m(0,+\infty)).
\end{align*}
Moreover, we have the following estimates
%\sum_{\substack{|\alpha|\leq m\\|\beta|\leq m-1}}
\begin{align}\begin{split}\label{A.4}
  (1)\quad
  &\sup_{x\in[0,L]}\|(u,h)\|_{H_l^m}\\
  &\leq C\vartheta_0^{-4}[\mathcal{P}(M_0+ C(u_b)+ \|(u_0,h_0)\|_{H_l^m}+C(u_b) M_0^6 x)]^{\frac{1}{2}}\\
  &\quad \cdot\{1-C\vartheta_0^{-24}[\mathcal{P}(M_0+ C(u_b)+ \|(u_0,h_0)\|_{H_l^m})+C(u_b) M_0^6 x]^2x\}^{-\frac{1}{4}},
\end{split}\end{align}
\begin{align}\begin{split}\label{A.5}
  (2)\quad
  &\|\y^{l+1}\p_y^i(u,h)\|_{L^\infty}\\
  &\leq C\vartheta_0^{-4}x[\mathcal{P}(M_0+ C(u_b)+ \|(u_0,h_0)\|_{H_l^m}+C(u_b) M_0^6 x)]^{\frac{1}{2}}\\
  &\quad\cdot\{1-C\vartheta_0^{-24}[\mathcal{P}(M_0+ C(u_b)+ \|(u_0,h_0)\|_{H_l^m})+C(u_b) M_0^6 x]^2x\}^{-\frac{1}{4}}\\
  &\quad +\|\y^{l+1}\p_y^i(u_0,h_0)\|_{L^\infty},\quad i=1,2,
\end{split}\end{align}
\begin{align}\begin{split}\label{A.7}
  (3)\quad
  &h(x,y)\\
  &\geq h_0-C\vartheta_0^{-4}x\mathcal{P}[M_0+ C(u_b)+ \|(u_0,h_0)\|_{H_l^m}+C(u_b) M_0^6 x)]^{\frac{1}{2}}\\
  &\quad \cdot\{1-C\vartheta_0^{-24}[\mathcal{P}(M_0+ C(u_b)+ \|(u_0,h_0)\|_{H_l^m})+C(u_b) M_0^6 x]^2x\}^{-\frac{1}{4}},
\end{split}\end{align}
\begin{align}\begin{split}\label{A.8}
  (4)\quad
  &u(x,y)-h(x,y)\\
  &\geq u_0-h_0+C\vartheta_0^{-4}x[\mathcal{P}(M_0+ C(u_b)+ \|(u_0,h_0)\|_{H_l^m}+C(u_b)M_0^6 x)]^{\frac{1}{2}}\\
  &\quad \cdot\{1-C\vartheta_0^{-24}[\mathcal{P}(M_0+ C(u_b)+ \|(u_0,h_0)\|_{H_l^m})+C(u_b) M_0^6 x]^2x\}^{-\frac{1}{4}},
\end{split}\end{align}
in which $\mathcal{P}$ is a polynomial of $\|(u_0,h_0)\|_{H_l^m}$.
In addition, for any $(x,y)\in[0,L]\times[0,+\infty)$, it yields that
\begin{align}\label{A.9}
\begin{split}
  u+u_b+(\overline u_e^0-u_b)\phi'(y)>h+\overline h_e^0\phi'(y)\geq\frac{\vartheta_0}{2}>0,\\
  |\y^{l+1}\p_y(u,h)|\leq\sigma_0,\quad |\y^{l+1}\p_y^2(u,h)|\leq\vartheta_0^{-1}.
\end{split}
\end{align}
\end{proposition}
The proof of \textit{Proposition \ref{propA.1} }  will be achieved by the following three subsections.
\subsection{Weighted $H_l^m-$ estimates with Normal Derivatives}
In this subsection, the weighted estimates for $D^\alpha(u,h)$ with $D^\alpha=\p_x^\beta\partial_y^k,~|\alpha|\leq m,~\beta\leq m-1$ will be given by standard energy method, since one order tangential regularity loss is allowed in this case.
\begin{lemma}\label{lemmaA.1}
$(Weighted~~estimates~~ for~~ D^\alpha(u,h) ~~with~~ |\alpha|\leq m,~\beta\leq m-1)$ Let $m\geq 5$ be an integer, $l\geq0$ be a real number, and the hypotheses for $(u_e^0,h_e^0,p_e^0)$ in Proposition \ref{prop2.1} hold. Suppose that $(u,v,h,g)$ are classical solutions to the problem \eqref{2.10}-\eqref{2.11} in $[0,L]\times(0,+\infty)$ for small $L>0$ satisfying
\begin{align*}
  (u,h)\in L^\infty(0,L;H_l^m(0,+\infty)),\quad (\p_y u,\p_y h)\in L^2(0,L;H_l^m(0,+\infty)).
\end{align*}
Then there exists $C > 0$, depending on $m,l$, such that for any small positive constant $\delta_1$,

\begin{align}\label{A.10}
\begin{split}
  & \sum_{\substack{|\alpha|\leq m\\ \beta\leq m-1}} \large(s_\alpha(x)+\nu\intx\|\p_yD^\alpha u\|_{ L_l^2}^2
    +\kappa\intx\|\p_yD^\alpha h\|_{ L_l^2}^2  \large)\\
  \leq
  &~\sum_{\substack{|\alpha|\leq m\\ \beta\leq m-1}} s_\alpha(0)+C\delta_1\intx\|\p_y(u,h)\|_{ H_0^m}^2+C\sum_{\substack{|\alpha|\leq m\\ \beta\leq m-1}}
     \intx\|\y^{l+k}D^\alpha(r_1,r_2)\|_{ L^2}^2\\
  &~ +C\delta_1^{-1}\intx E_{u,h}^2(1+E_{u,h}^2)+C\intx\sum_{\beta\leq m+1}\|\p_x^\beta(\overline u_e^0,\overline h_e^0,
      \overline p_e^0)\|_{ L^2(0,L)}^2,
\end{split}
\end{align}
where
\begin{equation*}
  s_\alpha(x)=\|\{u+u_b+(\overline u_e^0-u_b)\phi'\}^{1/2}\y^{l+k}D^\alpha(u,h)\|_{ L_y^2}^2,
  \quad E_{u,h}^2=\sum_{|\alpha|\leq m}s_\alpha(x).
\end{equation*}
\end{lemma}

\begin{proof}
Applying operator $D^\alpha=\p_x^\beta \p_y^k$ in \eqref{2.11} yields that
\begin{equation}\label{A.11}
\begin{cases}
    &\{u+u_b+(\overline u_e^0-u_b)\phi'\}\p_x D^\alpha u-(h+\overline h_e^0\phi')\p_x D^\alpha h\\
    &\quad+(v-\overline u_{ex}^0\phi)\p_y D^\alpha u-(g-\overline h_{ex}^0\phi)\p_y D^\alpha h-\nu\p_y^2D^\alpha u\\
  &=-[D^\alpha,\{u+u_b+(\overline u_e^0-u_b)\phi'\}\p_x+(v-\overline u_{ex}^0\phi)\p_y]u\\
    &\quad -D^\alpha(u\overline u_{ex}^0\phi'+v(\overline u_e^0-u_b)\phi'')+D^\alpha(h\overline h_{ex}^0\phi'+g\overline h_e^0\phi'')\\
    &\quad +[D^\alpha,(h+\overline h_e^0\phi')\p_x+(g-\overline h_{ex}^0\phi)\p_y]h+D^\alpha r_1,\\
    &\{u+u_b+(\overline u_e^0-u_b)\phi'\}\p_x D^\alpha h-(h+\overline h_e^0\phi')\p_x D^\alpha u\\
    &\quad +(v-\overline u_{ex}^0\phi)\p_y D^\alpha h(g-\overline h_{ex}^0\phi)\p_y D^\alpha u-\kappa\p_y^2 D^\alpha h\\
  &=-[D^\alpha,\{u+u_b+(\overline u_e^0-u_b)\phi'\}\p_x+\{v-\overline u_{ex}^0\phi\}\p_y]h\\
    &\quad -D^\alpha(u\overline h_{ex}^0\phi'+v\overline h_e^0\phi'')+D^\alpha(h\overline u_{ex}^0\phi'+g(\overline u_e^0-u_b)\phi'')\\
    &\quad +[D^\alpha,(h+\overline h_e^0\phi')\p_x+(g-\overline h_{ex}^0\phi)\p_y]u+D^\alpha r_2.
\end{cases}
\end{equation}
Multiplying $\eqref{A.11}_{1,2}$ by $\y^{2l+2k}D^\alpha u$ and $\y^{2l+2k}D^\alpha h$, respectively, integrating them over $[0,\infty)$ with respect to spatial variable $y$, it gives
\begin{align}\label{A.12}
\begin{split}
  &\frac{1}{2}\frac{d}{dx}\|\{u+u_b+(\overline u_e^0-u_b)\phi'\}^{1/2}\y^{l+k}D^\alpha(u,h)\|_{ L_y^2}^2\\
  =&\inty\y^{2(l+k)}(\nu\p_y^2 D^\alpha u\cdot D^\alpha u+\kappa\p_y^2 D^\alpha h\cdot D^\alpha h)\\
  &+\inty(l+k)y\y^{2(l+k)-2}(v-\overline u_{ex}^0\phi)(|D^\alpha u|^2+|D^\alpha h|^2)\\
  &+\inty\y^{2(l+k)}(D^\alpha r_1\cdot D^\alpha u+D^\alpha r_2\cdot D^\alpha h)\\
  &-\inty\y^{2(l+k)}(I_1\cdot D^\alpha u+I_2\cdot D^\alpha h),
\end{split}
\end{align}
where
\begin{equation}\label{A.13}
\begin{cases}
  I_1&=
  -[(h+\overline h_e^0\phi')\p_x+(g-\overline h_{ex}^0\phi)\p_y]D^\alpha h\\
  &\quad +\Big\{[D^\alpha,(u+u_b+(\overline u_e^0-u_b)\phi')\p_x+(v-\overline u_{ex}^0\phi)\p_y]u\\
  &\quad -[D^\alpha,(h+\overline h_e^0\phi')\p_x+(g-\overline h_{ex}^0\phi)\p_y]h\Big\}\\
  &\quad +D^\alpha[u\overline u_{ex}^0\phi'+v(\overline u_e^0-u_b)\phi''-h\overline h_{ex}^0\phi'-g\overline h_e^0\phi'']\\
  &:= I_1^1+I_1^2+I_1^3,\\
  I_2&=
  -[\{h+\overline h_e^0\phi'\}\p_x+\{g-\overline h_{ex}^0\phi\}\p_y]D^\alpha u\\
  &\quad +\Big\{[D^\alpha,(u+u_b+(\overline u_e^0-u_b)\phi')\p_x+(v-\overline u_{ex}^0\phi)\p_y]h\\
  &\quad -[D^\alpha,(h+\overline h_e^0\phi')\p_x+(g-\overline h_{ex}^0\phi)\p_y]u\Big\}\\
  &\quad +D^\alpha[u\overline h_{ex}^0\phi'+v\overline h_e^0\phi''-h\overline u_{ex}^0\phi'-g(\overline u_e^0-u_b)\phi'']\\
  &:= I_2^1+I_2^2+I_2^3.
\end{cases}
\end{equation}

The estimates for the above terms in \eqref{A.12} can be obtained by following the similar arguments in \cite{DLX}, so we do not give details for every single term for simplifying the presentation. Here we only state some new different terms.

Firstly, we will handle the term $\inty\nu\p_y^2 D^\alpha u\cdot \y^{2(l+k)}D^\alpha u$, and then the term $\inty\kappa\p_y^2 D^\alpha h\cdot \y^{2(l+k)}D^\alpha h$ can be estimated similarly. Indeed, we have

\begin{align}\label{A.14}
\begin{split}
  \inty\nu\p_y^2 D^\alpha u\cdot \y^{2(l+k)}D^\alpha u
  =&-\nu\|\y^{l+k}\p_y D^\alpha u\|_{ L_y^2}^2+\nu(\p_y D^\alpha u\cdot D^\alpha u)|_{ y=0}\\
  &-2\nu(l+k)\inty\y^{2(l+k)-2}y\p_y D^\alpha u\cdot D^\alpha u.
\end{split}
\end{align}

The boundary term in \eqref{A.14} shall be treated carefully in two cases: $|\alpha|\leq m-1$ and $|\alpha|= m$. Here we only discuss the case of $|\alpha|= m$ which is different from that in \cite{DLX}. It should be noticed that we have $k \geq 1$ by virtue of $\beta\leq m-1$. Denote $\gamma=(\beta, k-1)$ with $|\gamma|= m-1$, and using the equations \eqref{2.11} for $(u,v,h,g)$, there holds
\begin{align}\label{A.15}
\begin{split}
  \nu\p_y D^\alpha u
  =& \nu D^\gamma \p_y^2 u\\
  =& D^\gamma \{[(u+u_b+(\overline u_e^0-u_b)\phi')\p_x+(v-\overline u_{ex}^0\phi)\p_y]u+u\overline u_{ex}^0\phi'+v(\overline u_e^0-u_b)\phi''\\
  &-[(h+\overline h_e^0\phi')\p_x+(g-\overline h_{ex}^0\phi)\p_y]h-h\overline h_{ex}^0\phi'-g\overline h_e^0\phi''-r_1\}.
\end{split}
\end{align}
Then, according to the definition of $r_1$ with the fact $\phi \equiv 0$ for $y\leq R_0$, we get at $\{y=0\}$ that
\begin{align}\label{A.16}
\begin{split}
  \nu\p_y D^\alpha u|_{y=0}
  =&D^\gamma\{[(u+u_b)\p_x+v\p_y]u-(h\p_x+g\p_y)h+2\overline p_{ex}^0\}\\
  =&D^\gamma[(u+u_b)\p_x-h\p_x h]+D^\gamma(v\p_y u-g\p_y h)+2D^\gamma\overline p_{ex}^0,
\end{split}
\end{align}
note that the pressure term is generated by nonshear flow in our consideration. An application of \textit{Newton-Lebniz formula} and \textit{$H\ddot{o}lder$ inequality} yields
\begin{align}\label{A.17}
\begin{split}
  2D^\gamma\overline p_{ex}^0\cdot D^\alpha u|_{y=0}
  &\leq 2\sqrt{2} \|D^\gamma\overline p_{ex}^0\|_{L^\infty(0,L)}\|D^\alpha u\|_{L_y^2}^\frac{1}{2}\|\p_y D^\alpha u\|_{L_y^2}^\frac{1}{2}\\
  &\lesssim \frac{\nu}{14}\|\p_y D^\alpha u\|_{L_y^2}^2+C\|D^\alpha u\|_{L_y^2}^2+C\|D^\gamma\overline p_{ex}^0\|_{L^\infty(0,L)}^2.
\end{split}
\end{align}
With \eqref{A.17} in hands, and operating the proof similar to \cite{DLX} in pages \rm{39 - 41} to the other terms of \eqref{A.16}, we can obtain the estimate of the boundary term for $|\alpha| = \beta+k \leq m$ with $\beta\leq m-1$ as follows:
\begin{align}\label{A.18}
\begin{split}
  \large|(\nu\p_y D^\alpha u\cdot D^\alpha u)|_{y=0}\large|
  \leq&~\frac{3\nu}{7}\|\p_y D^\alpha u\|_{L_y^2}^2+C\delta_1^{-1}E_{u,h}^2(1+E_{u,h}^2)\\
  &+\delta_1\|\p_y(u,h)\|_{H_0^m}^2+C\|D^\gamma\overline p_{ex}^0\|_{H^1}^2+C(u_b).
\end{split}
\end{align}

Therefore, we have the following estimate
\begin{align}\label{A.19}
\begin{split}
  &\inty\nu\p_y^2 D^\alpha u\cdot \y^{2(l+k)}D^\alpha u\\
  \leq &~~-\frac{\nu}{2}\|\y^{(l+k)}\p_y D^\alpha u\|_{L_y^2}^2+C\delta_1^{-1}E_{u,h}^2(1+E_{u,h}^2)\\
  &~~+\delta_1\|\p_y(u,h)\|_{H_0^m}^2+C\sum_{\beta\leq m-1}\|\p_x^\beta \overline p_{ex}^0\|_{H^1}^2+C(u_b).
\end{split}
\end{align}
In a similar fashion, one gets
\begin{align}\label{A.20}
\begin{split}
  &\inty\kappa\p_y^2 D^\alpha h\cdot \y^{2(l+k)}D^\alpha h\\
  \leq &~~-\frac{\kappa}{2}\|\y^{(l+k)}\p_y D^\alpha h\|_{L_y^2}^2+C\delta_1^{-1}E_{u,h}^2(1+E_{u,h}^2)+\delta_1\|\p_y(u,h)\|_{H_0^m}^2.
\end{split}
\end{align}

Secondly, using \textit{Hardy inequality} and \textit{Sobolev embedding inequality}, we can obtain that
\begin{align}\label{A.21}
\begin{split}
  &\inty(l+k)y\y^{2(l+k)-2}(v-\overline u_{ex}^0\phi)(|D^\alpha u|^2+|D^\alpha h|^2)\\
  \leq& C\bigg(\bigg\|\frac{v}{\y}\bigg\|_{L^\infty}+\|\overline u_{ex}^0\|_{L^\infty(0,L) }\bigg)
        \| \{u+u_b+(\overline u_e^0-u_b)\phi'\}^{1/2}\y^{l+k}D^\alpha(u,h)\|_{ L_y^2}^2\\
  \leq& C(\|u_x\|_{L^\infty}+\|\overline u_{ex}^0\|_{ L^\infty(0,L)})E_{u,h}^2
  \leq C(\|u\|_{H_0^3}+\|\overline u_{ex}^0\|_{L^\infty(0,L) })E_{u,h}^2,
\end{split}
\end{align}
in which the following \textit{a priori assumption}
\begin{equation}\label{priori}
  u+u_b+(\overline u_e^0-u_b)\phi'(y)\geq \tilde c_0>0
\end{equation}
has been applied.

For the third term in \eqref{A.12}, using the above \textit{a priori assumption} \eqref{priori} again, it is easy to get
\begin{align}\label{A.22}
\begin{split}
  &\inty\y^{2(l+k)}(D^\alpha r_1\cdot D^\alpha u+D^\alpha r_2\cdot D^\alpha h)\\
  \leq&\frac{1}{2}\|\{u+u_b+(\overline u_e^0-u_b)\phi'\}^{1/2}\y^{l+k}D^\alpha(u,h)\|_{ L_y^2}^2\\
  &+\frac{1}{2}C(\vartheta_0,\tilde c_0)\|\y^{l+k}D^\alpha(r_1,r_2)\|_{L_y^2}^2.
\end{split}
\end{align}

Finally, it remains to handle the terms $-\inty\y^{2(l+k)}(I_1\cdot D^\alpha u+I_2\cdot D^\alpha h)$. Recall the definition of $I_1$ and $I_2$ in \eqref{A.13}, we first divide the target term into three parts:
\begin{align}\label{A.23}
\begin{split}
  &-\inty\y^{2(l+k)}(I_1\cdot D^\alpha u+I_2\cdot D^\alpha h)\\
  =& ~-\sum_{i=1}^3\inty\y^{2(l+k)}(I_1^i\cdot D^\alpha u+I_2^i\cdot D^\alpha h) \\
  :=& ~G_1+G_2+G_3.
\end{split}
\end{align}

Noticing that
\begin{align*}
  \phi(y)\equiv y,\quad \phi'(y)\equiv 1,\quad \phi^{(i)}\equiv 0\quad {\rm{for}} \quad y\geq 2R_0,\quad i\geq 2,
\end{align*}
there exists some positive constant $C$ satisfying
\begin{align}
  \label{A.24}
  \|\y^{i-1}\phi^{(i)}(y)\|_{L^\infty(0,+\infty)}& \leq C,\quad i=0,1,\\
   \label{A.25}
  \|\y^{\lambda}\phi^{(j)}(y)\|_{L^\infty(0,+\infty)}& \leq C,\quad j\geq 2,\quad \lambda\in\mathbb{R}.
\end{align}
{\bf\textit{{Estimate for $G_1$}}}

We obtain by integration by parts that
\begin{align}\label{A.26}
\begin{split}
  G_1=
  &\frac{d}{dx}\inty (h+\overline h_e^0\phi')D^\alpha h\cdot \y^{2(l+k)} D^\alpha u\\
  &~~-\inty (l+k)(g-\overline h_{ex}^0\phi) \y^{2(l+k)-2}\cdot 2y D^\alpha h D^\alpha u.
\end{split}
\end{align}
The first term will be absorbed into the left hand side by using \textit{a priori assumption}
\begin{equation}\label{priori_assume}
  u+u_b+(\overline u_e^0-u_b)\phi'(y)>h+\overline h_e^0\phi'(y)\geq\frac{\vartheta_0}{2}>0
\end{equation}
later on, so here we set it aside. The second term can be estimated as
\begin{align}\label{A.27}
  &-\inty (l+k)(g-\overline h_{ex}^0\phi) \y^{2(l+k)-2} \cdot 2y D^\alpha h D^\alpha u\nonumber\\
  \leq&~~ 2(l+k)\Bigg\|\frac{g-\overline h_{ex}^0\phi}{1+y}\Bigg\|_{L^\infty}\cdot
          \|\{u+u_b+(\overline u_e^0-u_b)\phi'\}^{1/2} \y^{l+k} D^\alpha(u,h)\|_{L_y^2}^2 \\
  \leq&~~ C(\|h\|_{H_0^3}+\|\overline h_{ex}^0\|_{L^\infty(0,L)})E_{u,h}^2.\nonumber
\end{align}
{\bf\textit{{Estimate for $G_2$}}}

For $G_2$, an application of \textit{Cauchy-Schwarz inequality} yields that
\begin{align}\label{A.28}
\begin{split}
  G_2\leq
  &~~C\|\y^{l+k}I_1^2\|_{L_y^2}\cdot \|\{u+u_b+(\overline u_e^0-u_b)\phi'\}^{1/2} \y^{l+k} D^\alpha u\|_{L_y^2}\\
  &~~+C\|\y^{l+k}I_2^2\|_{L_y^2}\cdot \|\{u+u_b+(\overline u_e^0-u_b)\phi'\}^{1/2} \y^{l+k} D^\alpha h\|_{L_y^2}.
\end{split}
\end{align}
Hence, we shall estimate $\|\y^{l+k}I_1^2\|_{L_y^2}$, and $\|\y^{l+k}I_2^2\|_{L_y^2}$ will be handled by a similar argument.

The terms in $I_1^2$ can be rearranged as
\begin{align*}
\begin{split}
  I_1^2
  &=\{[D^\alpha,(u+u_b)\p_x+v\p_y]u - [D^\alpha,h\p_x+g\p_y]h\}\\
  &\quad\quad+\{[D^\alpha,((\overline u_e^0-u_b)\phi')\p_x-(\overline u_{ex}^0\phi)\p_y]u
   - [D^\alpha,(\overline h_e^0\phi')\p_x-(\overline h_{ex}^0\phi)\p_y]h\}\\
  &\triangleq I_{1,1}^2+I_{1,2}^2.
\end{split}
\end{align*}
Applying some similar technical estimates as the shear flows to $I_{1,1}^2$, we achieve
\begin{equation}\label{A.29}
  \|\y^{l+k}I_{1,1}^2\|_{L_y^2}\leq C E_{u,h}^2.
\end{equation}
Next, rewriting $I_{1,2}^2$ as
\begin{align}\label{A.30}
\begin{split}
  I_{1,2}^2
  =\sum_{0<\tilde\alpha\leq\alpha}\big(\begin{array}{c}
                                     \alpha \\
                                     \tilde\alpha
                                   \end{array}\big)
   & \Bigg\{[D^{\tilde\alpha}((\overline u_e^0-u_b)\phi')\p_x-D^{\tilde\alpha}(\overline u_{ex}^0\phi)\p_y]D^{\alpha-\tilde\alpha}u\\
   & -[D^{\tilde\alpha}(\overline h_e^0\phi')\p_x-D^{\tilde\alpha}(\overline h_{ex}^0\phi)\p_y]D^{\alpha-\tilde\alpha}h\Bigg\}
\end{split}
\end{align}
with $D^{\tilde\alpha}=\p_x^{\tilde\beta}\p_y^{\tilde k} $. And then each term on the right hand side can be estimated as follows:
\begin{align*}
    &\|\y^{l+k}D^{\tilde\alpha}((\overline u_e^0-u_b)\phi')\cdot\p_x D^{\alpha-\tilde\alpha}u\|_{L_y^2}\\
  \leq& \|\y^{\tilde k} D^{\tilde\alpha}((\overline u_e^0-u_b)\phi')\|_{ L^{\infty} }\cdot\|\y^{l+k-\tilde k}\p_x D^{\alpha-\tilde\alpha}u\|_{L_y^2}\\
  \leq& \big(C(u_b)+\|\p_x^{\tilde\beta}\overline u_e^0\|_{L^{\infty}(0,L)}\big)\|u\|_{H_l^m},\\
\end{align*}
and
\begin{align*}
    &\|\y^{l+k}D^{\tilde\alpha}(\overline u_{ex}^0\phi)\cdot\p_y D^{\alpha-\tilde\alpha}u\|_{L_y^2}\\
  \leq& \|\y^{\tilde k-1} D^{\tilde\alpha}(\overline u_{ex}^0\phi)\|_{L^\infty}\cdot\|\y^{l+k-\tilde k+1}\p_y D^{\alpha-\tilde\alpha}u\|_{L_y^2}\\
  \leq& \|\p_x^{\tilde\beta}\overline u_{ex}^0\|_{L^{\infty}(0,L)}\cdot\|u\|_{H_l^m},
\end{align*}
in which we have used the boundedness of $\phi$ from \eqref{A.24} and \eqref{A.25}. Similarly, the other two terms yields that
\begin{align*}
   &\|\y^{l+k}[D^{\tilde\alpha}(\overline h_e^0\phi')\p_x-D^{\tilde\alpha}(\overline h_{ex}^0\phi)\p_y]D^{\alpha-\tilde\alpha}h\|_{L_y^2}\\
  \leq& C(\|\p_x^{\tilde\beta}\overline h_e^0\|_{L^{\infty}(0,L)}+\|\p_x^{\tilde\beta}\overline h_{ex}^0\|_{L^{\infty}(0,L)})\cdot\|h\|_{H_l^m}.
\end{align*}
So collecting the above three estimates we have
\begin{equation}\label{A.31}
  \|\y^{l+k}I_{1,2}^2\|_{L_y^2}\leq \big(\sum_{\beta\leq m}\|\p_x^{\beta}(\overline u_e^0,\overline h_e^0)\|_{L^{\infty}(0,L)}+C(u_b)\big)E_{u,h}.
\end{equation}

Therefore, together with \eqref{A.29} it gives
\begin{equation}\label{A.32}
  \|\y^{l+k}I_1^2\|_{L_y^2}
  \leq \big(\sum_{\beta\leq m}\|\p_x^{\beta}(\overline u_e^0,\overline h_e^0)\|_{L^{\infty}(0,L)}+C(u_b)+E_{u,h}\big)E_{u,h},
\end{equation}
and in the same way
\begin{equation}\label{A.33}
  \|\y^{l+k}I_2^2\|_{L_y^2}
  \leq \big(\sum_{\beta\leq m}\|\p_x^{\beta}(\overline u_e^0,\overline h_e^0)\|_{L^{\infty}(0,L)}+C(u_b)+E_{u,h}\big)E_{u,h}.
\end{equation}
Accordingly, we can conclude that
\begin{equation}\label{A.34}
  G_2\leq  \big(\sum_{\beta\leq m}\|\p_x^{\beta}(\overline u_e^0,\overline h_e^0)\|_{L^{\infty}(0,L)}+C(u_b)+E_{u,h}\big)E_{u,h}^2.
\end{equation}
{\bf\textit{{Estimate for $G_3$}}}

Applying \textit{Cauchy-Schwarz inequality} in $G_3$, we get
\begin{align}\label{A.35}
  G_3
  \leq&~C\|\y^{l+k}I_1^3\|_{L_y^2}\cdot \|\y^{l+k} D^\alpha u\|_{L_y^2}+C\|\y^{l+k}I_2^3\|_{L_y^2}\cdot
   \|\y^{l+k} D^\alpha h\|_{L_y^2}.
\end{align}
Noted that the weighted estimate of $I_2^3$ will be treated similarly, once we have the estimate for $\|\y^{l+k}I_1^3\|_{L_y^2}$. From the definition of $I_1^3$, the term can be rewritten as
\begin{align}\label{A.36}
\begin{split}
  I_1^3
  =\sum_{0\leq\tilde\alpha\leq\alpha}\big(\begin{array}{c}
                                     \alpha \\
                                     \tilde\alpha
                                   \end{array}\big)
   & \Bigg\{ D^{\tilde\alpha}u \cdot D^{\alpha-\tilde\alpha}(\overline u_{ex}^0\phi')
   +D^{\tilde\alpha}v \cdot D^{\alpha-\tilde\alpha}((\overline u_e^0-u_b)\phi'')\\
   &-D^{\tilde\alpha}h \cdot D^{\alpha-\tilde\alpha}(\overline h_{ex}^0\phi')
   -D^{\tilde\alpha}g \cdot D^{\alpha-\tilde\alpha}(\overline h_e^0\phi'')
   \Bigg\}.\nonumber
\end{split}
\end{align}
Operating a similar procedure to the above estimates for $I_1^2$ and referring to \cite{LXYwell} in details as well, we obtain that
\begin{equation}\label{A.37}
  \|\y^{l+k}I_j^3\|_{L_y^2}
  \leq \big(\sum_{\beta\leq m}\|\p_x^{\beta}(\overline u_e^0,\overline h_e^0)\|_{L^{\infty}(0,L)}+C(u_b)\big)E_{u,h},\quad {\rm{for}}\quad j=1,2.
\end{equation}
Substituting \eqref{A.37} into \eqref{A.35}, it yields that
\begin{equation}\label{A.38}
  G_3
  \leq \bigg(\sum_{\beta\leq m}\|\p_x^{\beta}(\overline u_e^0,\overline h_e^0)\|_{L^{\infty}(0,L)}+C(u_b)\bigg)E_{u,h}^2.
\end{equation}

Now, collecting the estimates for $G_i$, together with equation \eqref{A.23}, we can deduce that
\begin{align}\label{A.39}
  &-\inty\y^{2(l+k)}(I_1\cdot D^\alpha u+I_2\cdot D^\alpha h)\nonumber\\
  \leq& C\big(\sum_{\beta\leq m}\|\p_x^{\beta}(\overline u_e^0,\overline h_e^0)\|_{L^{\infty}(0,L)}+C(u_b)+E_{u,h}\big)E_{u,h}^2\\
  & +\frac{d}{dx}\inty (h+\overline h_e^0\phi')D^\alpha h\cdot \y^{2(l+k)} D^\alpha u.\nonumber
\end{align}

At present, plugging the estimates \eqref{A.19},\eqref{A.20},\eqref{A.21},\eqref{A.22} and \eqref{A.39} into \eqref{A.12}, integrating in $x$-direction, and summing over $\alpha$ with $\beta\leq m-1$, we find that
\begin{align}\label{A.40}
\begin{split}
  & \sum_{\substack{|\alpha|\leq m\\ \beta\leq m-1}} \large(s_\alpha(x)+\nu\intx\|\p_yD^\alpha u\|_{ L_l^2}^2
  +\kappa\intx\|\p_yD^\alpha h\|_{ L_l^2}^2  \large)\\
  \leq&~ C\delta_1\intx\|\p_y(u,h)\|_{ H_0^m}^2+C\sum_{\substack{|\alpha|\leq m\\ \beta\leq m-1}}
  \intx\|\y^{l+k}D^\alpha(r_1,r_2)\|_{ L^2}^2\\
  &~ +C\delta_1^{-1}\intx E_{u,h}^2(1+E_{u,h}^2)+C\intx \bigg(C(u_b)+\sum_{\beta\leq m+1}\|\p_x^\beta(\overline u_e^0,\overline h_e^0,\overline p_e^0)\|_{L^2(0,L)}^2\bigg)\\
  &~~+\sum_{\substack{|\alpha|\leq m\\ \beta\leq m-1}} \inty (h+\overline h_e^0\phi')D^\alpha h\cdot \y^{2(l+k)} D^\alpha u+\sum_{\substack{|\alpha|\leq m\\ \beta\leq m-1}} s_\alpha(0),
\end{split}
\end{align}
in which the positive constant $C$ depends on $m,l$. Moreover, the \textit{a priori assumption} \eqref{priori_assume}(which shall be verified latter in our energy closing arguments) yields that
\begin{align}\label{A.41}
\begin{split}
  &\inty (h+\overline h_e^0\phi')D^\alpha h\cdot \y^{2(l+k)} D^\alpha u\\
  \leq&~ \frac{1}{2}\|\{u+u_b+(\overline u_e^0-u_b)\phi'\}^{1/2}\y^{l+k}D^\alpha(u,h)\|_{ L_y^2}^2
  \leq~ \frac{1}{2}s_\alpha(x),
\end{split}
\end{align}

Putting \eqref{A.41} into \eqref{A.40}, we finish the proof of \eqref{A.10}.
\end{proof}

\subsection{Weighted $H_l^m$- Estimates only in Tangential Variables}
Similar to classical Prandtl equations, the essential difficulty for solving the MHD boundary equations also lies on the loss of one derivative in tangential variable $x$, which comes from the terms $v\p_y u-g\p_y h$ and $v\p_y h-g\p_y u$ because of the divergence-free conditions.

Taking the $m^{th}$- order tangential derivatives on \eqref{2.11}, we have the following equations for $\p_x^\beta(u,h)$ with $\beta=m$,
\begin{align}\label{A.42}
\begin{cases}
  &[\{u+u_b+(\overline u_e^0-u_b)\phi'\}\p_x+\{v-\overline u_{ex}^0\phi\}\p_y]\p_x^\beta u\\
  &\quad -[\{h+\overline h_e^0\phi'\}\p_x+\{g-\overline h_{ex}^0\phi\}\p_y]\p_x^\beta h-\nu\p_y^2\p_x^\beta u\\
  &\quad +[\p_y u+(\overline u_e^0-u_b)\phi'']\p_x^\beta v-(\p_y h+\overline h_e^0\phi'')\p_x^\beta g=\p_x^\beta r_1+R_u^\beta,\\
  &[\{u+u_b+(\bar
   u_e^0-u_b)\phi'\}\p_x+\{v-\overline u_{ex}^0\phi\}\p_y]\p_x^\beta h\\
  &\quad -[\{h+\overline h_e^0\phi'\}\p_x+\{g-\overline h_{ex}^0\phi\}\p_y]\p_x^\beta u-\kappa\p_y^2\p_x^\beta h\\
  &\quad +(\p_y h+\overline h_e^0\phi'')\p_x^\beta v-[\p_y u+(\overline u_e^0-u_b)\phi'']\p_x^\beta g=\p_x^\beta r_2+R_h^\beta,
\end{cases}
\end{align}
where
\begin{align*}
    R_u^\beta=
  &-[\p_x^\beta, \{u+u_b+(\overline u_e^0-u_b)\phi'\}\p_x-\overline u_{ex}^0\phi\p_y] u
   +\p_x^\beta(-\overline u_{ex}^0\phi'u+\overline h_{ex}^0\phi'h)\\
  &+[\p_x^\beta,  \{h+\overline h_e^0\phi'\}\p_x+\overline h_{ex}^0\phi\p_y] h-[\p_x^\beta,(\overline u_e-u_b)\phi'']v
   +[\p_x^\beta,\overline h_e^0\phi'']g\\
  &-\sum_{0<\tilde\beta<\beta}\bigg(\begin{array}{c}
                                 \beta \\
                                 \tilde\beta
                               \end{array}\bigg)(\p_x^{\tilde\beta}v\cdot\p_x^{\beta-\tilde\beta}\p_y u
                               -\p_x^{\tilde\beta}g\cdot\p_x^{\beta-\tilde\beta}\p_y h),\\
    R_h^\beta=
  &-[\p_x^\beta, \{u+u_b+(\overline u_e^0-u_b)\phi'\}\p_x-\overline u_{ex}^0\phi\p_y] h
   +\p_x^\beta(-\overline h_{ex}^0\phi'u+\overline u_{ex}^0\phi'h)\\
  &+[\p_x^\beta,  \{h+\overline h_e^0\phi'\}\p_x+\overline h_{ex}^0\phi\p_y] u-[\p_x^\beta,\overline h_e^0\phi'']v+[\p_x^\beta,(\overline u_e-u_b)\phi'']g\\
  &-\sum_{0<\tilde\beta<\beta}\bigg(\begin{array}{c}
                                 \beta \\
                                 \tilde\beta
                               \end{array}\bigg)(\p_x^{\tilde\beta}v\cdot\p_x^{\beta-\tilde\beta}\p_y h
                               -\p_x^{\tilde\beta}g\cdot\p_x^{\beta-\tilde\beta}\p_y u).
\end{align*}
Then, provided $m \geq 5$, we can deduce that
\begin{equation}\label{A.43}
  \|(R_u^\beta,R_h^\beta)\|_{L_l^2}
  \leq C\bigg(\sum_{\beta\leq m+1}\|\p_x^{\beta}(\overline u_e^0,\overline h_e^0)\|_{L^{\infty}(0,L)}
  +\|(u,h)\|_{H_l^m}\bigg)\|(u,h)\|_{H_l^m}
\end{equation}
which is proved by a similar argument of \cite{LXYwell}, the readers can refer to the paper \cite{LXYwell} in page 91 for more details.

Now, back to the equation \eqref{A.42}, since $v=-\p_y^{-1}\p_x u$ and $g=-\p_y^{-1}\p_x h$ will create a loss of the $x$-derivative, it prevents us from applying standard energy methods. To overcome this difficulty, we give \textit{a priori assumption} (which shall be verified latter in our energy closing arguments) that there exists a positive constant $\vartheta_0$ such  that
\begin{equation*}
  h(x,y)+\overline h_e^0(x)\phi'(y)\geq \frac{\vartheta_0}{2}>0,\quad {\rm{for~~ any }}~~(x,y)\in[0,L]\times (0,+\infty),
\end{equation*}
and introduce the following quantities:
\begin{align}\label{A.44}
\begin{split}
  u_\beta:=& \p_x^\beta u-\frac{\p_y u+(\overline u_e^0-u_b)\phi''}{h+\overline h_e^0\phi'}\p_x^\beta\p_y^{-1} h,\\
  h_\beta:=& \p_x^\beta h-\frac{\p_y h+\overline h_e^0\phi''}{h+\overline h_e^0\phi'}\p_x^\beta\p_y^{-1} h.
\end{split}
\end{align}
We note here that $(u_\beta,h_\beta)$ are almost equivalent to $\p_x^\beta(u,h)$ in $L_l^2$-norm, which will be demonstrated later in Lemma \ref{lemmaA.3} at this subsection.

Rewriting the second equation of \eqref{2.11} for $h$ as
\begin{equation}\label{A.45}
 \p_y[(v-\overline u_{ex}^0\phi)(h+\overline h_e^0\phi')-(g-\overline h_{ex}^0\phi)(u+u_b+(\overline u_e^0-u_b)\phi')]-\kappa\p_y^2 h=\kappa\overline h_e^0\phi^{(3)},
\end{equation}
in which we have used Bernoulli's law for twice. And then it gives
\begin{equation}\label{A.46}
  (v-\overline u_{ex}^0\phi)(h+\overline h_e^0\phi')-(g-\overline h_{ex}^0\phi)(u+u_b+(\overline u_e^0-u_b)\phi')-\kappa\p_y h=\kappa\overline h_e^0\phi''.
\end{equation}
Thanks to divergence-free condition, there exists a stream function $\overline\psi$ satisfying
\begin{equation}\label{A.47}
  (h,g)=(\p_y\overline\psi,-\p_x\overline\psi),\quad {\rm{with}}~~\overline\psi|_{y=0}=0.
\end{equation}
So that the equation for $\overline\psi$ is deduced to
\begin{equation}\label{A.48}
  [(u+u_b+(\overline u_e^0-u_b)\phi')\p_x+(v-\overline u_{ex}^0\phi)\p_y]\overline\psi+\overline h_{ex}^0\phi u+\overline h_e^0\phi'v-\kappa\p_y^2\overline\psi=r_3,
\end{equation}
where we have used the Bernoulli's law and $r_3$ is defined by
\begin{equation}\label{A.49}
  r_3=u_b\overline h_{ex}^0\phi(\phi'-1)+\kappa\overline h_e^0\phi''.
\end{equation}

Next, applying $\p_x^\beta$ to equation \eqref{A.48} and using $\p_y\overline\psi=h$, we get
\begin{align}\label{A.50}
\begin{split}
  [(u+u_b+(\overline u_e^0-u_b)\phi')\p_x+(v-\overline u_{ex}^0\phi)\p_y]\p_x^\beta\overline\psi&\\
  \quad\quad\quad +(h+\overline h_e^0\phi')\p_x^\beta v-\kappa\p_y^2\p_x^\beta\overline\psi
  &=\p_x^\beta r_3+ R_{\psi}^\beta,
\end{split}
\end{align}
where
\begin{align}\label{A.51}
\begin{split}
  R_\psi^\beta=
  &-[\p_x^\beta, (u+u_b+(\overline u_e^0-u_b)\phi')\p_x -\overline u_{ex}^0\phi\p_y]\overline\psi\\
  &-\sum_{0<\tilde\beta<\beta}
  \bigg(
        \begin{array}{cc}
        \beta\\
        \tilde\beta
  \end{array}
  \bigg)\p_x^{\tilde\beta}v \cdot \p_x^{\beta-\tilde\beta}\p_y\overline\psi
  -\p_x^\beta(\overline h_{ex}^0\phi u)-[\p_x^\beta,\overline h_e^0\phi']v.
\end{split}
\end{align}
And $R_\psi^\beta$ can be estimated as
\begin{equation}\label{A.52}
  \bigg\|\frac{R_\psi^\beta}{1+y}\bigg\|_{L_0^2}
  \leq C\bigg(\sum_{\beta\leq m+1}\|\p_x^{\beta}(\overline u_e^0,\overline h_e^0)\|_{L^\infty(0,L)}
       +\|(u,h)\|_{H_0^m}\bigg)\|(u,h)\|_{H_0^m},
\end{equation}
in which we have used \textit{Hardy inequality} and the boundedness of $\phi$ from \eqref{A.24}.

At present, recalling the definition of $\overline\psi$, we can rewrite quantity $u_\beta,h_\beta$ in the following form:
\begin{equation}\label{A.53}
  u_\beta=\p_x^\beta u-\eta_1\p_x^\beta \overline\psi,\quad
  h_\beta=\p_x^\beta h-\eta_2\p_x^\beta \overline\psi,
\end{equation}
with
\begin{equation}\label{A.54}
  \eta_1:=\frac{\p_y u+(\overline u_e^0-u_b)\phi''}{h+\overline h_e^0\phi'},\quad
  \eta_2:=\frac{\p_y h+\overline h_e^0\phi''}{h+\overline h_e^0\phi'}.
\end{equation}
Then, using equations for $\p_x^\beta(u,h)$ and $\p_x^\beta\overline\psi$, we can obtain the following equations and boundary conditions for $(u_\beta,h_\beta)$, in which the tough terms have been cancelled,
\begin{align}\label{A.55}
\begin{cases}
    &[(u+u_b+(\overline u_e^0-u_b)\phi')\p_x+(v-\overline u_{ex}^0\phi)\p_y] u_\beta -\nu\p_y^2 u_\beta\\
  &\quad\quad-[(h+\overline h_e^0\phi')\p_x+(g-\overline h_{ex}^0\phi)\p_y] h_\beta+(\kappa-\nu)\eta_1\p_y h_\beta = R_1^\beta,\\
    &[(u+u_b+(\overline u_e^0-u_b)\phi')\p_x+(v-\overline u_{ex}^0\phi)\p_y] h_\beta -\kappa\p_y^2 h_\beta\\
  &\quad\quad-[(h+\overline h_e^0\phi')\p_x+(g-\overline h_{ex}^0\phi)\p_y] u_\beta = R_2^\beta,\\
  &(u_\beta,\p_y h_\beta)|_{y=0}=0,
\end{cases}
\end{align}
where
\begin{equation}\label{A.56}
\begin{cases}
  &R_1^\beta=
  \p_x^\beta r_1+R_u^\beta-\eta_1 \p_x^\beta r_3-\eta_1 R_\psi^\beta -\zeta_1 \p_x^\beta\overline\psi\\
  &\quad\quad\quad\quad +\{2\nu\p_y\eta_1+(g-\overline h_{ex}^0\phi)\eta_2-(\kappa-\nu)\eta_1\eta_2\}\p_x^\beta h,\\
  &R_2^\beta=
  \p_x^\beta r_1+R_h^\beta-\eta_2 \p_x^\beta r_3-\eta_2 R_\psi^\beta -\zeta_2 \p_x^\beta\overline\psi\\
  &\quad\quad\quad\quad +\{2\kappa\p_y\eta_2+(g-\overline h_{ex}^0\phi)\eta_1\}\p_x^\beta h,
\end{cases}
\end{equation}
with
\begin{equation}\label{A.57}
\begin{cases}
  \zeta_1=
  &[(u+u_b+(\overline u_e^0-u_b)\phi')\p_x+(v-\overline u_{ex}^0\phi)\p_y]\eta_1-\nu\p_y^2\eta_1\\
  &\quad -[(h+\overline h_e^0\phi')\p_x+(g-\overline h_{ex}^0\phi)\p_y] \eta_2+(\kappa-\nu)\eta_1\p_y\eta_2,\\
  \zeta_2=
  &[(u+u_b+(\overline u_e^0-u_b)\phi')\p_x+(v-\overline u_{ex}^0\phi)\p_y]\eta_2-\kappa\p_y^2\eta_2\\
  &\quad -[(h+\overline h_e^0\phi')\p_x+(g-\overline h_{ex}^0\phi)\p_y] \eta_2.
\end{cases}
\end{equation}
Also, through direct calculation, the corresponding ``initial data'' (values at $x=0$) becomes
\begin{align}\label{A.58}
\begin{split}
    &u_\beta\big|_{x=0}
  = \p_x^\beta u(0,y)- \frac{\p_y u_0(y)+(\overline u_e^0(0)-u_b)\phi''}{h_0(y)+\overline h_e^0(0)\phi'}\int_0^y\p_x^\beta h(0,z)dz
   \triangleq u_{\beta 0}(y),\\
    &h_\beta\big|_{x=0}
  = \p_x^\beta h(0,y)- \frac{\p_y h_0(y)+\overline h_e^0(0)\phi''}{h_0(y)+\overline h_e^0(0)\phi'}\int_0^y\p_x^\beta h(0,z)dz
  \triangleq h_{\beta 0}(y).
\end{split}
\end{align}

On the one hand, by virtue of $\overline\psi=\p_y^{-1}h$,
\begin{equation}\label{A.59}
  \|\y^{-1}\p_x^\beta \overline\psi\|_{L_y^2}\leq\|\p_x^\beta h\|_{L_y^2}.
\end{equation}
On the other hand, according to the definition of $\eta_i,\p_y\eta_i$ and $\zeta_i$ with $i=1,2$, using \textit{Hardy inequality}, \textit{Sobolev embedding} and the assumption \eqref{A.1}, for any $\lambda\in\mathbb{R}$, it holds that
\begin{align}
\label{A.60}
    \|\y^\lambda\eta_i\|_{L_y^\infty}
  \leq& \vartheta_0^{-1}\bigg(\sum_{|\alpha|\leq 3}\|\{u+u_b+(\overline u_e^0-u_b)\phi'\}^{1/2}\y^{\lambda-1}D^\alpha(u,h)\|_{ L_y^2}\nonumber\\
  &\quad\quad +\|(\overline u_e^0,\overline h_e^0)\|_{L^{\infty}(0,L)}+C(u_b)\bigg),\\
\label{A.61}
    \|\y^{\lambda}\p_y\eta_i\|_{L_y^\infty}
  \leq& \vartheta_0^{-2}\bigg(\sum_{|\alpha|\leq 4}\|\{u+u_b+(\overline u_e^0-u_b)\phi'\}^{1/2}\y^{\lambda-1}D^\alpha(u,h)\|_{ L_y^2}\nonumber\\
  &\quad\quad +\|(\overline u_e^0,\overline h_e^0)\|_{L^{\infty}(0,L)}+C(u_b)\bigg)^2,\\
\label{A.62}
    \|\y^{\lambda}\zeta_i\|_{L_y^\infty}
  \leq& \vartheta_0^{-3}\bigg(\sum_{|\alpha|\leq 5}\|\{u+u_b+(\bar u_e^0-u_b)\phi'\}^{1/2}\y^{\lambda-1}D^\alpha(u,h)\|_{ L_y^2}\nonumber\\
  &\quad\quad +\sum_{\beta\leq1}\|\p_x^\beta(\overline u_e^0,\overline h_e^0)\|_{L^{\infty}(0,L)}+C(u_b)\bigg)^3.
\end{align}
Then, for the terms $R_1^\beta$ and $R_2^\beta$ defined by \eqref{A.56}, we can deduce the following estimates for $m \geq 5,~~l \geq0$,
\begin{align}
    \|R_1^\beta\|_{L_l^2} \label{A.63}
  &\leq C\vartheta_0^{-3}\bigg(\sum_{\beta\leq m+1}\|\p_x^\beta(\overline u_e^0,\overline h_e^0)\|_{L^{\infty}(0,L)}
        +C(u_b)+\|(u,h)\|_{H_l^m}\bigg)^3 \cdot\|(u,h)\|_{H_l^m}\nonumber\\
  &\quad\quad +\|\p_x^\beta r_1-\eta_1\p_x^\beta r_3\|_{L_l^2}, \\
    \|R_2^\beta\|_{L_l^2} \label{A.631}
  &\leq C\vartheta_0^{-3}\bigg(\sum_{\beta\leq m+1}\|\p_x^\beta(\overline u_e^0,\overline h_e^0)\|_{L^{\infty}(0,L)}
        +C(u_b)+\|(u,h)\|_{H_l^m}\bigg)^3 \cdot\|(u,h)\|_{H_l^m}\nonumber\\
  &\quad\quad +\|\p_x^\beta r_2-\eta_2\p_x^\beta r_3\|_{L_l^2}.
\end{align}

Now, we are prepared to estimate the $L_l^2$-norm of $(u_\beta,h_\beta)$.
\begin{lemma}\label{lemmaA.2}
$(Weighted~~estimates~~ for~~ (u_\beta,h_\beta))$ Under the assumption of Proposition \ref{propA.1}, there holds that for any $x\in[0,L]$,

\begin{align}\label{A.64}
\begin{split}
  & s_\beta(x)+\nu\intx\|\p_y u_\beta\|_{ L_l^2}^2+\kappa\intx\|\p_y h_\beta\|_{ L_l^2}^2 \\
  \leq&~ \intx\|\p_x^\beta r_1-\eta_1\p_x^\beta r_3\|_{ L_l^2}^2+\intx\|\p_x^\beta r_2-\eta_2\p_x^\beta r_3\|_{ L_l^2}^2+s_\beta(0)\\
  &~ +C\vartheta_0^{-2}\intx\bigg(\sum_{\beta\leq m+2}\|\p_x^\beta(\overline u_e^0,\overline h_e^0)\|_{L^2(0,L)}+ E_{u,h}
     +C(u_b)\bigg)^2\cdot s_\beta(x)\\
  &~ +C\vartheta_0^{-4}\intx\bigg(\sum_{\beta\leq m+2}\|\p_x^\beta(\overline u_e^0,\overline h_e^0)\|_{L^2(0,L)}+ E_{u,h}
     +C(u_b)\bigg)^4\cdot E_{u,h}^2,
\end{split}
\end{align}
where
\begin{equation}\label{s_beta}
  s_\beta(x)=\|\{u+u_b+(\overline u_e^0-u_b)\phi'\}^{1/2}\y^l(u_\beta,h_\beta)\|_{ L_y^2}^2.
\end{equation}
\end{lemma}
\begin{proof}
Multiplying $\eqref{A.55}_2$ and $\eqref{A.55}_3$ by $\y^{2l}u_\beta$ and $\y^{2l}h_\beta$, respectively, and integrating them over $y\in[0,+\infty)$, we have
\begin{align}\label{A.65}
\begin{split}
  & \frac{1}{2}\frac{d}{dx} s_\beta(x)+\nu\|\y^l\p_y u_\beta\|_{ L_y^2}^2+\kappa\|\y^l\p_y h_\beta\|_{ L_y^2}^2 \\
  =&2l\inty\y^{2l-2}y\cdot \bigg[(v-\overline u_{ex}^0\phi)\cdot\frac{|u_\beta|^2+|h_\beta|^2}{2}-(g-\overline h_{ex}^0\phi)u_\beta h_\beta\bigg]\\
  &+\frac{d}{dx}\inty\y^{2l} (h+\overline h_e^0\phi')u_\beta h_\beta+(\nu-\kappa)\inty\y^{2l}(\eta_1\p_y h_\beta u_\beta)\\
  &+ \inty\y^{2l}(R_1^\beta u_\beta+R_2^\beta h_\beta)-2l\inty\y^{2l-2}\cdot y(\nu\p_yu_\beta\cdot u_\beta+ \kappa\p_y h_\beta\cdot h_\beta).
\end{split}
\end{align}

First, by virtue of \textit{Hardy inequality} and divergence-free conditions, we get
\begin{align}\label{A.66}
\begin{split}
  & 2l\inty\y^{2l-2}y\cdot \bigg[(v-\overline u_{ex}^0\phi)\cdot\frac{|u_\beta|^2+|h_\beta|^2}{2}-(g-\overline h_{ex}^0\phi)u_\beta h_\beta\bigg]\\
  \leq&~ 2l\bigg(\bigg\|\frac{v-\overline u_{ex}^0\phi}{1+y}\bigg\|_{L^\infty}
        +\bigg\|\frac{g-\overline h_{ex}^0\phi}{1+y}\bigg\|_{L^\infty}\bigg)\|(u_\beta,h_\beta)\|_{L_l^2}^2 \\
  \leq&~ 2l(\|(\overline u_{ex}^0,\overline h_{ex}^0)\|_{L^{\infty}(0,L)}
         +\|u_x\|_{L^\infty}+\|h_x\|_{L^\infty})\|(u_\beta,h_\beta)\|_{L_l^2}^2\\
  \leq&~ C(\|(\overline u_{ex}^0,\overline h_{ex}^0)\|_{L^{\infty}(0,L)}+E_{u,h})s_\beta(x).
\end{split}
\end{align}
Second, by integrating by parts and the boundary condition $u_\beta|_{y=0}=0$, it gives
\begin{align}\label{A.67}
\begin{split}
  &(\nu-\kappa)\inty\y^{2l}(\eta_1\p_y h_\beta u_\beta)\\
  =&~~ -\nu\inty h_\beta\p_y(\y^{2l}\eta_1 u_\beta)- \kappa\inty \y^{2l}\eta_1\p_y h_\beta\cdot u_\beta\\
  \leq&~~ \frac{\nu}{4}\|\p_y u_\beta\|_{L_l^2}^2+ \frac{\kappa}{4}\|\p_y h_\beta\|\|_{L_l^2}^2+ C(1+\|\eta_1\|_{L^\infty}^2
          +\|\p_y \eta_1\|_{L^\infty}^2)\|(u_\beta,h_\beta)\|_{L_l^2}^2\\
  \leq&~~ \frac{\nu}{4}\|\p_y u_\beta\|_{L_l^2}^2+\frac{\kappa}{4}\|\p_y h_\beta\|\|_{L_l^2}^2
          +C\vartheta_0^{-2}(\|(\overline u_{ex}^0,\overline h_{ex}^0)\|_{L^{\infty}(0,L)}+ C(u_b)+ E_{u,h})s_\beta.
\end{split}
\end{align}
Third, using the estimates of $R_1^\beta$ and $R_2^\beta$ in \eqref{A.63} and \eqref{A.631}, it yields that
\begin{align}\label{A.68}
\begin{split}
  &\inty\y^{2l}(R_1^\beta u_\beta+R_2^\beta h_\beta)
  \leq \|R_1^\beta\|_{L_l^2} \|u_\beta\|_{L_l^2}+  \|R_2^\beta\|_{L_l^2} \|h_\beta\|_{L_l^2}\\
  \leq&\|\p_x^\beta r_1-\eta_1\p_x^\beta r_3\|_{L_l^2}^2+\|\p_x^\beta r_2-\eta_2\p_x^\beta r_3\|_{ L_l^2}^2\\
      &+ C\vartheta_0^{-2}\bigg(\sum_{\beta\leq m+1}\|\p_x^\beta(\overline u_e^0,
       \overline h_e^0)\|_{L^{\infty}(0,L)}+C(u_b)+ E_{u,h}\bigg)^2\cdot s_\beta\\
  &\quad +C\vartheta_0^{-4}\bigg(\sum_{\beta\leq m+1}\|\p_x^\beta(\overline u_e^0,\overline h_e^0)\|_{L^{\infty}(0,L)}+C(u_b)
       +E_{u,h}\bigg)^4\cdot E_{u,h}^2,
\end{split}
\end{align}
Next, it is direct to get
\begin{align}\label{A.69}
\begin{split}
  &-2l\inty\y^{2l-2}\cdot y(\nu\p_yu_\beta\cdot u_\beta+ \kappa\p_y h_\beta\cdot h_\beta)\\
  \leq&~~\nu\|\p_y u_\beta\|_{L_l^2}^2+ \kappa\|\p_y h_\beta\|_{L_l^2}^2+ Cs_\beta(x).
\end{split}
\end{align}

Substituting the above estimates into \eqref{A.65}, we have
\begin{align}\label{A.70}
\begin{split}
  &\frac{1}{2}\frac{d}{dx} s_\beta(x)+\nu\|\y^l\p_y u_\beta\|_{ L_y^2}^2+\kappa\|\y^l\p_y h_\beta\|_{ L_y^2}^2 \\
  \leq&~\|\p_x^\beta r_1-\eta_1\p_x^\beta r_3\|_{ L_l^2}^2+\|\p_x^\beta r_2-\eta_2\p_x^\beta r_3\|_{ L_l^2}^2
        +\frac{d}{dx}\inty\y^{2l}(h+\overline h_e^0\phi)u_\beta h_\beta\\
  &~+C\vartheta_0^{-2}\bigg(\sum_{\beta\leq m+2}\|\p_x^\beta(\overline u_e^0,\overline h_e^0)\|_{L^2(0,L)}
    +E_{u,h}+C(u_b)\bigg)^2\cdot s_\beta(x)\\
  &~+C\vartheta_0^{-4}\bigg(\sum_{\beta\leq m+2}\|\p_x^\beta(\overline u_e^0,\overline h_e^0)\|_{L^2(0,L)}
    +E_{u,h}+C(u_b)\bigg)^4\cdot E_{u,h}^2.
\end{split}
\end{align}
Moreover, the \textit{a priori assumption} \eqref{priori_assume} yields that
\begin{align}\label{A.71}
  \inty\y^{2l}(h+\overline h_e^0\phi)u_\beta h_\beta
  \leq \frac{1}{2}\|\{u+u_b+(\overline u_e^0-u_b)\phi'\}^{1/2}\y^l (u_\beta,h_\beta)\|_{ L_y^2}^2
  \leq \frac{1}{2}s_\beta,
\end{align}

Therefore, we completes the proof of this lemma by integrating \eqref{A.70} in $x$-direction and using \eqref{A.71}.
\end{proof}

Finally, we state the following equivalence norm between $\|(\p_x^\beta u,\p_x^\beta h)\|_{ L_y^2}^2$ and $\|(u_\beta,h_\beta)\|_{ L_y^2}^2$.
\begin{lemma}\label{lemmaA.3}
(Equivalence norm between $\|\p_x^\beta(u,h)\|_{ L_y^2}^2$ and $\|(u_\beta,h_\beta)\|_{ L_y^2}^2$) If the assumptions in Proposition \ref{propA.1} hold, then
\begin{align}\label{A.72}
\begin{split}
  &M(x)^{-1}\|\{u+u_b+(\overline u_e^0-u_b)\phi'\}^{1/2}\y^l \p_x^\beta(u,h)\|_{ L_y^2}\\
  &\quad \leq~ \|\{u+u_b+(\overline u_e^0-u_b)\phi'\}^{1/2}\y^l (u_\beta,h_\beta)\|_{ L_y^2}\\
  &\quad \leq~ M(x) \|\{u+u_b+(\overline u_e^0-u_b)\phi'\}^{1/2}\y^l \p_x^\beta(u,h)\|_{ L_y^2},
\end{split}
\end{align}
and
\begin{equation}\label{A.73}
  \|\p_y\p_x^\beta(u,h)\|_{ L_l^2}\leq \|\p_y(u_\beta,h_\beta)\|_{ L_l^2}+ M(x)\|h_\beta\|_{ L_l^2},
\end{equation}
where
\begin{align}\label{A.74}
\begin{split}
  M(x):=2\vartheta_0^{-1}(& \|\y^{l+1}\p_y(u,h)\|_{L^\infty}+\|\y^{l+1}\p_y^2(u,h)\|_{L^\infty}\\
        &+C(u_b)+C\|(\overline u_e^0,\overline h_e^0)\|_{L^\infty(0,L)}).
\end{split}
\end{align}
\end{lemma}

The proof of Lemma \ref{lemmaA.3} is similar to the proof of Lemma 3.4 in \cite{LXYwell}, here we omit the demonstration for simplicity.
\subsection{Completeness of the \textit{a priori estimates}}
In this subsection, we are going to prove Proposition \ref{propA.1}. To this end, we need the \textit{a priori assumption}
\begin{equation*}
  \|\y^{l+1}\p_y(u,h)\|_{L^\infty}\leq \sigma_0,\quad \|\y^{l+1}\p_y^2(u,h)\|_{L^\infty}\leq \vartheta_0^{-1}.
\end{equation*}
Combining this with the definitions of $M(x)$, $\eta_i$ and the estimates in \eqref{A.60}, we get
\begin{align*}
  \|\y^{l+1}\eta_i\|_{L^\infty}\leq 2\vartheta_0^{-2},\quad M(x)\leq 5\vartheta_0^{-2}.
\end{align*}
Hence, by virtue of the estimates in Lemma \ref{lemmaA.3}, it gives
\begin{align}\label{A.75}
\begin{split}
  E_{u,h}^2
  &=\sum_{\substack{|\alpha|\leq m\\ \beta\leq m-1}}s_\alpha(x)
    +\sum_{\beta=m}\|\{u+u_b+(\overline u_e^0-u_b)\phi'\}^{1/2}\y^l \p_x^\beta(u,h)\|_{ L_y^2}^2 \\
  &\leq\sum_{\substack{|\alpha|\leq m\\ \beta\leq m-1}}s_\alpha(x) + 25\vartheta_0^{-4}s_\beta(x),
\end{split}
\end{align}
and
\begin{align}\label{A.76}
  \|\p_y(u,h)\|_{H_l^m}^2
  \leq& \sum_{\substack{|\alpha|\leq m\\ \beta\leq m-1}}\|\p_y D^\alpha(u,h)\|_{L_l^2}^2
        +2\|\p_y(u_\beta,h_\beta)\|_{L_l^2}^2+ 50\vartheta_0^{-4}\|h_\beta\|_{L_l^2}^2,
\end{align}
where $\beta=m$ in the definition of $s_\beta, u_\beta, h_\beta$.

With the estimates \eqref{A.75},\eqref{A.76} in hands, we are prepared to achieve the desired \textit{a priori estimates} of $(u,h)$ for system \eqref{2.11}. By virtue of Lemma \ref{lemmaA.1} and Lemma \ref{lemmaA.2}, together with the above estimates \eqref{A.75},\eqref{A.76}, for any $m\geq 5$, the following estimate holds
\begin{align*}
    &\intx\bigg(\sum_{\substack{|\alpha|\leq m\\ \beta\leq m-1}}\|\p_yD^\alpha (u,h)\|_{ L_l^2}^2
     +25\vartheta_0^{-4}\|\p_y (u_\beta,h_\beta)\|_{ L_l^2}^2\bigg)
     +\sum_{\substack{|\alpha|\leq m\\ \beta\leq m-1}} s_\alpha(x)+ 25\vartheta_0^{-4}s_\beta(x)\\
  &\leq \bigg(\sum_{\substack{|\alpha|\leq m\\ \beta\leq m-1}} s_\alpha(0)+ 25\vartheta_0^{-4} s_\beta(0)\bigg)
        + C\delta_1\intx\|\p_y(u,h)\|_{ H_0^m}^2+ C\delta_1^{-1}\intx E_{u,h}^2(1+E_{u,h}^2)\\
  &\quad+\sum_{\substack{|\alpha|\leq m\\ \beta\leq m-1}}\intx\|\y^{l+k}D^\alpha(r_1,r_2)\|_{ L^2}^2
        +C \intx\sum_{\beta\leq m+2}\|\p_x^\beta(\overline u_e^0,\overline h_e^0,\overline p_e^0)\|_{ L^2(0,L)}^2\\
  &\quad+25\vartheta_0^{-4}\intx\bigg(\|\p_x^\beta r_1-\eta_1\p_x^\beta r_3\|_{ L_l^2}^2+\|\p_x^\beta r_2
        -\eta_2\p_x^\beta r_3\|_{ L_l^2}^2\bigg)\\
  &\quad+C\vartheta_0^{-6}\intx\bigg(\sum_{\beta\leq m+2}\|\p_x^\beta(\overline u_e^0,\overline h_e^0)\|_{ L^2(0,L)}
        + E_{u,h} +C(u_b)\bigg)^2\cdot s_\beta(x)\\
  &\quad+C\vartheta_0^{-8}\intx\bigg(\sum_{\beta\leq m+2}\|\p_x^\beta(\overline u_e^0,\overline h_e^0)\|_{ L^2(0,L)}
        + E_{u,h} +C(u_b)\bigg)^4\cdot E_{u,h}^2\\
  &\leq \bigg(\sum_{\substack{|\alpha|\leq m\\ \beta\leq m-1}} s_\alpha(0)+ 25\vartheta_0^{-4} s_\beta(0)\bigg)
       + \sum_{\substack{|\alpha|\leq m\\ \beta\leq m-1}}\intx\|D^\alpha(r_1,r_2)\|_{ L_l^2}^2+ \intx C(u_b,\vartheta_0)\\
  &\quad +C\vartheta_0^{-4}\intx\bigg(\|\p_x^m(r_1,r_2)\|_{ L_l^2}^2+ 4\vartheta_0^{-4}\|\p_x^m r_3\|_{ L_{-1}^2}^2\bigg)\\
  &\quad +C\vartheta_0^{-8}\intx\bigg(1+\sum_{\beta\leq m+2}\|\p_x^\beta(\overline u_e^0,\overline h_e^0,
        \overline p_e^0)\|_{ L^2(0,L)}^2\bigg)^3\\
  &\quad+C\vartheta_0^{-8}\intx\bigg(\sum_{\substack{|\alpha|\leq m\\ \beta\leq m-1}} s_\alpha(x)
        + 25\vartheta_0^{-4}s_\beta(s)\bigg)^3.
\end{align*}

Define
\begin{equation*}
  F_0:=\sum_{\substack{|\alpha|\leq m\\ \beta\leq m-1}} s_\alpha(0)+ 25\vartheta_0^{-4} s_\beta(0),
\end{equation*}
and
\begin{align*}
  F(x):=
  &\sum_{\substack{|\alpha|\leq m\\ \beta\leq m-1}}\|D^\alpha(r_1,r_2)\|_{ L_l^2}^2+  C(u_b,\vartheta_0)\nonumber\\
  &\quad +C\vartheta_0^{-8}\bigg(1+\sum_{\beta\leq m+2}\|\p_x^\beta(\overline u_e^0,\overline h_e^0,\overline p_e^0)\|_{ L^2}^2\bigg)^3\nonumber\\
  &\quad +C\vartheta_0^{-4}\bigg(\|\p_x^m(r_1,r_2)\|_{ L_l^2}^2+ 4\vartheta_0^{-4}\|\p_x^m r_3\|_{ L_{-1}^2}^2\bigg).\nonumber
\end{align*}
Using \textit{Gronwall inequality}, we find
\begin{align}\label{A.78}
\begin{split}
  &\sum_{\substack{|\alpha|\leq m\\ \beta\leq m-1}} s_\alpha(x)+ 25\vartheta_0^{-4} s_\beta(x)\\
  \leq~ &(F_0+\intx F(s){\rm{d}}s)\{1-2C\vartheta_0^{-8}(F_0+\intx F(s){\rm{d}}s)^2 x\}^{-\frac{1}{2}}.
\end{split}
\end{align}
Together with the estimate in \eqref{A.75}, it follows that
\begin{equation}\label{A.79}
  \sup_{x\in[0,L]} E_{u,h}\leq (F_0+\intx F(s){\rm{d}}s)^{\frac{1}{2}}\{1-2C\vartheta_0^{-8}(F_0+\intx F(s){\rm{d}}s)^2 x\}^{-\frac{1}{4}},
\end{equation}
and hence, we have
\begin{equation}\label{A.80}
  \sup_{x\in[0,L]}\|(u,h)\|_{H_l^m}\leq (F_0+\intx F(s){\rm{d}}s)^{\frac{1}{2}}\{1-2C\vartheta_0^{-8}(F_0+\intx F(s){\rm{d}}s)^2 x\}^{-\frac{1}{4}}.
\end{equation}

Moreover, using \textit{Newton-Lebniz formula}, \textit{Sobolev embedding} and \eqref{A.79}, we know that for $i=1,2$
\begin{align}\label{A.81}
\begin{split}
  \|\y^{l+1}\p_y^i(u,h)\|_{L^\infty}
  \leq &Cx(F_0+\intx F(s){\rm{d}}s)^{\frac{1}{2}}\{1-2C\vartheta_0^{-8}(F_0+\intx F(s){\rm{d}}s)^2 x\}^{-\frac{1}{4}}\\
       &+\|\y^{l+1}\p_y^i(u_0,h_0)\|_{L^\infty}.
\end{split}
\end{align}
In addition,
\begin{align}\label{A.82}
  h(x,y)
  &\geq h_0(y)- Cx(F_0+\intx F(s){\rm{d}}s)^{\frac{1}{2}}\{1-2C\vartheta_0^{-8}(F_0+\intx F(s){\rm{d}}s)^2 x\}^{-\frac{1}{4}}.
\end{align}
And in a similar way, it follows that
\begin{align}\label{A.84}
\begin{split}
  (u-h)(x,y)
  &\geq- Cx(F_0+\intx F(s){\rm{d}}s)^{\frac{1}{2}}\{1-2C\vartheta_0^{-8}(F_0+\intx F(s){\rm{d}}s)^2 x\}^{-\frac{1}{4}}\\
  &\quad\quad +(u_0-h_0)(y).
\end{split}
\end{align}

It remains to give the estimates for the terms $F(s)$ and $F_0$ on the right-hand side of inequalities \eqref{A.80}-\eqref{A.84}, using the definition of $M_0$, $r_3$, and the bounds of $r_1,r_2$ in \eqref{2.14}, we can deduce that
\begin{equation}\label{A.85}
  \intx F(s){\rm{d}}s\leq C(u_b)\vartheta_0^{-8} M_0^6 x,
\end{equation}
and hence, by virtue of the definition of $s_\beta$ in \eqref{s_beta} and $u_{\beta 0},h_{\beta 0}$ in \eqref{A.58}, we have
\begin{equation}\label{A.86}
  F_0\leq C\vartheta_0^{-8}\mathcal{P}(M_0+ C(u_b)+ \|(u_0,h_0)\|_{H_l^m}),
\end{equation}
where $\mathcal{P}$ is a polynomial of $\|(u_0,h_0)\|_{H_l^m}$.

Therefore, plugging \eqref{A.85} and \eqref{A.86} into \eqref{A.80}-\eqref{A.84}, the proof of Proposition \ref{propA.1} is completed.

\smallskip
{\bf Acknowledgment.}

Ding's research is supported by the National Natural Science Foundation of China (No.11371152, No.11571117, No.11871005 and No.11771155) and Guangdong Provincial Natural Science Foundation (No.2017A030313003).

\bigskip
%\newpage
%\newpage

\end{document}